\documentclass[10pt]{amsart}

\usepackage{amssymb}
\usepackage{amsmath}
\usepackage{amsthm}
\usepackage{color}
\usepackage{euscript}
\usepackage[linktoc=page,colorlinks,citecolor = green!50!black,linkcolor= blue,urlcolor = blue]{hyperref}
\usepackage{enumitem}
\usepackage{tikz-cd} 
\usepackage{tikz}
\usepackage{mathrsfs}
\usepackage[all]{xy}
\usepackage{verbatim}
\usepackage{comment}
\usepackage{cleveref}
\usepackage[margin=1.3in]{geometry}

\theoremstyle{plain}
\newtheorem{thm}{Theorem}[section]
\newtheorem{lem}[thm]{Lemma}

\newtheorem{cor}[thm]{Corollary}
\newtheorem{prop}[thm]{Proposition}

\newtheorem{assump}[thm]{Assumption}

\newtheorem{construction}[thm]{Construction}

\theoremstyle{definition}
\newtheorem{defn}[thm]{Definition}
\newtheorem{exmp}[thm]{Example}

\theoremstyle{definition}
\newtheorem{note}[thm]{Note}
\newtheorem{rk}[thm]{Remark}

\newcommand{\bC}{{\mathbb C}}
\newcommand{\bD}{{\mathbb D}}

\newcommand{\bF}{{\mathbb F}}

\newcommand{\bH}{{\mathbb H}}

\newcommand{\bK}{{\mathbb K}}

\newcommand{\bP}{{\mathbb P}}

\newcommand{\bR}{{\mathbb R}}
\newcommand{\bS}{{\mathbb S}}

\newcommand{\bZ}{{\mathbb Z}}

\newcommand{\cA}{{\mathcal A}}

\newcommand{\cE}{{\mathcal E}}
\newcommand{\cF}{{\mathcal F}}
\newcommand{\cG}{{\mathcal G}}
\newcommand{\cH}{{\mathcal H}}

\newcommand{\cJ}{{\mathcal J}}

\newcommand{\cL}{{\mathcal L}}
\newcommand{\cM}{{\mathcal M}}
\newcommand{\cN}{{\mathcal N}}
\newcommand{\cO}{{\mathcal O}}
\newcommand{\cP}{{\mathcal P}}

\newcommand{\cU}{{\mathcal U}}
\newcommand{\cV}{{\mathcal V}}
\newcommand{\cW}{{\mathcal W}}
\newcommand{\cX}{{\mathcal X}}

\newcommand{\scrH}{\EuScript H}

\newcommand{\scrU}{\EuScript U}

\newcommand{\R}{\mathbb{R}}

\newcommand{\fX}{\mathfrak{X}}
\newcommand{\fY}{\mathfrak{Y}}

\newcommand{\sheafhom}{\mathcal{H} \kern -.5pt \mathit{om}}
\newcommand{\derivedsheafhom}{\mathcal{R}\mathcal{H} \kern -.5pt \mathit{om}}
\newcommand{\hocolim}{\operatornamewithlimits{hocolim}\limits}
\newcommand{\colim}{\operatornamewithlimits{colim}\limits}

\newcommand{\red}[1]{\textcolor{red}{#1}}

\newcommand{\ev}{\operatorname{ev}}

\numberwithin{equation}{section}
\numberwithin{figure}{section}
\title[Equivariant Floer homotopy via Morse--Bott theory]{Equivariant Floer homotopy via Morse--Bott theory}
\author{Laurent C\^{o}t\'{e}}
\address{Max Planck Institute for Mathematics, Vivatsgasse 7, 53111 Bonn, Germany}
\email{lcote@math.harvard.edu}
\author{Yusuf Bar{\i}\c{s} Kartal}
\address {School of Mathematics, University of Edinburgh, Edinburgh, UK}
\email {ykartal@ed.ac.uk}

\date{}

\date{}

\begin{document}

\maketitle	

\begin{abstract}
We generalize the Cohen--Jones--Segal construction to the Morse--Bott setting. In other words, we define framings for Morse--Bott analogues of flow categories and associate a stable homotopy type to this data. We use this to recover the stable homotopy type of a closed manifold from Morse--Bott theory, and the stable equivariant homotopy type of a closed manifold with the action of a compact Lie group from Morse theory. 

We use this machinery in Floer theory to construct a genuine circle equivariant model for symplectic cohomology with coefficients in the sphere spectrum. Using the formalism of relative modules, we define equivariant maps to (Thom spectra over) the free loop space of exact, compact Lagrangians. We prove that this map is an equivalence of Borel equivariant spectra when the Lagrangian is the zero section of a cotangent bundle -- an equivariant Viterbo isomorphism theorem over the sphere spectrum.

\end{abstract}

\tableofcontents

\section{Introduction}

\subsection{Background} 
%
%
%

In \cite{cohen1995floer}, Cohen--Jones--Segal explained how to recover the stable homotopy type of a manifold via Morse theory. More precisely, they defined the notion of a \emph{flow category}, that packages the data of critical points of a Morse function, and its gradient trajectories with respect to a metric. Via a construction called \emph{geometric realization}, they produce the stable homotopy type of the underlying manifold from the flow category. Their motivation was to produce the \emph{Floer homotopy type}, i.e. a stable homotopy type that refines the Floer homology, which is also produced from the data of critical points and the gradient trajectories of an action functional. More recent work has been conducted to produce stable homotopy types associated to various Floer theories (e.g.\ \cite{lipshitzsarkarkhovanov}, \cite{largethesis}), or to produce Floer theoretic invariants with coefficients in generalized cohomology (e.g.\ \cite{abouzaidblumberg}). The notion of a flow category has also proven to be useful in ordinary Hamiltonian Floer homology beyond the monotone/exact setting (e.g.\ \cite{baixuarnold}, \cite{rezchikovarnold}). 

Despite lying at the heart of Floer homotopy theory, the notion of a flow category as defined by Cohen--Jones--Segal is insufficient for equivariant constructions. The reason is that the flow categories which occur in Morse--Floer theory are very seldom equivariant. For instance, for a manifold carrying action of a compact Lie group, usually there is no equivariant Morse--Smale pair. Similarly, one expects an $S^1$-action by loop rotation on the Hamiltonian Floer homotopy type, but there is no regular $S^1$-equivariant Floer data even in the exact case. Finally, the ordinary flow categories are insufficient to produce $G$-action on symplectic/quantum cohomology (homotopy type) out of Hamiltonian $G$-manifolds. To tackle this problem, one needs to go beyond the Morse theory and allow Morse--Bott functions. In \cite{zhoumorsebott}, Morse--Bott analogues of flow categories are studied to produce homology level invariants out of Morse--Bott functions.

In this work, we generalize the geometric realization of \cite{cohen1995floer} to associate stable homotopy types to Morse--Bott functions. We use this to recover the stable equivariant homotopy type of a $G$-manifold from its Morse theory. We also utilize this to define a model for the Hamiltonian Floer homotopy type/symplectic cohomology with coefficients in the sphere spectrum $\bS$ that carries a genuine $S^1$-action. As an application, we show that the symplectic cohomology of $T^*Q$ with $\bS$-coefficients has the same (Borel) $S^1$-equivariant stable homotopy type as (a Thom spectrum over) the free loop space $\cL Q$. 

\subsection{A Morse--Bott generalization of the Cohen--Jones--Segal construction}
A \emph{Morse--Bott flow category} is a non-unital directed category whose object set is given by a finite disjoint union of closed manifolds, and whose morphisms are manifolds with corners. The prototypical example arises in Morse--Bott theory: given a Morse--Bott function $f:N\to\bR$ on a closed manifold (and a nice metric), define $\cM_f$ to be the category whose object set is given by the union of critical manifolds of $f$, and morphisms from one critical point to another are given by moduli spaces of broken trajectories of $-grad(f)$. 

Our first goal is a Morse--Bott analogue of Cohen--Jones--Segal's result: to recover the stable homotopy type of $N$ from $\cM_{f}$. For this, one needs an extra piece of information, namely a \emph{framing}. To explain this, let $\cM_f(X,Y)$ denote the moduli of broken trajectories from a critical manifold $X$ to another $Y$. There are evaluation maps $\cM_f(X,Y)\to X,Y$ sending a trajectory to the points it is asymptotic to. We define a (relative) \emph{framing} to be a choice of virtual bundles $V_X,V_Y$ on each critical manifold and identifications 
\begin{equation}\label{eq:introframe}
	V_X\oplus T_X= T_{\cM_f(X,Y)}\oplus \underline\bR\oplus V_Y
\end{equation}
compatible with the flow category structure. This can also be seen as a relative framing of the virtual bundle $T_{\cM_f(X,Y)}+ \underline\bR-T_X$. We call a flow category $\cM$ endowed with this data a \emph{framed flow category}, and we construct a spectrum $|\cM|$ called the geometric realization. Our first theorem is the following Morse--Bott analogue of the Cohen--Jones--Segal theorem
\begingroup
\def\thethm{\ref*{thm:morsebottrecovers}}
\begin{thm}
There is a natural choice of framing on $\cM_f$ such that $|\cM_f|\simeq \Sigma^\infty N_+$. Moreover, for a given virtual bundle $\nu$ on $N$, one can twist the framing so that the geometric realization is given by the Thom spectrum $N^\nu$. 
\end{thm}
\addtocounter{thm}{-1}
\endgroup
The natural choice of framing bundle $V_X$ is given by the normal bundle of $X$ within the descending manifold. One can also use the dual choice, i.e.\ the tangent bundle of the ascending manifold, which by \Cref{thm:morsebottrecovers} gives the Atiyah dual, see \Cref{cor:dualframe}. Other immediate applications are Morse--Bott inequalities in generalized cohomology, see \Cref{corollary:betti-bound}, \Cref{cor:mbineqs}.

In  order to define the geometric realization, we use $X\leftarrow \cM(X,Y)\to Y$ as a ``correspondence'' and produce maps $X^{V_X}\to \Sigma Y^{V_Y}$. To produce this, we use wrong way/umkehr maps from the Thom spectrum $X^{V_X}$ to a Thom spectrum over $\cM(X,Y)$, and use \eqref{eq:introframe} to project to $\Sigma Y^{V_Y}$. Then, the geometric realization is essentially obtained by taking iterated cones of this correspondence maps. We use the language of $\cJ$-modules in order to organize this data. This can be seen analogous to Banyaga's construction of Morse--Bott homology, from the chains on the critical manifolds and correspondence maps. Note that defining correspondence/umkehr maps at the cohomology level do not necessarily require the twists by virtual bundles, as one often appeals to Thom isomorphism theorem in this case. 

To prove \Cref{thm:morsebottrecovers}, we define the notion of framed $P$-relative modules on the flow category $\cM$, where $P$ is a space. We show that, these induce maps from $|\cM|$ to (a Thom spectrum over) $P$. The category $\cM_f$ admits a natural $N$-relative module, allowing us to define a map $|\cM_f|\to N^\nu$. Then, we use a filtration argument to prove this is an equivalence. This proof is different from \cite{cohen1995floer} and possibly similar to \cite[Appendix D]{abouzaidblumberg}.

We show that geometric realizations of $G$-equivariant framed Morse--Bott flow categories carry a genuine $G$-action. The freedom to use Morse--Bott functions allow us to extend \Cref{thm:morsebottrecovers} to the equivariant setting. For instance, if $N$ admits a free $G$-action, where $G$ is a compact Lie group, one can lift Morse--Smale pairs on $N/G$ to obtain a $G$-equivariant Morse--Bott function on $N$ with critical manifolds given by $G$-torsors (and a compatible $G$-equivariant metric on $N$). When the action is not free, one can replace $N$ by $EG\times N$. By using appropriate Morse models for $EG$ we show
\begingroup
\def\thethm{\ref*{thm:equivariantcjs}}
\begin{thm}
There exists a framed $G$-equivariant flow category $\cM^G_{\tilde f, h_s}$ such that $|\cM^G_{\tilde f, h_s}|$ has the same stable equivariant homotopy type as $\Sigma^\infty N_+$.
\end{thm}
\addtocounter{thm}{-1}
\endgroup
More precisely, $|\cM^G_{\tilde f, h_s}|$ is equivalent to $\Sigma^\infty(EG\times N)_+$ as a genuine $G$-spectrum, implying it is equivalent to $\Sigma^\infty N_+$ as a ``homotopy/Borel'' $G$-spectrum. Similarly, we use this to prove Morse--Bott inequalities in generalized equivariant homology, see \Cref{cor:eqmbineq}, \Cref{cor:circleeqmbineq}.
 	
\subsection{$S^1$-equivariant Hamiltonian Floer theory} We use the machinery of Morse--Bott flow categories to construct genuine $S^1$-equivariant models for the Hamiltonian Floer homotopy type and spectral symplectic cohomology. In \cite{largethesis}, Large defines spectra $HF(H,\bS)$ associated to a Liouville manifold $M$ with stably trivial tangent bundle and a non-degenerate Hamiltonian $H:S^1\times M\to \bR$ that is linear at infinity, lifting the Hamiltonian Floer homology. Similarly, he defines a homotopical lift $SH(M,\bS)$ of the symplectic cohomology, as a colimit of $HF(H,\bS)$. 

Just as $SH(M,\bZ)$ carries an $S^1$-action by ``loop rotation'', one expects to have an analogous $S^1$-action on $SH(M,\bS)$. We encounter a standard issue: there is no $S^1$-equivariant Floer data. We go around this problem similarly to the construction outlined above, as well as \cite{seidelequivariantpants,shelukhinzhao,bourgeoisoancea}. Namely, we fix an equivariant Morse--Bott function $\tilde f:S:=ES^1=S^\infty\to \bR$ whose critical manifolds are circles. We choose Floer data parametrized by $S$, that satisfy non-degeneracy over the critical points of $\tilde f$, that are regular and such that the data on $s\in S$ and $z.s\in S$ are related by rotation by $z\in S^1$. This allows us to produce an $S^1$-equivariant flow category $\cM_{\tilde f,H,J}$, whose object set is given by a disjoint union of circles. We define $HF_S(H,\bS)$ to be its geometric realization. We define $SH_S(M,\bS)$ to be the genuine $S^1$-equivariant spectrum obtained by taking the colimit of $HF_S(H,\bS)$ as the slope of $H$ goes to infinity. We show that these coincide with $HF(H,\bS)$ and $SH(M,\bS)$ if one forgets the equivariant structure. We expect (but do not check) that the classical equivariant symplectic cohomology can be recovered as the homology of the homotopy quotient of $SH_S(M, \mathbb{S})$. In particular, $SH_S(M, \mathbb{S})$ is richer than a putative Floer homotopy version of the equivariant symplectic cohomology.

We also use the notion of $P$-relative modules to construct $S^1$-equivariant maps from $SH_S(M,\bS)$ to the (Thom spectrum over) the free loop space of an exact compact Lagrangian $Q$. For simplicity, assume $Q$ is simply connected. Then
\begingroup
\def\thethm{\ref*{thm:equivariantcl},\ref*{thm:viterbo}}
\begin{thm}
There is a morphism of genuine $S^1$-equivariant spectra $ SH_{S}(M,\bS)\to \Sigma^{\infty+n}\cL Q_+$.	Moreover, when $M=T^*Q$, this map is an equivalence of (Borel/homotopy) $S^1$-equivariant spectra.	
\end{thm}
\addtocounter{thm}{-1}
\endgroup
In other words, we prove an equivariant Viterbo isomorphism theorem over the sphere spectrum $\bS$. Note that a non-equivariant version of this claim appears in \cite{cohencotangent}; however, the proof in loc.\ cit.\ uses a comparison of Floer action functional on the loop space of $T^*Q$ with the Morse action functional on $\cL Q$, similar to \cite{abbondandoloschwarz1}. 

To produce the $S^1$-equivariant map $ SH_{S}(M,\bS)\to \Sigma^{\infty+n}\cL Q_+$, we define $S^1$-equivariant framed $\cL Q$-relative modules $\cN_{\tilde f,H,J}$ using moduli of half-cylinders as in \cite{zhaothesis}. To prove the second claim, we use the fact that this induces an isomorphism on homology by \cite{abouzaid2013symplectic}. As both sides of the map are bounded below in this case, we can apply the Hurewicz theorem to its cone. Note that, when $Q$ is not simply connected, $\Sigma^{\infty+n}\cL Q_+$ needs to be replaced by a Thom spectrum $\Sigma^\infty\cL Q^{-\rho^*W_{mas}}$, where $\rho^*W_{mas}$ is a virtual bundle whose rank matches the Maslov index over each component of $\cL Q$.  

\subsection{Other Applications}
This work is foundational for a long term project to understand topological Hochschild/cyclic homology of the Fukaya category as well as its algebraic $K$-theory. This requires understanding cyclotomic structures on $SH(M,\bS)$, of which the $S^1$-action is an important part. Separately, the genuine model we construct facilitates the computation of homotopy quotients/fixed points, and we plan to generalize the results of \cite{zhaoperiodic} and recover $R_*(M)$ from $SH(M,R)$, the circle action and the filtration on it for some ring spectra $R$, such as Morava $K$-theories. This would let us recover the $\bF_p$ Betti numbers of $M$. 

Another application which we hope to pursue is the construction of equivariant and $K$-theoretic Coulomb branches. More precisely, in the presence of an Hamiltonian $G$-action on $M$, we expect a lift of the construction of \cite{gonzalezmakpomerleanocoulombsh}, to a genuine $G$-spectrum $SH_{EG}(M,\bS)$, as well as the existence of $G$-equivariant maps $\Sigma^\infty\Omega G_+\to SH_{EG}(M,\bS)$ of spectra. One can apply equivariant $K$-theory and elliptic cohomology to this to produce the expected maps from the $K$-theoretic Coulomb branches. 

Another possible application is to Floer theory for autonomous Hamiltonians (c.f.\ \cite{bourgeois2009symplectic}). Given non-degenerate $H:M\to\bR$, and a generic (domain dependent) almost complex structure, the Floer action functional is Morse--Bott, and one can use our machinery to develop the corresponding Floer theory. More generally, (up to gluing issues) one can use our tools to give convenient descriptions of the spectral symplectic cohomology, when the Reeb flow is $1$-periodic.

\subsection*{Acknowledgments}
The authors would like to thank Jacob Lurie, Paul Seidel, Sanath Devalapurkar, Shaoyun Bai, Alex Oancea, Noah Porcelli, Ivan Smith, Tim Large and John Pardon for useful discussions and comments. This work was completed while the first author was supported by the National Science Foundation under Grant No. DMS-2305257 and by the Simons Collaboration on Homological Mirror symmetry; and the second author was supported by ERC Starting Grant 850713 (Homological mirror symmetry, Hodge theory, and symplectic topology), and the Simons Collaboration on Homological Mirror Symmetry.


\section{Morse--Bott flow categories and their geometric realizations}\label{sec:mbcatsgeomrealizations}
\subsection{Definition and the geometric realization}
In \cite{cohen1995floer}, Cohen--Jones--Segal explained how to reconstruct the stable homotopy type of a manifold from its Morse theory. To a Morse--Smale function on a Riemannian manifold $(M, g)$, they associate a \emph{flow category} whose objects are the critical points of $f$ and whose morphisms are moduli spaces of gradient trajectories. The Pontryagin--Thom construction is then used to define a ``twisted complex in spaces'' (or \emph{$\cJ$-module} to be precise) from the data of a flow category, whose geometric realization is a spectrum. 

In this section, we generalize the approach of \cite{cohen1995floer} to the Morse--Bott setting. The prototypical example to have in mind comes from Morse--Bott functions (\Cref{exmp:morsebottfunction}), but we will also apply our constructions to Floer theory. We will finally explain the (superficial) modifications needed to generalize all of this to the $G$-equivariant setting, where $G$ is a compact Lie group. 

We refer to  \cite{hajek2014manifolds} for foundational material on the category of manifolds with corners (diffeomorphisms, transversality, etc.). We also refer to \Cref{appendix:k-manifolds} for the notion of a $\langle k\rangle$-manifold.
\subsubsection{Morse-Bott flow categories and framings}
\begin{defn}[cf.\ \cite{zhoumorsebott,largethesis}]\label{defn:flowcategory}
A \emph{Morse--Bott flow category} $\cM$ is given by the following data:
\begin{itemize}
	\item a topological space $Ob(\cM)$ that is given by the disjoint union of smooth closed manifolds
	\item a proper continuous function $\mu:Ob(\cM)\to\bZ$ called the index (we refer to a connected component $X\subset \mu^{-1}(i)$ as an object of $\cM$) 
	\item for any two objects $X,Y$ of $\cM$, a compact $\langle \mu(X)-\mu(Y)-1\rangle$-manifold $\cM(X,Y)$ (in addition, $\cM(X,Y)$ is empty whenever $\mu(X)\leq \mu(Y)$)
	\item evaluation maps $\ev_X:\cM(X,Y)\to X$ and $\ev_Y:\cM(X,Y)\to Y$
	\item for any objects $X,Y,Z$ an associative composition map $\cM(X,Y)\times_Y \cM(Y,Z)\to \cM(X,Z)$
\end{itemize}
We impose the following additional conditions 

\begin{enumerate}
\item\label{item:iflow} The maps $\cM(X,Y),\cM(Y, Z) \to Y$ are transverse to each other (or equivalently $\cM(X,Y)\times \cM(Y,Z)\to Y\times Y$ is transverse to the diagonal). As a result, the fiber product $\cM(X,Y)\times_Y \cM(Y,Z)$ is cut out transversally 
\item\label{item:iiflow} the composition $\cM(X,Y)\times_Y \cM(Y,Z)\to \cM(X,Z)$ is a diffeomorphism onto a boundary face and $\partial_i\cM(X,Z)$ is given by the disjoint union of $\cM(X,Y)\times_Y \cM(Y,Z)$ with $\mu(Y)=\mu(Z)+i$ 	
\end{enumerate}
\end{defn} 
We will usually ignore the prefix Morse--Bott, and refer to these simply as flow categories. 
\begin{rk}
We emphasize two salient features of \Cref{defn:flowcategory}: (i) we allow objects to be manifolds, not just points; (ii) the morphisms spaces are smooth manifolds with corners and the composition maps are smooth embeddings. 
  
For comparison with other definitions in the literature: \cite[Def.\ 2.9]{zhoumorsebott} also applies in the Morse--Bott setting, but imposes weaker conditions on the composition maps (see \cite[Def.\ 2.9, (4)]{zhoumorsebott}). \cite[Def.\ 2.2]{largethesis} does not apply in the Morse--Bott setting, but imposes similarly strong smoothness conditions as ours. 

The reason why we impose smoothness conditions on our morphisms spaces and composition maps is because we wish to extract spectral invariants from our flow categories. This is in contrast to e.g. \cite{zhoumorsebott} which is focused on cohomological invariants. 
\end{rk}
\begin{note}\label{note:fiber-products}
Recall from \cite[Def.\ 16]{hajek2014manifolds} that two smooth maps $P,Q\to N$ from manifolds with corners $P,Q$ to a manifold without boundary are called \emph{transverse} if their restriction to each pair of faces is transverse. As a result, (\ref{item:iflow}) and (\ref{item:iiflow}) together imply that for any sequence $X^0,\dots,X^n$ of objects, the natural map $\mathcal{M}(X^0, X^1) \times \dots \times \mathcal{M}(X^{n-1}, X^n) \to X^1 \times X^1 \times \dots \times X^{n-1} \times X^{n-1}$ is transverse to $\Delta_{X^1} \times \dots \times \Delta_{X^{n-1}}$.

On the other hand, if one assumes the transversality condition (\ref{item:iflow}) only in the interiors, the transversality to $\Delta_{X^1} \times \dots \times \Delta_{X^{n-1}}$ has to be imposed separately. There is one condition that is at least as strong, and is easy to check though: in addition to assuming (\ref{item:iflow}) for the interiors, and that $int(\cM(X,Y))\times_Y int(\cM(Y,Z))\hookrightarrow \cM(X,Z)$ is a smooth embedding into the interior of the boundary, one assumes that this embedding has a normal vector field whose push-forward to $Z$ vanishes. This implies $\cM(X,Y)\times_Y \cM(Y,Z)\times_Z \cM(Z,W)$ has interior cut out transversally. One then further imposes the same normal field assumption for $\cM(X,Y)\times_Y \cM(Y,Z)\times_Z \cM(Z,W)\hookrightarrow \cM(X,Z)\times_Z \cM(Z,W)$, etc.  

 (\ref{item:iflow}) is used to establish \Cref{collary:neat-embedding-category} (via \Cref{lemma:fiber-j-module}) which is in turn used in the proof of \Cref{proposition:j-modules}.
\end{note}

A \emph{metric} on a flow category is the data of a Riemannian metric on each $\cM(X,Y)$, such that the product metric on $\cM(X,Y) \times_Y \cM(Y,Z) \hookrightarrow \cM(X,Y) \times \cM(Y, Z)$ coincides with the pullback metric under the embedding $\cM(X,Y) \times_Y \cM(Y,Z) \hookrightarrow \cM(X,Z)$.

A metric on a flow category induces isomorphisms
\begin{equation}\label{equation:boundary-isom}
    T_{\cM(X,Y) \times_Y \cM(Y,Z) } \oplus \underline\bR \cong T_{\cM(X,Z)}|_{\cM(X,Y) \times_Y \cM(Y,Z)}
\end{equation}
where $\underline{\bR}$ denotes the trivial bundle of rank $1$. 

\Cref{defn:flowcategory}(i) implies that we have an exact sequence of vector bundles over $\mathcal{M}(X, Y) \times_Y \mathcal{M}(Y, Z)$ 
\begin{equation}\label{equation:fiber-prod-tangent}
0 \to T_{ \mathcal{M}(X, Y) \times_Y \mathcal{M}(Y, Z)} \to  T_{\mathcal{M}(X, Y)}  \oplus  T_{\mathcal{M}(Y, Z)} \to \operatorname{ev}^*_Y T_Y \to 0,
\end{equation}
where the map into $\operatorname{ev}^*_Y T_Y$ takes $(v,w) \mapsto d \operatorname{ev}_Y(v) - d \operatorname{ev}_Y(w)$. A metric on the flow category thus also induces a splitting 
\begin{equation}\label{equation:splitting-fiber}
T_{\mathcal{M}(X, Y) \times_Y \mathcal{M}(Y, Z)}  \oplus \operatorname{ev}^*_Y T_Y \simeq  T_{\mathcal{M}(X, Y)}  \oplus  T_{\mathcal{M}(Y, Z)}.
\end{equation} 

\begin{defn}[{cf. \cite[Definition 2.4]{largethesis}}] \label{definition:stable-framing} A \emph{stable framing} on a flow category with metric is a choice of a virtual bundle $V_X\to X$ for each object $X$ along with an isomorphism of virtual bundles
	\begin{equation}\label{equation:framing-equation}
		\ev_X^* V_X\oplus \ev_X^* T_X \simeq T_{\cM(X,Y)}\oplus \underline{\bR}\oplus \ev_Y^*V_Y
	\end{equation}
 over $\cM(X, Y)$ for all pairs of objects $X, Y$.  We require \eqref{equation:framing-equation} to be compatible with the splitting \eqref{equation:splitting-fiber}. We call a flow category endowed with a metric and a stable framing a \emph{stable flow category}. 
\end{defn}
From now on, for the ease of notation, we will drop terms like $ev_X^*$, $ev_Y^*$ when there is no risk of confusion. 

If one unfolds the meaning of the compatibility condition, it is the commutativity of the following diagram: 
\begin{equation}\label{equation:framing-diagram-square}
\xymatrix{V_X\oplus T_X\oplus T_Y \ar[r]\ar[d]& T_{\cM(X,Z)}\oplus \underline\bR\oplus V_Z\oplus T_Y \ar[d] \\ 
T_{\cM(X,Y)}\oplus \underline{\bR}\oplus V_Y\oplus T_Y  \ar[r]& T_{\cM(X,Y)}\oplus \underline{\bR}\oplus T_{\cM(Y,Z)}\oplus  \underline{\bR} \oplus V_Z  }
\end{equation}
The right vertical arrow is an extension of \eqref{equation:splitting-fiber}. 
\begin{rk}
	In the Morse case, the index $\mu$ controls both the boundary structure of the moduli spaces, and the dimension. In our case, it imposes no restrictions on the dimension. However, in most examples, we will have $rank(V_X)=\mu(X)$, forcing $\cM(X,Y)$ to be $\mu(X)+dim(X)-\mu(Y)-1$ dimensional. The ones failing this condition will appear in \Cref{sec:bimodreal}, when we construct ``the cone flow category corresponding to continuation trajectories''.
\end{rk}
\begin{exmp}\label{exmp:morsebottfunction}

Let $M$ be a manifold. A function $f: M \to \bR$ is said to be \emph{Morse--Bott} if the critical points form a finite union of closed submanifolds and the Hessian is non-degenerate in the normal directions on each submanifold. Explicitly, this means that for each critical submanifold $X \subset M$ and any auxiliary Riemannian metric $g$, the Hessian $Hess_g(f)$ defines a symmetric bilinear form of maximal rank on $T_M/T_X$.  We let $T_M/T_X \supset V_X \to M$ be the bundle of negative eigenspaces of $Hess_g(f)$ and call it the \emph{index bundle}. Note that this bundle depends on the choice of metric.   

We say that the pair $(f,g)$ is \emph{Morse--Bott--Smale} if, for all pairs of connected critical manifolds, $X, X'$, the stable manifold $W^s(X')$ intersects the unstable manifold $W^u(p)$ transversally for all $p\in X$ \cite[Def.\ 9]{banyagamorsebott}. 

If $(f,g)$ is Morse--Bott--Smale, we define $\cM_f$ by 
\begin{itemize}
	\item $Ob(\cM_f)$ is the set of critical points of $f$. 
	\item $V_X$ is the index bundle as above and $\mu(X)$ is its rank 
    \item $\cM(X,Y)= \overline{W^u(X) \cap W^s(Y)}/ \bR$. Explicitly, this is the set of (unparametrized, broken) negative gradient trajectories that start from a point of $X$ and end at a point of $Y$
\end{itemize}
The fact that $\cM_f$ forms a flow category is a long-standing folklore theorem, a proof of which can be extracted from \cite[Sec.\ 8.2]{zhoumorsebott}. 
To endow $\cM_f$ with a framing, note that (at interior points of $\cM(X,Y)$) we have $ T_{\cM(X,Y)}\oplus \underline{\bR} = T_{\overline{(W^u(X) \cap W^s(Y)}}$. Observe that $W^u(X)$ is diffeomorphic to the total space of $V_X$, so has tangent bundle $\operatorname{ev}^*V_X + \operatorname{ev}_X^*T_X$. Since the intersection of $W^u(X)$ with $W^s(Y)$ is transverse, the residual direction of descent is the pullback of $\operatorname{ev}_Y^*V_Y$, and we have: $ T_{\cM(X,Y)} \oplus \underline{\bR}  \oplus \operatorname{ev}_Y^*V_Y= \operatorname{ev}_Y^*V_X \oplus \operatorname{ev}_X^*T_X $ as desired. An elaboration of this argument extends the framings to the boundaries of the moduli spaces such that \eqref{equation:framing-diagram-square} holds. We omit the details. Note that, in practice, we will work with Morse--Bott functions pulled back by quotients by free actions. 


\end{exmp}

Note that if $f: \to \mathbb{R}$ is a Morse--Bott function, there need not exist a metric $g$ so that $(f,g)$ is a Morse--Bott--Smale pair. On the other hand, there are many natural examples for which the condition can be verified. See \cite[Rmk.\ 2.4, Ex.\ 2.5]{latschev2000gradient} for examples and counterexamples.

\begin{exmp}\label{example:flow-cat-act}
Let $M$ be a manifold. Given a Morse--Bott function $f$, the critical manifolds are ordered by the value of $f$. For a critical manifold $X \subset M$, let $\mu(X)$ be the position of $X$ in this ordering starting at zero, i.e. $\mu^{act}(X)=0$ if $X$ is the critical manifold of smallest critical value, $\mu(X)=1$ if $X$ has the second smallest critical value, etc. 

Given a generic metric $g$, we define an alternative flow category $\mathcal{M}_f^{act}$ with the same object, morphisms, and index bundles $V_X$ as in \Cref{exmp:morsebottfunction}, and with index function $\mu^{act}$. As explained in \cite[Sec.\ 8]{zhoumorsebott}, these assumptions can always be satisfies for a generic choice of $g$. In particular, it is not necessary to assume that $(f,g)$ is Morse--Bott--Smale as in \Cref{exmp:morsebottfunction}.\footnote{We alert the reader to a potential clash of terminology: in \cite{zhoumorsebott}, a pair $(f,g)$ which is called Morse--Bott--Smale is not assumed to satisfy the Morse--Bott--Smale transversality condition.}
\end{exmp}

\subsubsection{$\cJ$-modules and their geometric realization}
We review the notion of a \emph{$\cJ$-module} and its geometric realization, as introduced in \cite{cohen1995floer}. Good references include \cite{cohen1995floer,cohenrevisited}. 

\begin{defn}
Given $i>j\in\bZ$, let $J(i,j)$ denote the set of sequences $\{\lambda_k\}_{k\in\bZ}\subset\bR_+=[0,\infty)$ such that $\lambda_k=0$ unless $i>k>j$. In particular, $J(i,j)\cong \bR_+^{i-j-1}$ (we let $J(j+1,j)$ be a point by convention). Addition of sequences defines a map $J(i,j)\times J(j,k)\to J(i,k)$, inducing $J(i,j)^+\wedge J(j,k)^+\to J(i,k)^+$ on the one point compactifications. We let $\cJ$ denote the category enriched in spaces with objects $\bZ$ and morphisms $\cJ(i,j)=J(i,j)^+$ (for $i\leq j$ we let it be empty). The composition is as above. 
\end{defn}
\begin{defn}\label{definition:j-module} A \emph{$\cJ$-module (in spaces)} is a collection $X=\{X(i)\}_{i\in\bZ}$ of based spaces and maps $X(i)\wedge \cJ(i,j)\to X(j)$ satisfying strict associativity. One can define a \emph{$\cJ$-module in spectra} similarly: this is a collection $X= \{ X(i)\}_{i \in \bZ}$ of spectra $X(i) \in \operatorname{Sp}^O$ along with maps $X(i) \wedge \cJ(i,j) \to X(j)$ satisfying strict associativity.
\end{defn}

The meaning of $\cJ$-modules can be illustrated in the simplest cases as follows: assume only $X(i)$ and $X(j)$ are non-trivial and valued in spaces. Then, if $i=j+1$, the $\cJ$-module structure is equivalent to a map $X(j+1)\to X(j)$. If $i>j+1\geq 1$, a priori one has a map $X(i)\wedge (\bR_+^{i-j-1})^+\to X(j)$. On the other hand, if $i>l>j$, the map $X(i)\wedge \cJ(i,l)\wedge\cJ(l,j)\to X(j)$ factors through the trivial based space $X(l)\wedge\cJ(i,j)$; hence, it is trivial. The images of $\cJ(i,l)\wedge\cJ(l,j)$ under the composition cover the boundary of $\cJ(i,j)$ as $l$ varies. As a result the map above factors as 
\begin{equation}
	X(i)\wedge ((\bR_+^{i-j-1})^+/\partial \bR_+^{i-j-1})=\Sigma^{i-j-1}X(i)\to X(j)
\end{equation}

Let us now consider a slightly more involved example: we assume we have a $\cJ$-module (in spaces) where only $X(i), X(j), X(k)$ ($i>j>k\geq 0$). In this case, one has maps $\Sigma^{i-j-1}X(i)\to X(j)$, $\Sigma^{j-k-1}X(j)\to X(k)$, and $X(i)\wedge (\bR_+^{i-k-1})^+\to X(k)$. Similarly to previous case, the last map is trivial on the image of $\cJ(i,l)\wedge \cJ(l,k)$, unless $l=j$. Hence, we obtain a map 
\begin{equation}
X(i)\wedge S^{i-k-2}\wedge [0,\infty]=c(\Sigma^{i-k-2}X(i) )\to X(k)
\end{equation}
giving a null homotopy to the map $\Sigma^{i-k-2}X(i) \to X(k)$ obtained by composing $\Sigma^{i-j-1}X(i)\to X(j)$ and $\Sigma^{j-k-1}X(j)\to X(k)$. 

In general, a finite $\cJ$-module can be thought as a twisted complex in based spaces/spectra. One would like to associate to it a based spectrum. When only $X(i)$ and $X(j)$ are non-trivial, we would like this to be (the suspension spectrum of) the cone of the induced map $\Sigma^{i-1}X(i)\to \Sigma^{j} X(j)$. When $X(i)$, $X(j)$ and $X(k)$ are non-trivial, we would like the geometric realization to be an iterated cone, i.e. \begin{equation}
cone(cone(\Sigma^{i-2}X(i)\to \Sigma^{j-1}X(j))\to \Sigma^kX(k))
\end{equation}
This motivates the following general definition:
\begin{defn}[{cf.\ \cite[Def.\ 2.18]{largethesis}}]\label{definition:geometric-real}
Let $\{X(i)\}_{i \in \bZ}$ be a $\cJ$-module in spaces or in spectra. Given $q\in\bZ$, we define a semi-simplicial space/spectrum $\fX_q$ by 
\begin{equation}\label{eq:pseudosimplforq}
[k]\mapsto \bigvee_{i_0>\dots>i_k} X(i_0)\wedge \cJ(i_0,i_1)\wedge \dots\wedge \cJ(i_{k-1},i_k)\wedge \cJ(i_k,-q)	
\end{equation}
Define $|X|_q$ to be the geometric realization of $\fX_q$. 
\end{defn}
This construction only has simplicial boundary maps, given by compositions of $\cJ$/structure maps of $X$ (hence, it is only semi-simplicial), and to define degeneracy maps, one needs to expand the category $\cJ$ by adding units ($\cJ(i,i)=S^0$ or $\cJ(i,i)=\bS$, depending on whether we consider $\cJ$-modules in spaces or spectra), and extend the bar construction accordingly. This is not needed to define the geometric realization. 
%
%

Observe the following: 
\begin{lem}\label{lem:suspendedrealization}
Given a finite $\cJ$-module $X$, $\Sigma |X|_q\simeq |X|_{q+1}$ for $q\gg 0$ (more precisely, when $-q$ is less than all $i$ for which $X(i)$ is non-trivial).
\end{lem}
\begin{proof}
Assume $q$ is large enough so that $X(i)$ is trivial for $i\leq -q$. Observe that \begin{equation}\fX_{q+1}[0]=\fX_{q}[0]\wedge[0,\infty]\text{ and }\fX_{q+1}[k]=(\fX_{q}[k]\wedge[0,\infty])\vee \fX_q[k-1]
\end{equation}
The $\fX_{q}[k]\wedge[0,\infty]$-component corresponds to the terms where $i_k\neq -q$, and $\fX_q[k-1]$ corresponds to the terms where $i_k=-q>-q-1$. Moreover, $\fX_{q}[k]\wedge[0,\infty]$ lies as a sub-complex of $\fX_{q+1}[k]$. All but the last boundary maps of the $\fX_q[k-1]\subset \fX_{q+1}[k]$ agree. The last boundary map on the $\fX_q[k-1]$ component is given by its embedding into $\fX_{q}[k-1]\wedge \{0\}\subset \fX_{q+1}[k-1]$. Thus $\fX_{q+1}$ can be seen as the simplicial cone of the inclusion map $\fX_q\hookrightarrow \fX_{q}\wedge[0,\infty]$, into $\fX_{q}\wedge\{0\}$, and its realization is homotopy equivalent to the quotient of $|\fX_{q}\wedge[0,\infty]|=|X|_{q}\wedge[0,\infty]$ by $|\fX_{q}\wedge\{0\}|=|X|_{q}\wedge\{0\}$. 

It is possible to check the last statement by hand as follows: the geometric realization of $\fX_{q+1}$ is defined as 
\begin{equation}
	\bigvee \fX_{q+1}[k]\wedge \Delta^k_+/``(\partial_j\sigma\wedge x)\sim(\sigma,\iota_j(x) )"
\end{equation}
where $\iota_j$ denotes the $j^{th}$-face inclusion (more precisely, one writes this as a co-equalizer). The $\fX_{q}[k]\wedge[0,\infty]$-components will produce the space/spectrum $|\fX_{q}|\wedge[0,\infty]=|X|_{q}\wedge[0,\infty]$. When the last boundary map of $\fX_q[k-1]$-component is ignored, the identifications will produce the cone of $|\fX_q|$ (as one takes a quotient of $\bigvee \fX_{q+1}[k-1]\wedge \Delta^k_+$, we see every $\Delta^k_+$ as the cone $c(|\fX_q|)$ of $\Delta^{k-1}_+$). Finally, the last boundary map identifies $|\fX_q|\subset c(|\fX_q|)$ with $|\fX_{q}\wedge\{0\}|\subset |\fX_{q}\wedge[0,\infty]|$. 
\end{proof}
%
%

%
The statement actually holds for any $\cJ$-module that is bounded below, although in practice we only need finite ones. Define
\begin{defn}
For a finite $\cJ$-module $X$ and $q \gg 0$ sufficiently large, set $|X|:=\Sigma^{-q+1}|X|_q$. We call this spectrum \emph{the geometric realization of $X$}. The geometric realization is well-defined by \Cref{lem:suspendedrealization}.
\end{defn}
%
\begin{rk}
For a $\cJ$-module $X$ in spaces, $|X|_q$ is well-defined as a space, and the geometric realization is a de-suspended suspension spectrum of a space. 
\end{rk}
In the sequel we will mostly work with $J$-modules valued in spectra (and $G$-equivariant spectra; cf.\ \Cref{subsection:extension-equivariant}).

The following is an important sanity check:
\begin{lem}\label{lem:singledegree}
If a $\cJ$-module $X$ is supported in degree $i_0$, its geometric realization is given by $\Sigma^{i_0}X(i_0)$.	
\end{lem}
\begin{proof}
Let $-q=i_0-1$. In this case, $\fX_q$ has only one term: $X(i_0)\wedge \cJ(i_0,-q)=X(i_0)\wedge \cJ(i_0,i_0-1)=X(i_0)$. Therefore, $|X|_q=X(i_0)$, and $|X|=\Sigma^{-q+1}|X|_q=\Sigma^{-q+1}X(i_0)=\Sigma^{i_0}X(i_0)$.
%
\end{proof}
One can also show that, in the presence of two non-zero $X(i)$, the geometric realization produces a mapping cone. We will not check this explicitly as it will follow from the machinery of inclusions of $\cJ$-modules and their quotients (see \Cref{corollary:mapping-cones-J}).  
\begin{rk}
Presumably $\cJ$-modules form another model for the stable $(\infty,1)$-category of spectra, at least after some localization. However, we will not work out the homotopy theory of this new category, and work with the geometric realizations.	
\end{rk}

Next we discuss the inclusions and quotients of $\cJ$-modules. This will be relevant later. 
\begin{defn}
Let $X$ and $Y$	be two $\cJ$-modules. A \emph{(strict) map} of $\cJ$-modules $X\to Y$ is a collection of based maps $X(i)\to Y(i)$ that commute with the $\cJ$-module structure maps. We call it an \emph{inclusion} if each $X(i)\to Y(i)$ is cofibrant.
\end{defn} 
The inclusions that we will encounter will be of the form $Y(i)=X(i)\vee Z(i)$, where the structure maps preserve the $X(i)$-component, but not the $Z(i)$-component. 

The following is immediate:
\begin{lem}\label{lem:jmapinducegeomreal}
Maps of $\cJ$-modules induce maps of geometric realizations. Cofibrant maps of finite $\cJ$-modules induce cofibrant maps. 
\end{lem}
\begin{defn}
	Let $X\to Y$ be an inclusion of $\cJ$-modules, i.e. a collection of based, cofibrant inclusions $X(i)\to Y(i)$ that commute with the $\cJ$-module structure maps. Define \emph{the quotient $\cJ$-module $Y/X$} by $(Y/X)(i)=Y(i)/X(i)$ and such that $(Y/X)(i)\wedge \cJ(i,j)\to (Y/X)(j)$ is induced by the structure map $Y(i)\wedge \cJ(i,j)\to Y(j)$.  
\end{defn} 
It is easy to check that the structure maps of $Y/X$ are well-defined and $Y/X$ is a $\cJ$-module. If $Y(i)$ is of the form $X(i)\vee Z(i)$, then $(Y/X)(i)=Z(i)$.  
\begin{lem}\label{lemma:quotient-inclusion}
The sequence of $\cJ$-modules $X\to Y\to Y/X$ induces a quotient sequence $|X|\to |Y|\to |Y/X|$. In particular, $|Y/X|\simeq |Y|/|X|$, and there is a connecting map $|Y/X|\to \Sigma|X|$. 
\end{lem}
\begin{proof}
Let $q$ be large enough, and let $\fX_q$, $\mathfrak{Y}_q$ and $\mathfrak{Z}_q$ denote the semi-simplicial spectra (see \Cref{definition:geometric-real}) corresponding to $X$, $Y$ and $Y/X$, respectively. We have a quotient sequence 
\begin{equation}
	\fX_q[i]\to \mathfrak{Y}_q[i]\to \mathfrak{Z}_q[i]
\end{equation}	
which induces a quotient sequence $|X|_q\to |Y|_q\to |Y/X|_q$. Applying $\Sigma^{-q+1}$ finishes the proof. 
\end{proof}
\begin{cor}\label{corollary:mapping-cones-J}
If $X$ is a $\cJ$-module supported at degrees $i_0>j_0$ then $|X|$ is equivalent to a homotopy fiber $\Sigma^{i_0}X(i_0)\to \Sigma^{j_0+1}X(j_0)$ and to a mapping cone $\Sigma^{i_0-1}X(i_0)\to \Sigma^{j_0}X(j_0)$.
\end{cor}
\begin{proof}
Consider the truncations $X'=\tau_{\leq j_0}X$, $X''=\tau_{> j_0}X$ which are the $\cJ$-modules supported at degree $j_0$, resp. $i_0$ with $X'(j_0)=X(j_0)$, resp. $X''(i_0)=X(i_0)$. There is a quotient sequence of $\cJ$-modules given by $X'\to X\to X''$, inducing a quotient sequence $\Sigma^{j_0}X(j_0)\to |X|\to \Sigma^{i_0}X(i_0)$. This finishes the proof. 
\end{proof}

\subsubsection{Relative Pontryagin--Thom construction and $\cJ$-modules from flow categories}\label{subsubsec:relativethompontr}
Cohen--Jones--Segal \cite{cohen1995floer} associate a $\cJ$-module to a flow category of Morse type (meaning that the critical manifolds are points).  We now explain how to generalize their construction to the Morse--Bott setting. 


To associate a $\cJ$-module to a flow category, \cite{cohen1995floer,cohenrevisited} uses the Pontryagin--Thom construction. For instance, if the (ordinary) flow category $\cM$ has only two objects $x,y$ and a stable framing, one chooses an embedding of $\cM(x,y)\to \bR^N$, $N\gg 0$ with normal bundle $\underline{\bR}^{N}-\underline{V(x)}+\underline{V(y)}$. The geometric realization of the flow category is the cone of the collapse map $S^{N}\to \Sigma S^{N-\mu(x)+\mu(y)}$ (or rather its $\Sigma^{\infty-N+\mu(x)}$). When there are more objects, the collapse maps are organized into a $\cJ$-module using the neat embeddings.

We use a relative version of the collapse map to produce a $\cJ$-module. To motivate the construction, consider a Morse--Bott function $f:M\to \bR$ with two critical manifolds $X,Y$ such that $\mu(X)>\mu(Y)$. Let $T^d_X$ denote the descending tangent bundle of $X$. Then, one has a Puppe sequence $Y\hookrightarrow M\twoheadrightarrow X^{T^d_X}\to \Sigma Y$, where $X^{T^d_X}$ is the Thom space. In particular, one can obtain $M$ as the homotopy fiber of $X^{T^d_X}\to \Sigma Y$. Similarly assume $X$ and $Y$ are two critical sets with $\mu(X)>\mu(Y)$, $f(X)>f(Y)$, and such that no other critical set has index in between. Let $A_0$ denote the closure of the descending manifold of $X$, $A_1$ that of $Y$, and $A_2$ that of critical sets such that $\cM(Y,Z)\neq \emptyset$. Then, there are Puppe sequences
\begin{equation}
A_1\hookrightarrow A_0\twoheadrightarrow X^{T^d_X}\to \Sigma A_1 \text{ and }  	A_2\hookrightarrow A_1\twoheadrightarrow Y^{T^d_Y}\to \Sigma A_2
\end{equation} 
As a result $A_1$ is the homotopy fiber of $Y^{T^d_Y}\to \Sigma A_2$, and writing the map $X^{T^d_X}\to \Sigma A_1$ is equivalent to writing a map $X^{T^d_X}\to \Sigma Y^{T^d_Y}$ and a null-homotopy for the composition $X^{T^d_X}\to \Sigma Y^{T^d_Y}\to\Sigma^2 A_2$. 

If one similarly considers three adjacent critical manifolds $X,Y,Z$, one obtains maps $X^{T^d_X}\to \Sigma Y^{T^d_Y}$, $Y^{T^d_Y}\to \Sigma Z^{T^d_Z}$, and a null-homotopy $c(X^{T^d_X})\to \Sigma^2 Z^{T^d_Z}$ for the composition. The data is essentially the same as that of a $\cJ$-module with three non-trivial $X(i)$. 

Going back to two critical sets $X,Y$, one has natural embeddings $\cM_f(X,Y)\hookrightarrow T^d_X$ that is compatible with the evaluation/projection onto X. One can see the framing identity \begin{equation}
	T^d_X\oplus T_X= T_{\cM_f(X,Y)}\oplus T^d_Y\oplus \underline{\bR}
\end{equation} easily by considering the normal bundle of the embedding. It is not hard to see that the map $X^{T^d_X}\to Y^{T^d_Y+\bR}$ can be obtained by collapsing $X^{T^d_X}$. 

More generally, consider a framed flow category $\cM$ with two objects $X,Y$ of different index (say $\mu(X)>\mu(Y)$). Choose an embedding $\cM(X,Y)\to V_X$ compatible with the evaluation/projection onto $X$ (such an embedding can always be obtained by stabilizing the composition $\cM(X,Y)\to X\to V_X$). The normal bundle of such an embedding is $\ev_X^*V_X+\ev_X^*T_X-T_{\cM(X,Y)}=\ev_Y^*V_Y+\underline{\bR}$. By composing the collapse map
\begin{equation}
	X^{V_X}\to \cM(X,Y)^{\ev_X^*V_X+\ev_X^*T_X-T_{\cM(X,Y)}}=\cM(X,Y)^{\ev_Y^*V_Y+\underline{\bR}}
\end{equation}
with the natural map $Y^{V_Y+\underline{\bR}}$, we obtain a map $X^{V_X}\to \Sigma Y^{V_Y}$.
\begin{rk}
Note that this is similar with the construction of Morse--Bott homology as in \cite{banyagamorsebott}. Namely, the moduli spaces $\cM(X,Y)$ together with the evaluation maps to $X$ and $Y$ are used as correspondences. One can also define such a map $C_*(X)\to C_*(Y)$ by composing the natural map $C_*(\cM(X,Y))\to C_*(Y)$ with an umkehr map $C_*(X)\to C_*(\cM(X,Y))$. In general, one expects the space level version of an umkehr map to be $X^\xi\to \cM(X,Y)^{\xi+T_X-T_{\cM(X,Y)}}$, see \cite{cohenkleinumkehr} for instance. From another perspective, a space level umkehr map $X^{\xi-T_X}\to \cM(X,Y)^{\xi-T_{\cM(X,Y)}}$ can be produced using the functoriality of the Spanier--Whitehead duals. By letting $\xi=T_X+V_X$ (in the latter version), we obtain a map similar to before. 
\end{rk}

The rest follows as in \cite{cohen1995floer,cohenrevisited,lipshitzsarkarkhovanov}, namely we choose neat embeddings of the moduli spaces $\cM(X,Y)$, and apply the collapse map to obtain a $\cJ$-module. 

Before proceeding to make this precise, we record following easy facts about Thom spectra and Pontryagin--Thom collapse:
\begin{enumerate}
	\item Given embedding of manifolds $f:P\hookrightarrow Q$, collapse is a map $Q^+\to P^{\nu_{P/Q}}$, obtained by choosing a closed tubular neighborhood $P\subset N\subset Q$, and sending everything outside $N$ into $\partial N/\partial N\subset N/\partial N=P^{\nu_{P/Q}}$.
	\item More generally, if $E$ is a vector bundle on $Q$, there is a twisted collapse map $Q^E\to P^{f^*E+\nu_{P/Q}}$. One can see this as the collapse map for the embedding of $P$ into the total space of $E$.
	\item Given a map $g:Q_1\to Q_2$, and bundle $E$ over $Q_2$ there is natural map of Thom spaces/spectra $Q_1^{g^*E}\to Q_2^{E}$.
	\item These are compatible in the following sense: given commutative square as in the left side of the diagram below, where the normal bundles of $P_i$ are related by $h$ (i.e. $\nu_{P_1/Q_1}\cong h^*\nu_{P_2/Q_2}$), and given bundles $E_i\to Q_i$ such that $E_1\cong g^*E_2$, one obtains a square as in the right side of
	\begin{equation*}
		\xymatrix{P_1\ar@{^{(}->}[r]^{f_1}\ar[d]^{h}& Q_1\ar[d]^{g}\\
			P_2\ar@{^{(}->}[r]^{f_2}& Q_2  } 
		\Longrightarrow
		\xymatrix{Q_1^{E_1}\ar[r]\ar[d]& P_1^{f_1^*E_1+\nu_{P_1/Q_1}}\ar@<-25pt>[d]\\
			Q_2^{E_2}\ar[r]& P_2^{f_2^*E_2+\nu_{P_2/Q_2}}  }
	\end{equation*} 
   \item These facts extend naturally to manifolds with boundary, corners, $\langle k\rangle$-manifolds etc., as long as the embeddings respect the boundary decomposition.
\end{enumerate}
Notice that the twisted collapse map is very close to what was described above, and we will use this in the construction of the  $\cJ$-module. 

Now assume $\cM$ is finite and let $J_q(i,j)$, resp. $\cJ_q(i,j)$ denote $J(i,j)\times\bR^q$, resp. $J_q(i,j)^+=\cJ(i,j)\wedge S^q$. The spaces $J_q(i,j)\cong \bR_+^{i-j-1}\times \bR^q$ are prototypical examples of $\langle i-j-1\rangle$-manifolds. As before, there are natural products $J_{q_1}(i,j)\times J_{q_2}(j,k)\to J_{q_1+q_2}(i,j)$ and $\cJ_{q_1}(i,j)\wedge \cJ_{q_2}(j,k)\to \cJ_{q_1+q_2}(j,k)$. 

Let $X_i$ denote the disjoint union of index $i$ objects of $\cM$ (i.e. $\mu^{-1}(i)$), and let $\cM(i,j):=\cM(X_i,X_j)$. We will also abbreviate the fiber product $\times_{X_i}$ as $\times_i$.  Recall the notion of \emph{neat embeddings of $\langle k\rangle$-manifolds} from \cite{laures}, and \Cref{appendix:k-manifolds}: these are roughly the embeddings of a $\langle k\rangle$-manifold into $\bR_+^k\times \bR^q$, where the $i^{th}$-boundary of the given $\langle k\rangle$-manifold maps into the $i^{th}$-boundary of $\bR_+^k\times \bR^q$. By \Cref{collary:neat-embedding-category}, we may choose a \emph{neat embedding} of the flow category $\cM$. This is a collection of neat embeddings $\cM(i,j)\hookrightarrow J_{q_i-q_j}(i,j)$ such that the following diagram commutes
\begin{equation}\label{eq:neatflowemb}
	\xymatrix{\cM(i,j)\times_j\cM(j,k) \ar[r]\ar@{^{(}->}[d]& \cM(i,k)\ar@{^{(}->}[d] \\ J_{q_i-q_j}(i,j) \times J_{q_j-q_k}(j,k)\ar[r]& J_{q_i-q_k}(i,k) }
\end{equation}
for all $i,j,k$. 
We will often suppress the subscripts $q_i-q_j$ from the notation and write $J_q(i,j)$ instead. 

Let $V_i:=V_{X_i}$ and consider the lift $\cM(i,j)\to V_i$ given by the composition $\cM(i,j)\xrightarrow{ev_{X_i}} X_i \xrightarrow{0}V_i$. Note that $V_i$ is a virtual bundle, and can have negative virtual dimension. In this case, one can without loss of generality stabilize all $V_i$ simultaneously by a constant factor, so that they become actual vector bundles. Also note that these lifts are compatible with the gluing maps $\cM(i,j)\times_j\cM(j,k)\to\cM(i,k)$, i.e. the compositions $\cM(i,j)\times_j\cM(j,k)\to\cM(i,j)\to V_i$ and $\cM(i,j)\times_j\cM(j,k)\to\cM(i,k)\to V_i$ agree. One can use any choice of lifts satisfying this compatibility, but we stick with the choice above. 

As a result, we obtain embeddings $\cM(i,j)\to V_i\times J_q(i,j)$ such that
\begin{equation}\label{eq:relativeneatflowemb}
	\xymatrix{\cM(i,j)\times_j\cM(j,k) \ar[r]\ar@{^{(}->}[d]& \cM(i,k)\ar@{^{(}->}[d] \\ V_i\times J_q(i,j) \times J_q(j,k)\ar[r]& V_i\times J_q(i,k) }
\end{equation} 
commutes.
\begin{construction}\label{constr:jmodule}
Applying Pontryagin--Thom collapse as above gives rise to a map \begin{equation}
	X_i^{V_i}\wedge \cJ_q(i,j) =  X_i^{V_i}\wedge \cJ(i,j)\wedge S^q\to \cM(i,j)^{V_i+T_{X_i}+\underline{\bR}^{i-j-1+q}-T_{\cM(i,j)}}=\atop\cM(i,j)^{V_j+\underline{\bR}^{i-j+q}}\to \Sigma^{i-j+q}X_j^{V_j}
\end{equation}
By desuspension, we obtain maps $\Sigma^{-i}X_i^{V_i}\wedge \cJ(i,j)\to \Sigma^{-j}X_j^{V_j}$
\end{construction}
\begin{rk}
Desuspension is a spectra level operation. To actually construct a $\cJ$-module in spaces, one needs to apply $\Sigma^q$ for $q\gg 0$.
\end{rk}
Next,
\begin{prop}\label{proposition:j-modules}
The construction above gives a $\cJ$-module with $i^{th}$ spectrum given by $\Sigma^{-i}X_i^{V_i}$.	
\end{prop}

\begin{proof}
The proposition amounts to the statement that the following diagram commutes.

\begin{equation}\label{eq:jmodulediagramflow}
    \begin{tikzcd}
 \Sigma^{-i} X_i^{V_i} \wedge \cJ(i,j) \wedge \cJ(j,k)  \ar[r, hook] \ar[d] & \Sigma^{-i}X_i^{V_i} \wedge \cJ(i,k)  \ar[d]\\
\Sigma^{-j} X_j^{V_j} \wedge \cJ(j,k) \ar[r] & \Sigma^{-k}X_k^{V_k}
    \end{tikzcd}
\end{equation}

Applying the Pontryagin--Thom construction to \eqref{eq:relativeneatflowemb} (and desuspending), one obtains the commutative diagram
\begin{equation}
		\xymatrix{\Sigma^{-i}X_i^{V_i}\wedge \cJ(i,j) \wedge \cJ(j,k)\ar[r]\ar[d]& \Sigma^{-i}X_i^{V_i}\wedge \cJ(i,k)\ar[d]\\ (\cM(i,j)\times_j\cM(j,k))^\nu \ar[r]& \cM(i,k)^\nu}
\end{equation}
where $\nu$ denotes the Thom spectrum of the normal bundle (or the virtual bundle obtained by subtracting $\bR^{i+q}$ during the desuspension). The composition of the right vertical arrow with $\cM(i,k)^\nu\to \Sigma^{-k}X_k^{V_k}$ is the $\cJ$-module multiplication $\Sigma^{-i}X_i^{V_i}\wedge \cJ(i,k)\to \Sigma^{-k}X_k^{V_k}$. Hence, the composition in \eqref{eq:jmodulediagramflow} through the upper right corner can also be obtained as \begin{equation}\label{eq:longcomposethom}
	\Sigma^{-i}X_i^{V_i}\wedge \cJ(i,j) \wedge \cJ(j,k)\to (\cM(i,j)\times_j\cM(j,k))^\nu \to \cM(i,k)^\nu\to \Sigma^{-k}X_k^{V_k}
\end{equation}
As recalled above the Thom construction is natural with respect to pull-backs. The last two maps in \eqref{eq:longcomposethom} are of this form; hence, the composition of the last two can be directly induced by the evaluation map $\cM(i,j)\times_j\cM(j,k)\to X_k$, where the framing condition $V_i\oplus T_{X_i}=T_{\cM(i,k)}\oplus \underline{\bR}\oplus V_k$ is also used. 

The compositions $\cM(i,j)\times_j\cM(j,k)\to \cM(i,k)\to X_k$ and $\cM(i,j)\times_j\cM(j,k)\to \cM(j,k)\to X_k$ are the same, and one can obtain the composition of the last two arrows in \eqref{eq:longcomposethom} also as a composition
\begin{equation}
	(\cM(i,j)\times_j\cM(j,k))^\nu\to \cM(j,k)^\nu\to \Sigma^{-k}X_k^{V_k}
\end{equation}
These maps can also be seen maps of the pull-back type (i.e. of the form $Q_1^{g^*E}\to Q_2^E$). On the other hand, we can use the framing condition $V_j\oplus T_{X_j}=T_{\cM(j,k)}\oplus \underline{\bR}\oplus V_k$ on $\cM(j,k)$ to identify the bundle on it with $V_j+T_{X_j}-T_{\cM(j,k)}- \underline{\bR}$ (up to a trivial bundle). Once again, on $\cM(i,j)\times_j\cM(j,k)$, we can use the condition $V_i\oplus T_{X_i}=T_{\cM(i,j)}\oplus \underline{\bR}\oplus V_j$ to identify the pull-back bundle with the normal bundle of $\cM(i,j)\times_j\cM(j,k)\hookrightarrow V_i\times J_q(i,j) \times J_q(j,k)$. Notice, we have two identifications with this bundle, the identifications match by the compatibility condition \eqref{equation:framing-diagram-square}.

Consider the diagram 
\begin{equation}\label{eq:hugejmoduleflowdiagram}
	\xymatrix{ \Sigma^{-i}X_i^{V_i}\wedge \cJ(i,j) \wedge \cJ(j,k) \ar[r]\ar[d]& (\cM(i,j)\times_j\cM(j,k))^\nu\ar@/^2.5pc/[lddd] \\
	 \cM(i,j)^\nu \wedge \cJ(j,k)\ar[d]\ar[ru]& \\
 \Sigma^{-j}X_j^{V_j} \wedge \cJ(j,k)\ar[d]& \\
  \cM(j,k)^\nu\ar[d]& \\
\Sigma^{-k}X_k^{V_k}}
\end{equation}
The vertical maps are the usual ones appearing in the $\cJ$-module structure, i.e. the first one is a Pontryagin--Thom collapse on the first two components, the second and the fourth are induced by pull-backs, and the third is another collapse. 

The horizontal arrow and the upper diagonal arrow are also collapse maps, and by the naturality of the collapse, the upper triangle commutes. Note that the upper diagonal arrow is a twisted collapse map (i.e. of the form $Q^E\to P^{f^*E+\nu_{P/Q}}$), where $E$ is given by the normal bundle of the embedding of $\cM(i,j)\times J_q(j,k)$ into $V_i\times J_q(i,j)\times J_q(j,k)$. 

The lower diagonal arrow is induced by a pull-back, and as observed above the composition of the horizontal arrow, the lower diagonal arrow and the last vertical arrow is the same as the composition in \eqref{eq:jmodulediagramflow} through the upper right corner. The composition through the lower right corner is given by the composition of all vertical arrows in \eqref{eq:hugejmoduleflowdiagram}, and to establish commutativity of \eqref{eq:jmodulediagramflow}, we only need to check the commutativity of the mid-square in \eqref{eq:hugejmoduleflowdiagram}: 
\begin{equation}\label{eq:midsqr}
	\xymatrix{\cM(i,j)^\nu \wedge \cJ(j,k) \ar[r]\ar[d]& (\cM(i,j)\times_j\cM(j,k))^\nu\ar[d] \\
	\Sigma^{-j}X_j^{V_j} \wedge \cJ(j,k) \ar[r]&\cM(j,k)^\nu }
\end{equation}
Consider the following commutative diagram
\begin{equation}\label{eq:embdiagram1}
\xymatrix{ \cM(i,j) \times J_q(j,k) \ar[d]& \cM(i,j)\times_j\cM(j,k)\ar[d]\ar@{_{(}->}[l] \\
	X_j \times J_q(j,k) &\cM(j,k)\ar@{_{(}->}[l]
}
\end{equation}
As observed above, the upper horizontal arrow of \eqref{eq:midsqr} comes as a twisted Pontryagin--Thom collapse map applied to the upper horizontal map of \eqref{eq:embdiagram1} (with respect to the normal bundle of the embedding of $\cM(i,j) \times J_q(j,k)$ into $V_i\times J_q(i,j)\times J_q(j,k)$). Similarly, the lower horizontal arrow can be seen of this form, with respect to a stabilization of $V_j$. Moreover, the bundles on the left side, resp. right side, are related by pull-back along the vertical arrows (this forces the normal bundles of the upper and lower horizontal embeddings to agree, which can be checked more directly). By the naturality of the collapse maps observed above, this implies the commutativity of \eqref{eq:midsqr}.

\end{proof}
The construction involved choices of neat embeddings. Any two choices of neat embeddings are isotopic after stabilizing with a high enough $\bR^q$; therefore, the construction does not depend on them. Similarly, if one would like to use different lifts $\cM(i,j)\to V_i$, one can stabilize the framings by adding a trivial bundle, and two different lifts become isotopic afterwards. 

As a result, we obtain a $\cJ$-module, and hence a based spectrum, out of a framed Morse--Bott flow category. We denote this spectrum by $|\cM|$ as well.

\subsection{Some properties of (Morse--Bott) flow categories}
\subsubsection{Index shifts of flow categories}\label{subsubsec:indexshift}
\begin{defn}
Let $\cM$ be a flow category. Let \emph{the shift of $\cM$ by $a$}, denoted by $\cM(a)$, be the flow category with the same spaces of objects and morphisms, but $\mu_{\cM(a)}=\mu_\cM+a$. When $\cM$ is framed, $\cM(a)$ is endowed with a canonical framing given by $V'_{i}:=V_{i-a}$.
\end{defn}
Let $X_i$, resp. $X_i'$ denote the index $i$ component of $Ob(\cM)$, resp. $Ob(\cM(a))$. In this notation, $X'_{i}=X_{i-a}$, and the framing bundles agree. Therefore, the $\cJ$-modules corresponding to $\cM$, resp. $\cM(a)$, which are given by $X(i)=\Sigma^{-i}X_i^{V_i}$, resp. $X'(i)=\Sigma^{-i}X_{i-a}^{V_{i-a}}$, are related by $X'(i)= \Sigma^{-a}X(i-a)$. It is easy to identify the $\cJ$-module structure maps as well. As a result
\begin{lem}\label{lem:shiftsame}
The geometric realizations $|\cM|$ and $|\cM(a)|$ are canonically identified. 
\end{lem}
\begin{proof}
Let $\fX_{q}$, resp. $\fX_{q}'$ denote the collection of semi-simplicial sets given in \Cref{definition:geometric-real} corresponding to $X$ and $X'$ respectively. It is easy to check that $\fX_{q}'=\Sigma^{-a}\fX_{q+a}$. Therefore, $|X'|_q=\Sigma^{-a}|X|_{q+a}$, implying $|X'|=\Sigma^{-q+1}|X'|_q=\Sigma^{-q-a+1}|X|_{q+a}=|X|$.
\end{proof}
\subsubsection{Inclusions}
Geometric realizations have functoriality with respect to some fully-faithful embeddings. We define
\begin{defn}\label{defn:inclusion}
Let $\cM_1$ and $\cM_2$ be two flow categories. An \emph{inclusion of flow categories} is given by a fully faithful functor $\iota:\cM_1\to\cM_2$ such that
\begin{enumerate}
	\item\label{item:agreeingindices} the map $\iota:Ob(\cM_1)\to Ob(\cM_2)$ is a smooth closed embedding onto a union of connected components of $Ob(\cM_2)$ 
	\item $\mu_{\cM_2}\circ \iota$ and $\mu_{\cM_1}$ agree
	\item\label{item:alltrajectory} if $x\in Ob(\cM_1)$, $y\in Ob(\cM_2)$ and $\cM_2(\iota(x),y)\neq \emptyset$, then $y$ is in the image of $\iota$
\end{enumerate}
The index $i$ part $X_i$ of $\cM_1$ embeds into the index $i$ part $Y_i$ of $\cM_2$. Moreover, part of the data of the functor $\iota$ is a smooth embedding $\cM_1(X_i,X_j)\hookrightarrow\cM_2(Y_i,Y_j)$. By (\ref*{item:agreeingindices}), both $\cM_1(X_i,X_j)$ and $\cM_2(Y_i,Y_j)$ are $\langle k\rangle$-manifolds, where $k=\mu_{\cM_1}(X_i)-\mu_{\cM_1}(X_j)-1=\mu_{\cM_2}(Y_i)-\mu_{\cM_2}(Y_j)-1$. 
We also require 
\begin{enumerate}[resume]
	\item $\cM_1(X_i,X_j)\hookrightarrow\cM_2(Y_i,Y_j)$ is a morphism of $\langle k \rangle$-manifolds
\end{enumerate}
\end{defn}
\begin{rk}\label{rk:generalinclusions}
We assume $X_i$ is a union of connected components of $Y_i$ as these are the only examples we consider. However, one can drop this assumption, and allow higher codimension embeddings. In this case, we have further requirements. By (\ref*{item:alltrajectory}), $ev_{X_i}\times \iota:\cM_1(X_i,X_j)\xrightarrow{\cong} X_i\times_{Y_i}\cM_2(Y_i,Y_j)$. We ask
\begin{enumerate}[resume]
	\item\label{item:transvcut} $X_i\hookrightarrow Y_i$ and $\cM_2(Y_i,Y_j)\to Y_i$ are transverse maps (see also \Cref{note:fiber-products}). Therefore, $X_i\times_{Y_i}\cM_2(Y_i,Y_j)$ is a submanifold of $X_i\times\cM_2(Y_i,Y_j)$
\end{enumerate}
Notice that by (\ref*{item:transvcut}), $dim(\cM_1(X_i,X_j))-dim(X_i)$ and $dim(\cM_2(Y_i,Y_j))-dim(Y_i)$ coincide. When the flow categories are framed, these would be equal to $rank(V_i)-rank(V_j)-1$.
\end{rk}
The condition (\ref*{item:alltrajectory}) states that if there is a morphism from $\iota(x)$ to $y$, then $y$ is also in the image of $\iota$, i.e. $y=\iota(x')$ for a unique point $x'$. Combined with fully faithfulness, it actually implies that all the morphisms in $\cM_2$ from $\iota(x)$ are in the image of $\iota$. To define an inclusion into a flow category $\cM$, one only needs to specify a union of connected components of $Ob(\cM)$ such that (\ref*{item:alltrajectory}) holds.
\begin{exmp}
Given Morse/Morse--Bott function $f$ and given $a\in\bR$, one can define an inclusion of $\cM_f$ (see \Cref{exmp:morsebottfunction}) by taking the union of all critical sets whose value under $f$ is less than or equal to $a$.  
\end{exmp}
This example also illustrates why (\ref*{item:alltrajectory}) is necessary. Morally, if one was trying to produce a subcomplex of the Morse complex by taking a span of generators, (\ref*{item:alltrajectory}) would be required for this span to be closed under the differential. 
\begin{exmp}
Given Morse--Bott flow category $\cM$ and given $k\in \bZ$, an inclusion is given by by taking the union of objects of index less than or equal to $k$. 
\end{exmp}
\begin{rk}
Notice that if $\iota:\cM_1\hookrightarrow\cM_2$ is an inclusion, then a framing on $\cM_2$ induces a framing on $\cM_1$. This would not be the case if we allowed the generality in \Cref{rk:generalinclusions}. More precisely, as $T_{Y_i}|_{X_i}\neq T_{X_i}$, if $\{W_i\}$ are the corresponding vector bundles on $Y_i$, the restriction of the framing equation
\begin{equation}
	ev_{Y_i}^*W_i	\oplus ev_{Y_i}^*T_{Y_i}=T_{\cM_2(Y_i,Y_j)}\oplus ev_{Y_j}^*W_{Y_j}\oplus \underline{\bR}
\end{equation}
to $\cM_1$ does not define a framing. This can be resolved by considering inclusions with ``trivial normal bundle'', i.e. by assuming $X_i\hookrightarrow Y_i$, $X_j\hookrightarrow Y_j$ and $\cM_1(X_i,X_j)\hookrightarrow\cM_2(Y_i,Y_j)$ have (i) trivial normal bundles of constant rank $d$, (ii) trivializations of the normal bundles that are compatible with the flow category structure. 
\end{rk}
%
The following follows from the definitions, and \Cref{lem:jmapinducegeomreal}: 
\begin{lem}\label{lem:flowincljinclgeomincl}
The inclusions of flow categories induce inclusions of the associated $\cJ$-modules and thus maps of the geometric realizations.
\end{lem}
\begin{proof}
Given an inclusion $\iota:\cM_1\hookrightarrow \cM_2$, write $Ob(\cM_1)$ as $X_i\sqcup Y_i$, where $X_i$ is the image of $\iota$ at index $i$. Then, $\cM_2(i,j)$ decomposes as $\cM_2(X_i,X_j)\sqcup \cM_2(Y_i,X_j)\sqcup \cM_2(Y_i,Y_j)$. Let the framing bundles on $X_i$ and $Y_i$ be denoted by $V_i$, resp. $W_i$. To apply \Cref{constr:jmodule}, we embed the $\cM_2(X_i,X_j)$ component into $V_i\times J_q(i,j)$, and the $\cM_2(Y_i,X_j)$ and $\cM_2(Y_i,Y_j)$ components into $W_i\times J_q(i,j)$. As a result, the $X_i^{V_i}\wedge \cJ_q(i,j)$ component of $(X_i^{V_i}\vee Y_i^{W_i})\wedge \cJ_q(i,j)$ maps into the $\cM(X_i,X_j)^\nu$ component under the collapse map, and this projects to the $X_j^{V_j}$ component of $X_j^{V_j}\vee Y_j^{W_j}$. 
\end{proof}
The condition (\ref*{item:alltrajectory}) of \Cref{defn:inclusion} is necessary for \Cref{lem:flowincljinclgeomincl} to hold. For instance, if a framed flow category $\cM$ is supported in two indices $i_0>i_1$ and if $\cM(i_0,i_1)\neq \emptyset$, then $X_{i_0}^{V_{i_0}}\wedge \cJ(i_0,i_1)\to X_{i_1}^{V_{i_1}}$ is generally non-trivial and $X_{i_0}^{V_{i_0}}$ supported in degree $i_0$ does not define a sub-module of the $\cJ$-module corresponding to $\cM$. Therefore, there is no natural map from the geometric realization of the subcategory spanned by $X_{i_0}$ into $|\cM|$. 

\subsubsection{Quotients}
As mentioned, one can think of inclusions of flow categories as analogous to subcomplexes of chain complexes. If we assume the dual to the condition (\ref*{item:alltrajectory}), i.e.\ that the trajectories with endpoints in $Im(\iota)$ are in $Im(\iota)$, this does not define a notion of flow subcategory whose geometric realization maps to that of the original category. However, such a subcategory can be thought of as a quotient:
\begin{defn}
Let $\iota:\cM_0\to\cM$	be an inclusion of flow categories. Define the flow category $\cM/\cM_0$ to be the full subcategory of $\cM$ with objects $Ob(\cM/\cM_0)=Ob(\cM)\setminus Ob(\cM_0)$. We call $\cM/\cM_0$ the \emph{quotient flow category}. If $\cM$ is framed, $\cM/\cM_0$ also inherits the framing.
\end{defn}
As observed, to specify a quotient category, it suffices to specify a union of the connected components of $Ob(\cM)$ such that if there is a trajectory from a point $x$ in $Ob(\cM)$ to a point $y$ in this subset, $x$ also belongs to this subset. 
\begin{rk}
The reason $\cM/\cM_0$ is a still a flow category is because the morphisms in $\cM$ with domain and target in $\cM/\cM_0$ cannot factor through the objects of $\cM_0$, otherwise the target would also be contained in $Ob(\cM_0)$.	
\end{rk}
Note that there is no quotient map $\cM\to\cM/\cM_0$ of flow categories. For instance, when $\cM_0=\cM$, $\cM/\cM_0$ is the empty category and $\cM$ admits no functors to this, as there is no morphism from a non-empty set to the empty set. This issue is not present for based spaces/spectra, and we have maps of the induced $\cJ$-modules, which in turn induce maps of geometric realizations. We have: 
\begin{lem}\label{lem:jmodquot}
The $\cJ$-module associated to $\cM/\cM_0$ is the quotient of the $\cJ$-module associated to $\cM$ by that for $\cM_0$. 
\end{lem}
\begin{proof}
Assume we are in the setting of the proof of \Cref{lem:flowincljinclgeomincl}, i.e. decompose $Ob(\cM)$ as $X_i\sqcup Y_i$, where $X_i$ is the component corresponding to $\cM_0$ and $Y_i$ to $\cM/\cM_0$ (and let $V_i$, $W_i$ denote the framing bundles on $X_i$, $Y_i$, respectively). As before $\cM(i,j)$ decomposes as $\cM(X_i,X_j)\sqcup \cM(Y_i,X_j)\sqcup \cM(Y_i,Y_j)$, and applying \Cref{constr:jmodule} with each of these components, we obtain the structure maps $X_i^{V_i}\wedge \cJ(i,j)\to X_j^{V_j}$, $Y_i^{W_i}\wedge \cJ(i,j)\to X_j^{V_j}$, and $Y_i^{W_i}\wedge \cJ(i,j)\to Y_j^{W_j}$ respectively. The $\cJ$-module corresponding to $\cM$ has $i^{th}$-spectrum $X_i^{V_i}\vee Y_i^{W_i}$, where $X_i^{V_i}$ gives the submodule corresponding to $\cM_0$. Taking the quotient by this submodule, we obtain a $\cJ$-module with the $i^{th}$-spectrum $Y_i^{W_i}$, and the only structure map that survives is $Y_i^{W_i}\wedge \cJ(i,j)\to Y_j^{W_j}$. This is induced by applying \Cref{constr:jmodule} to $\cM(Y_i,Y_j)=(\cM/\cM_0)(Y_i,Y_j)$, and therefore is the same as the structure map of the $\cJ$-module corresponding to $\cM/\cM_0$. 
\end{proof}

The following \namecref{cor:puppemaprealization} will be used to write continuation type maps between realizations of flow categories. It follows from \Cref{lemma:quotient-inclusion}, \Cref{lem:flowincljinclgeomincl} and \Cref{lem:jmodquot}. 
\begin{cor}\label{cor:puppemaprealization}
There are homotopy equivalences $|\cM/\cM_0|\simeq |\cM|/|\cM_0|\simeq cone (|\cM_0|\to|\cM|)$. In particular we have a map $|\cM/\cM_0|\to\Sigma |\cM_0|$. 
\end{cor}

\subsection{Ind-flow categories} 
We will need to use Morse--Smale/Floer data parametrized by \emph{ind-manifolds}, such as $S^\infty$. An \emph{ind-manifold} is a colimit of finite dimensional smooth manifolds under smooth inclusions, and we will consider Morse/Morse--Bott functions on such manifolds. In general, the corresponding flow categories are not going to be finite; therefore, we introduce the following conceptual tool to work with them. 
\begin{defn}
An \emph{ind-flow category} is an ind-object in the category of finite Morse--Bott flow categories and their inclusions. In other words, it is a topological category $\cM$ such that there exists finite flow categories $\cM_k$ and inclusions $\cM_k\hookrightarrow \cM_{k+1}$ satisfying $\cM=\bigcup\cM_k$. A \emph{framed ind-flow category} is an ind-flow category with a presentation as above such that each $\cM_k$ admits a framing and the inclusions are compatible with framings.
\end{defn}
Note that $\cM_k$ are not strictly part of the data, and $\cM$ admits other equivalent presentations as such a colimit. 
\begin{defn}
For a framed ind-flow category, with a fixed presentation, define \emph{the geometric realization of $\cM$} to be the spectrum $\colim_k |\cM_k|=\hocolim_k |\cM_k|$. We denote this spectrum by $|\cM|$.
\end{defn}
Notice that this definition is independent of the choice of equivalent presentations. 
\begin{exmp}
Let $\cM$ be a (possibly infinite) Morse--Bott flow category with index function bounded below. Let $\tau_{\leq k}\cM$ denote the finite flow category obtained from $\cM$ by truncating the objects of index higher than $k$ (and the morphisms from them). If $\cM$ is framed, the framing restricts to each $\tau_{\leq k}\cM$, and $\tau_{\leq k}\cM\hookrightarrow\tau_{\leq k+1}\cM\hookrightarrow\cM$ describes a presentation of $\cM$ as a framed ind-flow category. 
\end{exmp}
\begin{exmp}
Consider an ind-manifold presented as $M=\bigcup_{i=1}^{\infty} M_i$, where each $M_i$ is a smooth compact manifold, and $M_i\hookrightarrow M_{i+1}$. Inductively, choose a function $f:M\to \bR$, such that the restriction of $f$ to each $M_i$ is Morse, $crit(f|_{M_i})\subset crit(f|_{M_{i+1}})$, and the indices for different restrictions agree. With a reasonable choice of a metric on $M$, this implies there are no trajectories from a point of $f|_{M_i}$ to outside $M_i$. Therefore, $\cM_{f|_{M_i}}$ form a sequence of inclusions, defining an ind-flow category corresponding to $f$. 
\end{exmp}
This will be the main type of example of ind-flow categories we consider, and more specific examples will be given later. By abuse of notation we will generally treat ind-flow categories as if it is a finite flow category.

\section{Bimodules, relative modules and maps between geometric realizations}\label{sec:bimodreal} 
\subsection{Bimodules}
Given two Morse--Smale pairs $(f_1,g_1)$ and $(f_2,g_2)$, one can write a quasi-isomorphism $CM(f_1)\to CM(f_2)$ by counting the trajectories from the critical points of $f_1$ to the critical points of $f_2$ that satisfy the (in our conventions negative) gradient flow equation for $(f_1,g_1)$ near $-\infty$, and that satisfy the gradient flow equation for $(f_2,g_2)$ near $+\infty$. This idea generalizes to Floer homology as well. To extend this to the homotopy theoretic setting, we need to collect all the continuation trajectories into an abstract structure. The relevant definition is that of a bimodule.

\subsubsection{Definition of a bimodule}
\begin{defn}[cf.\ {\cite[Defn.\ 3.6]{largethesis}}]\label{defn:bimodule}
Let $\cM_1$ and $\cM_2$ be two flow categories. A \emph{monotone $\cM_1$-$\cM_2$ bimodule $\cN$} is given by the following data
\begin{itemize}
	\item for each pair of objects $X\subset Ob(\cM_1)$, $Y\subset Ob(\cM_2)$ a smooth compact $\langle \mu_{\cM_1}(X)-\mu_{\cM_2}(Y)\rangle$-manifold $\cN(X,Y)$ (in particular, $\cN(X,Y)=\emptyset$ if $\mu(X)<\mu(Y)$)
	\item evaluation maps $ev_X:\cN(X,Y)\to X$, $ev_Y:\cN(X,Y)\to Y$
	\item smooth gluing/composition maps $\cM_1(W,X)\times_X \cN(X,Y)\to \cN(W,Y)$ and $\cN(X,Y)\times_Y \cM_2(Y,Z)\to \cN(X,Z)$
\end{itemize}
satisfying the following properties analogous to \Cref{defn:flowcategory}
\begin{itemize}
	\item the evaluation maps $\cM_1(W,X) \to X$ and $\cN(X,Y)\to X$ are transverse to each other; similarly the maps $\cN(X,Y) \to Y$ and $\cM_2(Y, Z) \to Y$ are transverse.
	\item the composition maps are diffeomorphisms into boundary faces and $\partial_i \cN(X,Y)$ is given by the disjoint union of the images of compositions from $\cM_1(X,R)\times_Q \cN(R,Y)$ with $\mu_{\cM_1}(R)=\mu_{\cM_2}(Y)+i-1$ and $\cN(X,Q)\times_Q \cM_2(Q,Y)$ with $\mu_{\cM_2}(Q)=\mu_{\cM_2}(Y)+i$
	\item $(\cM_1(X,R)\times_Q \cN(R,Y))\cap (\cN(X,Q)\times_Q \cM_2(Q,Y))=\emptyset$ if $\mu_{\cM_1}(R)<\mu_{\cM_2}(Q)$
\end{itemize}
We define an \emph{$\cM_1$-$\cM_2$ bimodule $\cN$} to be a monotone $\cM_1(a)$-$\cM_2$ bimodule for some $a$ (see \Cref{subsubsec:indexshift}).
\end{defn}
What we called monotone bimodules are defined in \cite[Defn. 3.6]{largethesis}. Note that in the Morse--Bott case, there can be continuation trajectories increasing the index, and a monotone $\cM_1(a)$-$\cM_2$ bimodule can be seen as a notion allowing an index increase of at most $a$. Throughout the paper, we will only use monotone bimodules; however, we include a discussion of the more general bimodules for completeness. 

Part of \cite[Defn. 3.6]{largethesis} is a finer boundary structure on $\cN(X,Y)$ than that of a $\langle \mu_{\cM_1}(X)-\mu_{\cM_2}(Y)\rangle$-manifold. This is imposed by \Cref{defn:bimodule}. To explain this further, recall that the $\langle i-j-1\rangle$-manifold structure of $\cM(i,j)$ keeps track of the indices between $i$ and $j$ so that the trajectories break. On the other hand, in the case of a monotone bimodule, a breaking can occur on both $\cM_1$-side and $\cM_2$ side. In other words, $\partial_b \cN(i,j)$ consists of the disjoint components $\cM_1(i,j+b-1)\times_{j+b-1} \cN(j+b-1,j)$ and $\cN(i,j+b)\times_{j+b} \cM_2(j+b,j)$. We can abstract this notion out, and define a \emph{$\langle 2,k\rangle$-manifold} as a $\langle k\rangle$-manifold $M$ such that $\partial_b M$ admits a decomposition $\partial_{1,b} M\sqcup \partial_{2,b} M$ into two closed unions of faces satisfying $\partial_{1,b_1}M\cap \partial_{2,b_2}M=\emptyset$ for every $b_2\geq b_1$. It follows that $\partial_{\epsilon,b}M\cap \partial_{\epsilon',b'}M$ is a boundary face of both $\partial_{\epsilon,b}M$ and $\partial_{\epsilon',b'}M$, as long as $(\epsilon,b)\neq (\epsilon',b')$. Clearly, $\cN(i,j)$ is a $\langle 2,i-j\rangle$-manifold, with $\partial_{1,b}\cN(i,j)=\cM_1(i,j+b-1)\times_{j+b-1} \cN(j+b-1,j)$ and $\partial_{2,b}\cN(i,j)=\cN(i,j+b)\times_{j+b} \cM_2(j+b,j)$. 

Because of this finer structure, the notion of $\cM_1$-$\cM_2$ bimodule does not depend on the constant $a$ above. More precisely, if we are given a monotone $\cM_1(a)$-$\cM_2$ bimodule $\cN$, and if $a'>a$, one automatically gets a monotone $\cM_1(a')$-$\cM_2$ bimodule. The spaces $\cN(X,Y)$ do not change; however, due to the index shift on $X$ by $a'-a$, one needs to turn the $\langle k\rangle$-manifold $\cN(X,Y)$ into an $\langle k+a'-a\rangle$-manifold. It is easier to define $\langle 2,  k+a'-a\rangle$-manifold structure out of a $\langle2,k\rangle$-manifold: define the new $\partial_{1,b}$ to be the old $\partial_{1,b-(a'-a)}$ (or $\emptyset$ if $b\leq a'-a$), and let the new $\partial_{2,b}$ be the same as the old one (or $\emptyset$ if $b>k$).

If $\cM_1, \cM_2$ are flow categories with metrics, a \emph{metric} on a monotone $\cM_1$-$\cM_2$ bimodule $\cN$ is the data of a Riemannian metric on each $\cN(X,Y)$, such that the product metric on $\cM_1(W,X)\times_X \cN(X,Y)$ and $\cN(X,Y)\times_Y \cM_2(Y,Z)$ coincides with the pullback metric under the embeddings into $\cN(W,Y)$, resp. $\cN(X,Z)$. The notion extends to (non-monotone) bimodules immediately. A metric on a bimodule $\mathcal{N}$ induces splittings analogous to \eqref{equation:boundary-isom} and \eqref{equation:splitting-fiber}. 

\begin{defn}\label{defn:bimoduleframing} Assume $\cM_1$ and $\cM_2$ are both framed and call the associated bundles $V_X$, $W_Y$, respectively. A \emph{framing} on a monotone bimodule $\cN$ is defined to be a collection of isomorphisms $V_X\oplus T_X\cong T_{\cN(X,Y)}\oplus W_Y$ satisfying compatibility conditions analogous to \eqref{equation:framing-diagram-square}. A \emph{framing} on a $\cM_1$-$\cM_2$-bimodule $\cN$ (say given as a monotone $\cM_1(a)$-$\cM_2$-bimodule) is a framing on it as a $\cM_1(a)$-$\cM_2$-bimodule, where $\cM_1(a)$ is endowed with the canonical induced framing (see \Cref{subsubsec:indexshift}). Note that this notion also does not depend on $a$. 
\end{defn}
One can turn monotone bimodules into flow categories. Given a monotone $\cM_1$-$\cM_2$-bimodule $\cN$, there exists a flow category $\cP$ such that 
\begin{itemize}
	\item $Ob(\cP)=Ob(\cM_1(1))\sqcup Ob(\cM_2)$ with the index function $\mu_{\cM_1(1)}\sqcup \mu_{\cM_2}$
	\item $\cM_1(1)$ and $\cM_2$ are both full subcategories of $\cP$ and the morphisms from a component $X\subset Ob(\cM(1))$ to $Y\subset Ob(\cM_2)$ are given by $\cN(X,Y)$
	\item there are no morphisms from $Y\subset Ob(\cM_2)$ to $X\subset Ob(\cM(1))$ 
\end{itemize}
As a result, $\cM_2\hookrightarrow\cP$ is an inclusion, with the corresponding quotient given by $\cM_1(1)$. More generally, given $\cM_1$-$\cM_2$-bimodule $\cN$ (say given as a monotone $\cM_1(a)$-$\cM_2$-bimodule), there exists a flow category $\cP$ constructed as above such that $Ob(\cP)=Ob(\cM_1(a+1))\sqcup Ob(\cM_2)$.
\begin{rk}
In the other direction, the data of a flow category $\cP$ and an inclusion $\cM_2\hookrightarrow\cP$ such that the corresponding quotient given by $\cM_1(a+1)$ is nearly equivalent to \Cref{defn:bimodule}, except the last condition that guarantees $\langle 2,k\rangle$-manifold structures may fail. 
\end{rk}
Assume $\cM_1$, $\cM_2$ and the monotone bimodule $\cN$ are framed. Then, $\cP$ is naturally framed. To see this, let $V_X$, $W_Y$ denote the framing bundles on $\cM_1$ and $\cM_2$, respectively. Define 
\begin{equation}
U_X:=	
\begin{cases}
V_X\oplus\underline{\bR}\text{ for }X\subset Ob(\cM_1(1))\subset Ob(P)\\
W_X\text{ for }	X\subset Ob(\cM_2)\subset Ob(P)	
\end{cases}
\end{equation}
It is easy to check the framing condition \eqref{equation:framing-equation} on $\cP$. More generally, if $\cN$ is a framed bimodule (not necessarily monotone), one can define a natural framing on $\cP$ with the same framing bundles.

We will use the construction of the flow category $\cP$ to induce maps $|\cM_1|\to |\cM_2|$. To give a better context, recall how one can induce chain maps $CM(f_1)\to CM(f_2)$ associated to an interpolation from a Morse function $f_1$ to another $f_2$. Let $\tilde{f}:I\times M\to \bR$ denote the interpolation, i.e. assume $\tilde{f}(0,\cdot)=f_1$, $\tilde{f}(1,\cdot)=f_2$. Let $g:I\to \bR$ be a Morse function on $I$ with only critical points at $0$ and $1$ and with negative gradient flow from $0$ to $1$ (e.g. let $g(t)=t^3/3-t^2/2$). Then $\tilde f+g$ is a Morse function on $I\times M$, with Morse complex given by $CM(f_1)\oplus CM(f_2)[-1]$ such that we have a short exact sequence 
\begin{equation}
0\to CM(f_2)[-1]\to CM(\tilde f+g)\to CM(f_1)\to 0	
\end{equation}
and the associated connecting map $CM(f_1)\to CM(f_2)$ is the map we are looking for.

Similarly, 
\begin{prop}\label{prop:inducedmap}
A framed $\cM_1$-$\cM_2$-bimodule $\cN$ induces a map $c_\cN:|\cM_1|\to |\cM_2|$.	
\end{prop}
\begin{proof}
Construct the flow category $\cP$ associated to $\cN$ with an inclusion $\cM_2\to\cP$ and quotient $\cM_1(a+1)$. By \Cref{cor:puppemaprealization}, we have a map $|\cM_1(a+1)|\to \Sigma |\cM_2|$. Notice that the framing bundles on $\cM_1(a+1)$ are given by $V_X\oplus \underline{\bR}$, and it is not the canonical framing we discussed in \Cref{subsubsec:indexshift}. Instead, it is the canonical framing stabilized by $\underline{\bR}$ and its geometric realization is $\Sigma|\cM_1|$ by \Cref{lem:shiftsame}. Hence we have a map $\Sigma|\cM_1|\to \Sigma |\cM_2|$, and we obtain $c_\cN$ by desuspending.
\end{proof}
It is easy to see the independence up to homotopy of the induced map from the chosen neat embeddings for $\cM_1$, $\cM_2$, and $\cN$. On the other hand, as a bimodule $\cN$ can be seen as a monotone $\cM_1(a)$-$\cM_2$-bimodule for every large enough $a$, and we have to check the independence of $c_\cN$ from $a$. As we will only use monotone bimodules, we skip this proof.

\subsection{$P$-relative modules}
A bimodule of flow categories induces a map between the geometric realizations. We will also need to write a map to the free loop space. Instead of realizing the latter as a geometric realization, we can use the following notion:
\begin{defn}\label{definition:relative-module}
Given flow category $\cM$ and space $P$, a \emph{$P$-relative module $\cN$} is a module over $\cM$ with a map to $P$. Concisely, it is a $\cM$-$*(0)$-bimodule $\cN$ (where $*(0)$ is the flow category supported at index $0$ and with object space $*$) with a map $\cN(i):=\cN(i,0)\to P$ such that the composition $\cM(i,j)\times_j \cN(j)\to \cN(i)\to P$ is the same as the map induced by $\cN(j)\to P$. A \emph{framing on $\cN$} is a choice of isomorphisms $V_i\oplus T_{X_i}=T_{\cN(i)}\oplus \nu$ over $\cN(i)$. Here $\nu$ is a virtual vector bundle over $P$.  
\end{defn}
More explicitly, $\cN$ is defined by the data of (i) an $\langle i+a\rangle$-manifold $\cN(i)$ for every $i$, (ii), evaluation maps $\cN(i)\to X_i$, (iii) gluing maps $\cM(i,j)\times_j \cN(j)\to \cN(i)$, (iii) compatible maps $\cN(i)\to P$ (also called evaluation maps). Here, $a$ is a shift constant as in the previous section. We assume $-a-1$ is less than the index support of $\cM$ for simplicity.

Observe that if $P$ is a closed manifold, this notion coincides with that of a $\cM$-$P(0)$ bimodule, where $P(0)$ is the Morse--Bott flow category supported at index $0$ with object space $P$. The important examples of relative modules will arise in Morse theory from the moduli of half gradient trajectories (\Cref{subsec:applcjsmorse,sec:equivariantmorse}), and in Floer theory from the moduli of holomorphic half cylinders (\Cref{sec:clmaps}). One can show:
\begin{prop}\label{prop:relative-module}
A framed $P$-relative module induces a morphism $|\cM|\to P^\nu$, and the morphism is independent of the shift constant $a$. 	
\end{prop}
\begin{proof}
The argument in \Cref{prop:inducedmap} goes through without much change. Namely, one can treat $P$ as a flow category and form the cone category with objects $ob(\cM(a+1)) \cup P$. Since this category has no morphisms with domain on $P$, the previous constructions go through. 

Since we have not shown independence from the shift constant before, we give a more direct construction. Fix the neat embeddings for $\cM$, and choose embeddings $\cN(i)\hookrightarrow J_{q_i}(i,-a-1)\cong J_{q_i}(i+a+1,0)$ such that 
\begin{equation}\label{eq:neatprel}
	\xymatrix{\cM(i,j)\times_j\cN(j) \ar[r]\ar@{^{(}->}[d]& \cN(i)\ar@{^{(}->}[d] \\ J_{q_i-q_j}(i,j) \times J_{q_j}(j,-a-1)\ar[r]& J_{q_i}(i,-a-1) }
\end{equation}
commutes (c.f. \eqref{eq:neatflowemb}). The induced embedding $\cN(i)\hookrightarrow V_i\times J_{q_i}(i,-a-1)$ where the $V_i$-component is given by the composition $\cN(i)\to X_i\xrightarrow{0} V_i$ induces 
\begin{equation}\label{equation:explicit-map-module}
	X_i^{V_i}\wedge \cJ_{q_i}(i,-a-1)\to \cN(i)^{V_i+T_{X_i}-T_{\cN(i)}+\underline{\bR}^{i+a+q_i}}\to P^{\nu+\underline{\bR}^{i+a+q_i}}
\end{equation}
which by suspension/desuspension gives 
\begin{equation}\label{eq:baserelmap}
	\Sigma^{-a}X(i)\wedge \cJ(i,-a-1)= \Sigma^{-i-a}X_i^{V_i}\wedge \cJ(i,-a-1)\to P^\nu
\end{equation}
that is compatible with the $\cJ$-module structure maps $X(i)\wedge \cJ(i,j)\to X(j)$ by \eqref{eq:neatprel}. In other words, 
\begin{equation}\label{eq:commutetopnu}
	\xymatrix{ \Sigma^{-a}X(i)\wedge \cJ(i,j)\wedge \cJ(j,-a-1)\ar[r]\ar[d]& \Sigma^{-a}X(j)\wedge  \cJ(j,-a-1)\ar[d] \\
	\Sigma^{-a}X(i)\wedge \cJ(i,-a-1) \ar[r]&	P^\nu
	}
\end{equation} 
commutes. This can be proved by an argument analogous to the proof of \Cref{proposition:j-modules}, using as input the diagram \eqref{eq:neatprel} in place of \eqref{eq:neatflowemb}. 
Let $\fX_{q}$ denote the semi-simplicial set given by 
\begin{equation}
	\fX_{q}[k]=\bigvee_{i_0>\dots>i_k} X(i_0)\wedge \cJ(i_0,i_1)\wedge \dots\wedge \cJ(i_{k-1},i_k)\wedge \cJ(i_k,-q)	
\end{equation}
as before. There is a well-defined map $\Sigma^{-a}\fX_{a+1}[k]\to P^\nu$, given by composing $\cJ$-module structure maps, the maps $\Sigma^{-a}X(i)\wedge \cJ(i,-a-1) \to P^\nu$, and the multiplication in the category $\cJ$. The composition can be taken to be in any order, and well-definedness also follows from the commutativity of \eqref{eq:commutetopnu}. In particular, this map commutes with the boundary maps of the semi-simplicial spectrum $\Sigma^{-a}\fX_{a+1}$. Combining this map with  $\Sigma^\infty\Delta_+^k\to \bS$ (induced by $\Delta_+^k\to S^0$), we obtain maps 
\begin{equation}
	\Sigma^{-a}\fX_{a+1}[k]\wedge \Sigma^\infty\Delta_+^k\to \Sigma^{-a}\fX_{a+1}[k]\wedge \bS=\Sigma^{-a}\fX_{a+1}[k]\to P^\nu
\end{equation}
inducing $|\cM|=\Sigma^{-a}|\fX_{a+1}|=|\Sigma^{-a}\fX_{a+1}|\to P^\nu$. 

To see this map is independent of $a$, one needs to check the compatibility with the isomorphism of \Cref{lem:suspendedrealization}. First, we replace $a$ by $a+1$, and use the embedding $\cN(i)\hookrightarrow J_{q_i}(i,-a-1)\hookrightarrow J_{q_i}(i,-a-2)$, where the second map is an appropriate inclusion into the interior. One can fix the one induced by $J(i,-a-1)\xrightarrow{id\times 1} J(i,-a-1)\times[0,\infty)=J(i,-a-2)$, i.e. it is equal to $1$ in the second component. 
The analogous map $\Sigma^{-a-1}X(i)\wedge \cJ(i,-a-2)\to P^\nu$ can be obtained by desuspending the composition
\begin{equation}
	X(i)\wedge \cJ(i,-a-2)\to X(i)\wedge \cJ(i,-a-1)^{\underline{\bR}}=\Sigma X(i)\wedge\cJ(i,-a-1)\to P^{\nu+\underline{\bR}^{a+1}}
\end{equation}
Here $\cJ(i,-a-2)\to  \cJ(i,-a-1)^{\underline{\bR}}$ is the collapse map corresponding to the inclusion. Under the identification $\cJ(i,-a-2)=\cJ(i,-a-1)\wedge[0,\infty]$, this map is induced by the map $[0,\infty]\to S^1$ that identifies the endpoints. Notice that the composition \begin{equation}
	X(i)\wedge \cJ(i,-a-1)\wedge\cJ(-a-1,-a-2)\to X(i)\wedge \cJ(i,-a-2)\to P^{\nu+\underline{\bR}^{a+1}}
\end{equation} is constant since $\cN(-a-1)=\emptyset$ (or since $X_{-a-1}=\emptyset$). Now consider the decomposition
\begin{equation}
\fX_{a+2}[k]=(\fX_{a+1}\wedge [0,\infty])\vee \fX_{a+1}[k-1]
\end{equation}
observed before (see the proof of \Cref{lem:suspendedrealization}). The map $\fX_{a+2}[k]\to P^{\nu+\underline{\bR}^{a+1}}$ is also constant over $\fX_{a+1}[k-1]$, and the map on the $(\fX_{a+1}\wedge [0,\infty])$ part is given by the composition
\begin{equation}
\fX_{a+1}\wedge [0,\infty]\to \fX_{a+1}\wedge S^1\to P^{\nu+\underline{\bR}^{a}}\wedge S^1=P^{\nu+\underline{\bR}^{a+1}}
\end{equation}
As a result, under the identification $|\fX_{a+2}|\simeq \Sigma|\fX_{a+1}|$, the map $|\fX_{a+2}|\to P^{\nu+\underline{\bR}^{a+1}}$ can be seen as the suspension of $|\fX_{a+1}|\to P^{\nu+\underline{\bR}^{a}}$. This finishes the proof.

\end{proof}
\begin{rk}
Assume $P$ is a closed manifold, and we regard it as a flow category supported at index $0$, framed by $\nu$. Then, one can induce a map $|\cM|\to P^\nu$ by \Cref{prop:inducedmap}. We leave it to the reader to check that this map matches the one constructed in \Cref{prop:relative-module}. 
\end{rk}
\begin{rk}
One can also define a $P$-relative version of $\cM_1$-$\cM_2$-bimodules, as a bimodule with extra maps to $P$. Presumably, one can obtain maps $|\cM_1|\to |\cM_2|\wedge P^\nu$. We have not checked the details, as this is not needed.
\end{rk}
\begin{exmp}\label{exmp:tautpmodule}
Let $\cM$ be a framed flow category supported at index $i_0$. Then, $X_{i_0}$ can be considered as a framed $X_{i_0}$-relative module over $\cM$. More precisely, let $P=X_{i_0}$ and assume $\nu=V_{i_0}$ (where $\nu$ refers to the bundle over $P$). Then, $\cN$ given by $\cN(i_0)=P=X_{i_0}$ is a $P$-relative $\cM$-module, with the evaluation maps $\cN(i_0)\to X_{i_0}$, $\cN(i_0)\to P$ given by the identity. The framing $V_{i_0}\oplus T_{X_{i_0}}=T_{\cN(i_0)}\oplus \nu$ is tautological. 
\end{exmp}
We observe:
\begin{lem}\label{lem:tautmoduleinduceid}
The relative module in \Cref{exmp:tautpmodule} induces the identity morphism.
\end{lem}
\begin{proof}
The claim follows by going through the construction in \Cref{prop:relative-module} for $\fX_{1-i_0}$ (i.e. $a=-i_0$, notice that the $0$-simplices of $\fX_{1-i_0}$ is given by $X(i_0)=\Sigma^{-i_0}X_{i_0}^{V_{i_0}}$, and the rest of its simplices vanish) 
and from the fact that the map produced by the relative Pontryagin-Thom construction in \eqref{equation:explicit-map-module} is the identity.
\end{proof}

\subsection{Application: the stable homotopy type via Morse--Bott theory}\label{subsec:applcjsmorse}
In this section, we prove a Morse--Bott analogue of the Cohen--Jones--Segal theorem. Namely, we show how to recover the stable homotopy type of smooth compact manifold from a Morse--Bott function. Indeed, we will recover the stable homotopy type of Thom spectra. Namely:
\begin{thm}\label{thm:morsebottrecovers}
	Let $M$ be a smooth, compact manifold, and let $\nu$ be a virtual bundle over $M$. Let $f: M \to \mathbb{R}$ be a Morse--Bott function. Consider the flow categories $\cM_f$ (in the presence of a Morse--Bott--Smale metric) and $\cM_f^{act}$ (see \Cref{example:flow-cat-act}) with the framing bundles given by $V_X\oplus \nu$, where $V_X$ is as in \Cref{exmp:morsebottfunction}. Then $|\cM_f|$ and $|\cM_f^{act}|$ are homotopy equivalent to the Thom spectrum $M^\nu$.
\end{thm}
As mentioned before, Morse--Bott--Smale metrics do not always exist; however, the flow category $\cM_f^{act}$ is always well-defined.

The proof of \Cref{thm:morsebottrecovers} is morally the same as considering the continuation maps between $|\cM_f|$, resp. $|\cM_f^{act}|$ and $|\cM_0|$, resp. $|\cM_0^{act}|$. The latter (with framing $\nu$) is clearly isomorphic to $M^\nu$. Instead, we use the $P$-relative modules constructed in the previous section. This has the advantage of generalizing to maps from symplectic cohomology to loop spaces, and is possibly similar to \cite[Appendix D]{abouzaidblumberg}. Before proceeding to the proof, we state two immediate corollaries:
\begin{cor}\label{cor:dualframe}
	Let $W_X$ denote the ascending bundle of a critical manifold. Since $W_Y\oplus T_Y\cong T_{\cM_f(X,Y)}\oplus \underline{\bR}\oplus W_X$, one can define an alternative framing of $\cM_f$, resp. $\cM_f^{act}$ with index bundle given by $X\mapsto -W_X-T_X$. The geometric realization is the Atiyah dual $M^{-T_M}$ of $M$.
\end{cor}
\begin{proof}
	One can easily match $-W_X-T_X$ with $V_X-T_M$ (and similarly the framing isomorphisms). The result follows from \Cref{thm:morsebottrecovers} by taking $\nu=-T_M$. 
\end{proof}
One can modify the framing in this way for every framed flow category, and the geometric realization is the Atiyah dual. We will not show this though.
\begin{cor}[Morse--Bott inequalities for extraordinary cohomology theories I.]\label{corollary:betti-bound}
	Let $R$ be an extraordinary cohomology theory such that $R_*(pt_+)=R_*(\bS)$ is finite in each degree. Then, $rk(R_*(M^\nu))\leq \sum rk(R_*(X_i^{V_i+\nu}))$. In particular, when the bundles $V_i$ are $R$-orientable, 
	\begin{equation}
		rk(R_*(M_+))\leq \sum rk(R_{*-rk(V_i)}(X_i))
	\end{equation} 
\end{cor}
\begin{proof}
	One can filter by flow categories $\cM_f$, resp. $\cM_f^{act}$ for which the subquotients have realization given by some $X_i^{V_i+\nu}$ (see \Cref{lem:singledegree}). The statement follows by induction and the equivalence $\cM_f\simeq\cM_f^{act}\simeq M^\nu$. 
\end{proof}
The finiteness condition is to ensure that the ranks above are well-defined. On the other hand, this condition can be weakened: for instance, one can assume $R$ is an $H\bK$-module, for a field $\bK$, and $R_*(\bS)$ is finite dimensional over $\bK$ (in this case, the rank would be replaced by $dim_\bK$). \Cref{corollary:betti-bound} generalizes the classical fact in Morse theory that the rank of homology is bounded above by the number of critical points of index $k$ of $f$.

Similarly, one can replace $rk$, resp. $dim_\bK$ with the $\bZ[[t,t^{-1}]]$-valued dimension $dim_t$. For a degreewise finite graded abelian group $M_*$, define $dim_t(M_*)=\sum rk(M_i)t^i\in \bZ[[t,t^{-1}]]$ (similar with $\bK$-modules). The group $\bZ[[t,t^{-1}]]$ is partially ordered: define $P(t) \preceq Q(t)$, if there exists $A(t) \in \mathbb{Z}[[t,t^{-1}]]$ with non-negative coefficients such that $Q(t)=P(t)+(1+t)A(t)$ (concretely, there is a sequence $a_i\in \bZ_{\geq 0}$ such that the $i^{th}$ coefficient of $Q(t)-P(t)$ is given by $a_i+a_{i-1}$). It is easy to check that $dim_t$ is sub-additive, i.e. for a given short exact sequence $0\to C'\to C\to C''\to 0$
\begin{equation}
	dim_t(H(C'))+dim_t(H(C''))\geq dim_t(H(C))
\end{equation}
The proof of \cite[Lem.\ 2.14]{nic} goes without change to establish this claim. By replacing $rk$ with $dim_t$ in the proof of \Cref{corollary:betti-bound}, we obtain: 
\begin{cor}[Morse--Bott inequalities for extraordinary cohomology theories II.]\label{cor:mbineqs}
Under the assumptions in \Cref{corollary:betti-bound}, $dim_t(R_*(M^\nu))\preceq \sum dim_t(R_*(X_i^{V_i+\nu}))$. Moreover, when each $V_i$ is $R$-orientable, 	
\begin{equation}
	dim_t(R_*(M_+))\preceq \sum t^{rk(V_i)}dim_t(R_{*}(X_i))
\end{equation} 
\end{cor}
\Cref{cor:mbineqs} strengthens \Cref{corollary:betti-bound} slightly, and it is closer to what is called Morse--Bott inequalities in the literature (e.g. see \cite{hurtubise2012three, rot2016morse}).  

Next, we proceed to the proof of \Cref{thm:morsebottrecovers}. 
\begin{defn}
	The \emph{standard module} on $M$ is the $M$-relative $\cM_f$, resp. $\cM_f^{act}$-module which associates to an object its descending manifold. Explicitly, $\cN(i)$ is given by the compactification of the moduli space of half-trajectories $\tilde{x}: (-\infty, 0] \to M$ satisfying $ \frac{d}{dt} \tilde{x} = - grad_{g}f$ and $\lim_{t \to -\infty} \tilde{x}(t)= x\in X_i$. The $M$-relative structure is given by $\tilde x\mapsto \tilde x(0)$. Clearly, $V_X\oplus T_{X_i} \cong T_{\cN(i)}$, which endows $\cN$ with a framing given by $(V_X\oplus \nu)\oplus T_{X_i} \cong T_{\cN(i)}\oplus \nu$.
\end{defn}
By \Cref{prop:relative-module}, we have maps
\begin{equation}\label{equation:eval}
	| \cM_f| \to M^\nu  \text{ and }     | \cM_f^{act}| \to M^\nu 
\end{equation} 
\begin{proof}[Proof of \Cref{thm:morsebottrecovers}]
	Both $\cM_f$ and  $\cM_f^{act}$ are filtered as follows: define $F^p\cM_f$, resp. $F^p\cM_f^{act}$ to be the full subcategory spanned by critical manifolds whose value of $f$ is less than or equal to $p$. This induces a filtration on the corresponding geometric realizations. 
	Moreover, if we consider the restriction of $\cN$ to $F^p\cM_f$, resp. $F^p\cM_f^{act}$, we obtain an $M_p:=M_{f\leq p}$-valued module, i.e. the evaluation map into $M$ lands in the sublevel set $M_p=\{x\in M:f(x)\leq p\}$. As a result, the maps $    | \cM_f| \to M^\nu$ and $| \cM_f^{act}| \to M^\nu$ respect the filtrations (the filtration on $M^\nu$ is induced by $f$). To prove that $    | \cM_f| \to M^\nu$ and $| \cM_f^{act}| \to M^\nu$ are equivalences, we only need to check the induced maps on the subquotients are equivalences. 
	
	If $f$ has no critical values in the interval $[p-\epsilon,p]$, then $F^p| \cM_f|/F^{p-\epsilon}| \cM_f|$, resp. $F^p| \cM_f^{act}|/F^{p-\epsilon}| \cM_f^{act}|$ is trivial by construction, and $F^pM^\nu/F^{p-\epsilon}M^\nu$ is trivial by standard Morse theory. 
	
	Assume $p$ is a critical value of $f$, $\epsilon$ is small, and without loss of generality only one $X_{i_0}$ is at level $p$ (this is a tautology for $\cM_f^{act}$). In particular, the flow category $F^p \cM_f/F^{p-\epsilon} \cM_f$, resp. $F^p \cM_f^{act}/F^{p-\epsilon} \cM_f^{act}$ is supported at index $0$. Let $F^p\fX_{q}$ denote the semi-simplicial set corresponding to $F^p \cM_f$, resp. $F^p \cM_f^{act}$. The map $\Sigma^{-a}F^p\fX_{a+1}[k]\to F^pM^\nu=M_p^\nu$ is constructed in the proof of \Cref{prop:relative-module} by first applying the boundary maps in any order to obtain $\Sigma^{-a}F^p\fX_{a+1}[k]\to\Sigma^{-a}F^p\fX_{a+1}[0]$ and then by composing with \eqref{eq:baserelmap}. As a result, if $k>0$, this map factors through $\Sigma^{-a} F^{p-\epsilon}\fX_{a+1}[k-1]$, and thus through $M^\nu_{p-\epsilon}$. Hence, the map $\Sigma^{-a}F^p\fX_{a+1}[k]\to M_p^\nu/M_{p-\epsilon}^\nu$ is $0$ for $k>0$ and only depends on $\Sigma^{-a}X(i_0)\wedge \cJ(i_0,-a-1)  \to M_{p}^\nu$.
	
	There is a commutative diagram
	\begin{equation}
		\xymatrix{   \Sigma^{-a}X(i_0)\wedge \partial\cJ(i_0,-a-1)  \ar[r]\ar@{^{(}->}[d]& \partial\cN(i_0)^\nu\ar[r]\ar@{^{(}->}[d]& M_{p-\epsilon}^\nu \ar@{^{(}->}[d]\\ 
			\Sigma^{-a}X(i_0)\wedge \cJ(i_0,-a-1)  \ar[r]& \cN(i_0)^\nu\ar[r]& M_{p}^\nu}
	\end{equation}
	inducing 
	\begin{equation}
		\Sigma^{-a}X(i_0)\wedge (\cJ(i_0,-a-1)/\partial\cJ(i_0,-a-1))\to \cN(i_0)^\nu/\partial \cN(i_0)^\nu\to M_{p}^\nu/M_{p-\epsilon}^\nu
	\end{equation}
	which coincides with the map induced on the subquotient. It is easy to check that the first map is an equivalence 
	and the second map is an equivalence by standard Morse theory, as $M_p$ contracts onto $\cN(i_0)$.

\end{proof}

\subsection{Compositions and homotopies}
Bimodules can be thought as morphisms between flow categories, and  it is natural to ask how to compose them. An infinity categorical version of this question, would be a discussion of conditions implying that the composition $c_{\cN_2}\circ c_{\cN_1}$ is homotopic to $c_{\cN_3}$, i.e. to describe $2$-cells of the corresponding infinity category (and to prove a weak Kan condition). This is beyond our goals: instead we will give a condition that implies $c_{\cN_2}\circ c_{\cN_1}\simeq c_{\cN_4}\circ c_{\cN_3}$. 

We use the intuition from \cite{audindamianmorse}. Consider the analogous Morse theoretic setting of four Morse functions $f_1$, $f_2$, $f_3$, $f_4$ with interpolations $f_1\to f_2$, $f_1\to f_3$, $f_2\to f_4$ and $f_3\to f_4$. We denote the interpolations by $\tilde f_{ij}$. One has a diagram
\begin{equation}
	\xymatrix{ CM(f_3)\ar[r]&CM(f_4)\\ CM(f_1)\ar[u]\ar[r]&CM(f_2)\ar[u]}
\end{equation}
To show that this diagram is commutative, one puts a Morse function on $I\times I$ with critical points only at $\{0,1\}\times \{0,1\}$ and with only negative gradient trajectories between adjacent corners at the edges (going in upper or right direction). See the arrows on \Cref{figure:morsesquare}. One can use the function $h(s,t)=s^3/3+t^3/3-s^2/2-t^2/2=g(s)+g(t)$ ($g$ is as above). 

\begin{figure}
	\centering
	\includegraphics[height=4cm]{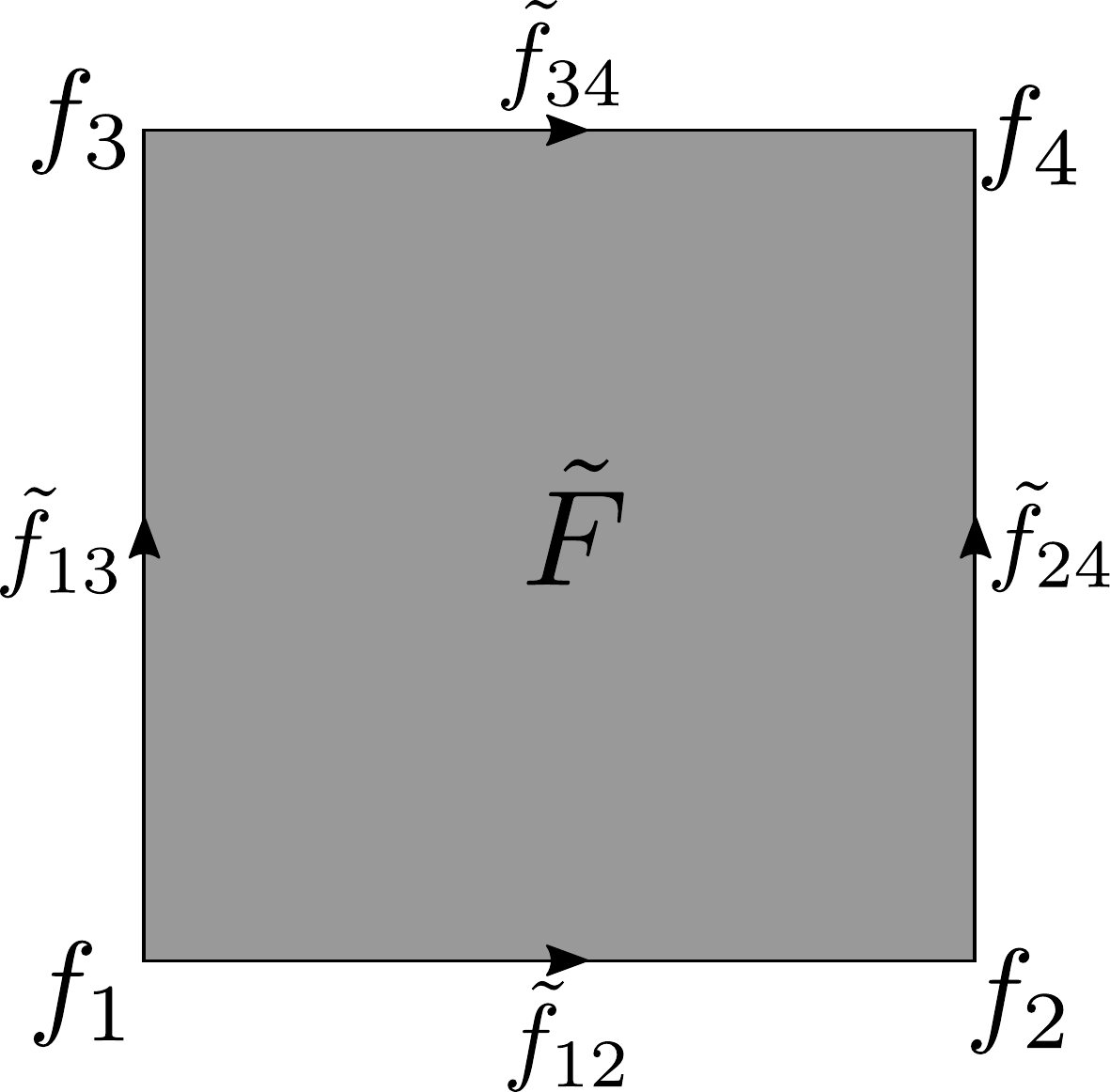}
	\caption{Interpolations of $f_1$, $f_2$, $f_3$ and $f_4$}
	\label{figure:morsesquare}
\end{figure}
Define a function $\tilde F:I\times I\times M\to\bR$ extending $f_i$ and the interpolations. In other words, place $f_1,f_2,f_3,f_4$ to the corners $(0,0), (1,0), (0,1), (1,1)$, in the respective orders. Put the chosen interpolations $\tilde f_{ij}$ to the edges between adjacent vertices corresponding to $f_i$. Now, we have a function $\partial(I\times I)\times M\to\bR$, and extend it smoothly to $\tilde F:I\times I\times M\to\bR$. See \Cref{figure:morsesquare}. The Morse complex $CM(\tilde F+h)$ fits into a diagram

\begin{equation}
\xymatrix{CM(f_4)\ar@{^{(}->}[r]\ar@{^{(}->}[d]& CM(\tilde f_{34}+g)\ar@{->>}[r]\ar@{^{(}->}[d]& CM(f_3)[-1]\ar@{^{(}->}[d] \\ 
CM(\tilde f_{24}+g)\ar@{^{(}->}[r]\ar@{->>}[d]& CM(\tilde F+h)\ar@{->>}[r]\ar@{->>}[d]& CM(\tilde f_{13})[-1]\ar@{->>}[d]\\
CM(f_2)[-1]\ar@{^{(}->}[r]& CM(\tilde f_{12}+g)[-1]\ar@{->>}[r]& CM(f_1)[-2]
} 	
\end{equation}
It can be checked explicitly on the generators that the diagram commutes. The rows and the columns are quotient sequences, and we need to show the two compositions of connecting maps $CM(f_1)\to CM(f_2)\to CM(f_4)$ and $CM(f_1)\to CM(f_3)\to CM(f_4)$ are homotopic. This is a simple diagram chasing argument. 

Motivated by this, we make the following definition
\begin{defn}\label{defn:homotopicbimodule}
Consider framed flow categories $\cM_1$, $\cM_2$, $\cM_3$, $\cM_4$, and framed (monotone) $\cM_i$-$\cM_j$-bimodules $\cN_{ij}$	for $(i,j)=(1,2),(1,3),(2,4),(3,4)$. See \eqref{eq:bimodulesqr}. Consider the category $Ob(\cM_1(2))\sqcup Ob(\cM_2(1))\sqcup Ob(\cM_3(1))\sqcup Ob(\cM_4)$, with the index function. One can extend this to a partial category by adding each $\cN_{ij}$ as the morphisms from $Ob(\cM_i)$ to $Ob(\cM_j)$, for $(i,j)=(1,2),(1,3),(2,4),(3,4)$. We call the pairs $(\cN_{12},\cN_{24})$ and $(\cN_{13},\cN_{34})$ \emph{homotopic} if this partial category extends fully faithfully to a framed flow category $\tilde\cP$ such that the only morphisms are from $Ob(\cM_i)$ to $Ob(\cM_j)$ for $(i,j)=(1,2),(1,3),(2,4),(3,4),(1,4)$ (hence, only morphisms added to the partial category are from $Ob(\cM_1)$ to $Ob(\cM_4)$). Moreover, we require the index bundle on $Ob(\cM_1(2))$, resp. $Ob(\cM_2(1))$, resp. $ Ob(\cM_3(1))$, resp. $Ob(\cM_4)$, to be that on $Ob(\cM_1)$, resp. $Ob(\cM_2)$, resp. $ Ob(\cM_3)$, resp. $Ob(\cM_4)$ plus $\underline{\bR}^2$, resp. $\underline{\bR}$, resp. $\underline{\bR}$, resp. $\underline{0}$, and the framing on $\tilde\cP$ restricts to those on $\cM_i$, $\cN_{ij}$. Such a $\tilde \cP$ is called \emph{an homotopy of the pairs $(\cN_{12},\cN_{24})$ and $(\cN_{13},\cN_{34})$}.
\end{defn}
One could give a definition in terms of the newly added morphisms from $\cM_1$ to $\cM_4$, which we denote by $\cO$. We find the current approach more practical. In this case, one can reconstruct $\tilde{P}$ from such $\cM_i$, $\cN_{ij}$, $\cO$. See \Cref{figure:flowcategorysquare}. 
\begin{equation}\label{eq:bimodulesqr}
	\xymatrix{ \cM_3\ar[r]^{\cN_{34}}&\cM_4\\ \cM_1\ar[u]^{\cN_{13}}\ar[r]^{\cN_{12}}&\cM_2\ar[u]_{\cN_{24}}}
\end{equation}

\begin{figure}
\centering
\includegraphics[height=4cm]{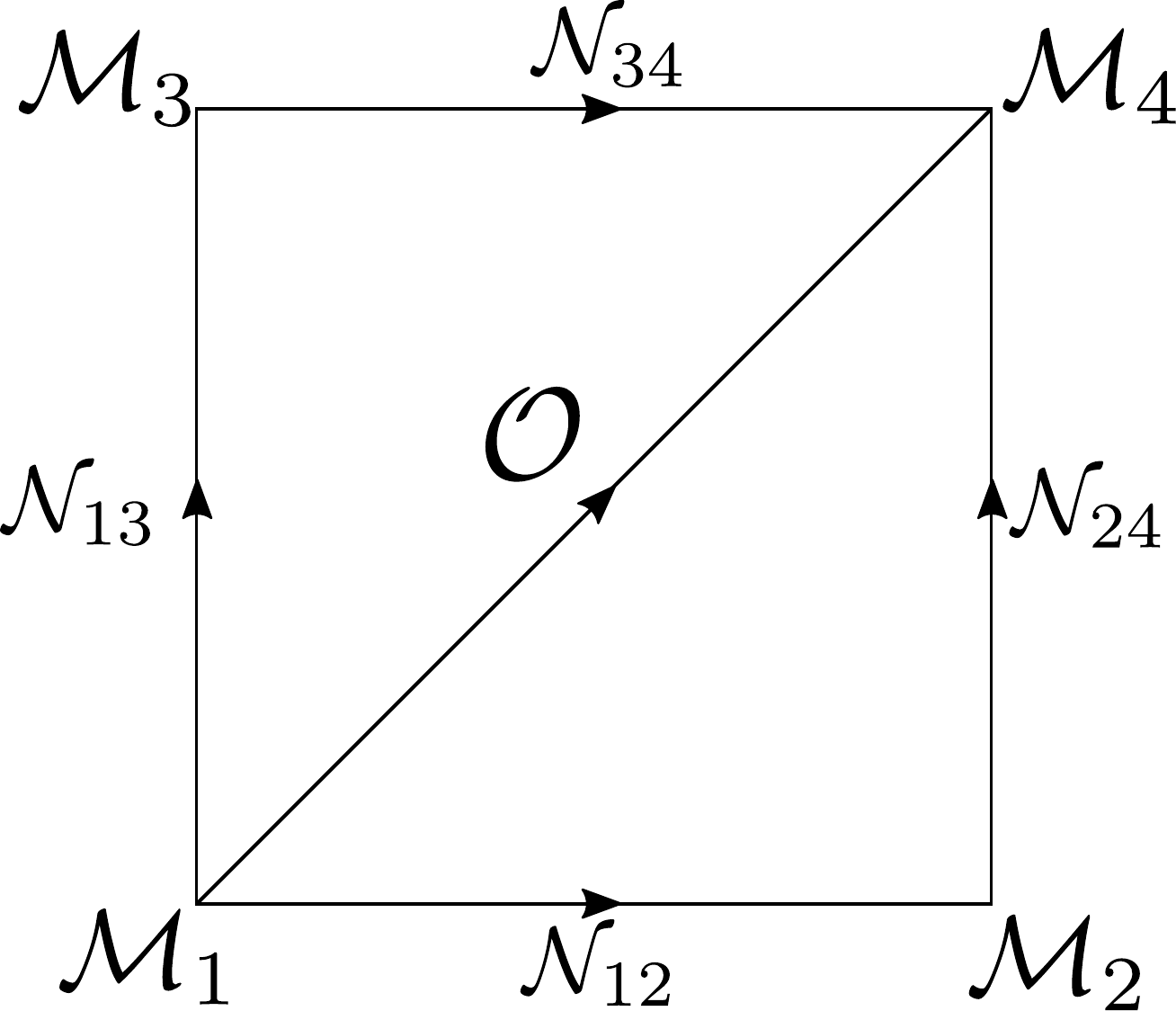}
\caption{Flow category $\tilde \cP$ reconstructed from $\cM_i$, $\cN_{ij}$ and $\cO$}
\label{figure:flowcategorysquare}
\end{figure}

Our goal is to show the following
\begin{prop}\label{prop:homotopyinduced}
If the pairs $(\cN_{12},\cN_{24})$ and $(\cN_{13},\cN_{34})$ are homotopic, then the induced maps $c_{\cN_{34}}\circ c_{\cN_{13}}$ and $c_{\cN_{24}}\circ c_{\cN_{12}}$ on geometric realizations are homotopic. 
\end{prop}
The proof is analogous to the Morse case. Let $\cP_{ij}$ denote the flow category associated to $\cN_{ij}$, and let $\cX_{\cM}$ denote the $\cJ$-module associated to a flow category $\cM$. We have the following diagram of $\cJ$-modules
\begin{equation}
	\xymatrix{\cX_{\cM_4}\ar@{^{(}->}[r]\ar@{^{(}->}[d]& \cX_{\cP_{34}} \ar@{->>}[r]\ar@{^{(}->}[d]& \cX_{\cM_3(1)}\ar@{^{(}->}[d] \\ 
		\cX_{\cP_{24}}\ar@{^{(}->}[r]\ar@{->>}[d]& \cX_{\tilde \cP}\ar@{->>}[r]\ar@{->>}[d]& \cX_{\cP_{13}(1)}\ar@{->>}[d]\\
		\cX_{\cM_2(1)}\ar@{^{(}->}[r]& \cX_{\cP_{12}(1)} \ar@{->>}[r]& \cX_{\cM_1(2)}
	} 	
\end{equation}
Each row and column is a quotient sequence of $\cJ$-modules, and it is clear that the diagram commutes. Therefore, we have the analogous diagram
\begin{equation}
	\xymatrix{|{\cM_4}| \ar@{^{(}->}[r]\ar@{^{(}->}[d]& |\cP_{34}| \ar@{->>}[r]\ar@{^{(}->}[d]& |{\cM_3(1)}|\ar@{^{(}->}[d] \\ 
		|{\cP_{24}}|\ar@{^{(}->}[r]\ar@{->>}[d]& |{\tilde \cP}|\ar@{->>}[r]\ar@{->>}[d]& |{\cP_{13}(1)}|\ar@{->>}[d]\\
		|{\cM_2(1)}|\ar@{^{(}->}[r]& |{\cP_{12}(1)}| \ar@{->>}[r]& |{\cM_1(2)}|
	} 	
\end{equation}
where the rows and columns are quotient sequences (indeed, the diagram can be obtained by applying $\Sigma^{\infty-m}$ to such a diagram in spaces). A quotient sequence induces a map from the quotient to the suspension of the first term. Composing $|{\cM_1(2)}|\to \Sigma|{\cM_2(1)}|\to \Sigma^2|{\cM_4}|$ and desuspending construct us $c_{\cN_{24}}\circ c_{\cN_{12}}$. Similarly, $|{\cM_1(2)}|\to \Sigma|{\cM_3(1)}|\to \Sigma^2|{\cM_4}|$ gives us $c_{\cN_{34}}\circ c_{\cN_{13}}$. The following elementary statement finishes the proof of \Cref{prop:homotopyinduced}:
\begin{lem}
Consider a diagram of spaces 	
\begin{equation}\label{eq:general9diag}
	\xymatrix{A_4 \ar@{^{(}->}[r]\ar@{^{(}->}[d]& A_{34} \ar@{->>}[r]\ar@{^{(}->}[d]& A_3\ar@{^{(}->}[d] \\ 
		A_{24}\ar@{^{(}->}[r]\ar@{->>}[d]& A_a\ar@{->>}[r]\ar@{->>}[d]& A_{13}\ar@{->>}[d]\\
		A_2\ar@{^{(}->}[r]& A_{12} \ar@{->>}[r]& A_1
	} 	
\end{equation}
such that each row and column are quotient sequences. One obtains two maps $A_1\to\Sigma^2A_4$, namely by composing the connecting maps $A_1\to\Sigma A_2\to\Sigma^2A_4$ and $A_1\to\Sigma A_3\to\Sigma^2A_4$. Then these two maps are homotopic.
\end{lem}
\begin{proof}
The proof is elementary topology. A quotient sequence $A\overset{\iota}{\hookrightarrow} B\twoheadrightarrow C$ extends functorially to 
\begin{equation}
A\hookrightarrow B\hookrightarrow cone(\iota)=cA\cup_\iota B\overset{\simeq}{\twoheadrightarrow} C	
\end{equation}
The connecting map is obtained by composing an inverse $C\to cA\cup_\iota B$ with the map $cA\cup_\iota B\to\Sigma A$ that collapses $B$. 

Applying this to \eqref{eq:general9diag}, we obtain the following diagram
\begin{equation}\label{eq:general16diag}
	\xymatrix{A_4 \ar@{^{(}->}[r]\ar@{^{(}->}[d]& A_{34} \ar@{^{(}->}[d]\ar@{^{(}->}[r]&cA_4\cup A_{34}\ar@{->>}[r]^\simeq\ar@{^{(}->}[d] & A_3\ar@{^{(}->}[d] \\ 
		A_{24}\ar@{^{(}->}[r]\ar@{^{(}->}[d]& A_a\ar@{^{(}->}[r]\ar@{^{(}->}[d]& cA_{24}\cup A_a\ar@{^{(}->}[d]\ar@{->>}[r]^\simeq & A_{13}\ar@{^{(}->}[d]\\
			cA_4\cup A_{24}	\ar@{^{(}->}[r]\ar@{->>}[d]^\simeq& cA_{34}\cup A_a \ar@{^{(}->}[r]\ar@{->>}[d]^\simeq & c^2A_4\cup cA_{24}\cup cA_{34}\cup A_a  \ar@{->>}[r]^{\qquad\;\;\simeq}\ar@{->>}[d]^\simeq & cA_3\cup A_{13} \ar@{->>}[d]^\simeq \\
		A_2\ar@{^{(}->}[r]& A_{12}\ar@{^{(}->}[r] & cA_2\cup A_{12} \ar@{->>}[r]^\simeq & A_1
	} 	
\end{equation}
Note that we have omitted the subscripts of $\cup$ from the notation, i.e. we have written $\cup$ instead of $\cup_\iota$. It is easy to check that this diagram commutes up to homotopy. 
To obtain the composition $A_1\to \Sigma A_2\to\Sigma^2 A_4$, one goes from $A_1$ to $cA_2\cup A_{12}$ via the homotopy inverse, collapses $A_{12}$, and repeats the same to go from $\Sigma A_2$ to $\Sigma^2A_4$. As $cA_2\cup A_{12}$ is the quotient of $c^2A_4\cup cA_{24}\cup cA_{34}\cup A_a$ by $c(cA_4\cup A_{34})=c^2A_4\cup cA_{34}$, for the first step (i.e. to go to $\Sigma A_2$), one can also go from $A_1$ to $c^2A_4\cup cA_{24}\cup cA_{34}\cup A_a$ via the homotopy inverses, then collapse $c^2A_4\cup cA_{34}$, then collapse the image of $A_{12}$. This is the same as collapsing $c^2A_4$, $cA_{34}$ and $A_a$ in $c^2A_4\cup cA_{24}\cup cA_{34}\cup A_a$ (due to the collapse of $c^2A_4$, one ends up in $\Sigma A_2$, not $cA_{24}$). To further go from $\Sigma A_2$ to $\Sigma^2 A_4$, one needs to first lift to $\Sigma (cA_4\cup A_{24})$. But instead, one can simply skip the collapse of $c^2A_4$ in the previous step, i.e. collapse $cA_{34}$ and $A_a$ in $c^2A_4\cup cA_{24}\cup cA_{34}\cup A_a$ only, and end up in $\Sigma (cA_4\cup A_{24})$. After this going to $\Sigma^2 A_4$ only requires collapsing the image of $cA_{24}$. 

In summary, the composition $A_1\to\Sigma A_2\to \Sigma^2A_4$ amounts to going from $A_1$ to $c^2A_4\cup cA_{24}\cup cA_{34}\cup A_a$ via the composed homotopy inverses (does not matter which way), then collapsing $cA_{24}$, $cA_{34}$ and $A_a$. The same holds for the composition of $A_1\to\Sigma A_3\to \Sigma^2A_4$, as well, by interchanging indices $2$ and $3$ (and $24$ and $34$). Therefore, the two compositions are homotopic. 
\end{proof}

\subsection{Extension to the $G$-equivariant case}\label{subsection:extension-equivariant}
In this section, we briefly explain the modifications needed in the equivariant case. Let $G$ be a compact Lie group. By a \emph{$G$-action on a smooth manifold with corners $M$} we always mean a smooth action. When $M$ carries the structure of a $\langle k\rangle$, resp. $\langle 2, k\rangle$-manifold, we also assume $\partial_aM$, resp. $\partial_{i,a}M$ is fixed by the $G$-action (this is automatic for connected $G$). A \emph{$G$-equivariant flow category $\cM$} is a flow category such that $G$-acts on $ob(\cM)$ and each $\cM(i,j)$ such that the gluing/composition maps are $G$-equivariant. A \emph{$G$-equivariant bimodule} and \emph{$G$-equivariant homotopy} is also defined analogously. 

To define framings, we recall basic facts from \cite{segalequivariantktheory} about (real) equivariant vector bundles and virtual bundles. Given a $G$-space $X$, let $Vect_G(X)$ denote the category of real $G$-equivariant vector bundles over $X$ (when $X$ and the action are smooth this is equivalent to the category of smooth bundles). 
For any representation $V$ of $G$, there is a $G$-equivariant vector bundle $\underline{V}$ over $X$ with total space $X\times V$. This defines a symmetric monoidal functor from $Rep_O(G)$ (the category of finite dimensional, real representations of $G$) to $Vect_G(X)$. Let $KO_G(X)$ denote the Grothendieck ring of $Vect_G(X)$. Elements of $KO_G(X)$ can be represented as formal differences $V-W$ for $V,W\in  Vect_G(X)$. Note that the following versions of the standard facts for non-equivariant bundles hold in the equivariant case as well:
\begin{enumerate}
	\item short exact sequences in $Vect_G(X)$ split (this can be shown by choosing a non-equivariant splitting for $V\hookrightarrow W$ and averaging over its orbit)
	\item\label{item:addtotriv} for any $V\in Vect_G(X)$, there is a $V'\in Vect_G(X)$, $W\in Rep_O(G)$, and an isomorphism $V\oplus V'\cong \underline{W}$ in $Vect_G(X)$ (see \cite[Prop (2.4)]{segalequivariantktheory})
	\item as a result, every element of $KO_G(X)$ can be represented as a difference $V-\underline{W}$, for $V'\in Vect_G(X)$, $W\in Rep_O(G)$
\end{enumerate}
We call a formal difference $V-W$ of \emph{$G$-equivariant vector bundles}. 
We consider them up to \emph{stable equivalence}, i.e. $V-W$ is the same as $V'-W'$ if and only if there exists a $U\in Vect_G(X)$ such that $V\oplus W'\oplus U\cong V'\oplus W\oplus U$. This condition implies that they are the same in $KO_G(X)$. By the claim (\ref{item:addtotriv}), any $V-W$ is also stably equivalent to some $V'-\underline{W}'$, where $V'\in Vect_G(X)$, $W'\in Rep_O(G)$. Claim (\ref{item:addtotriv}) also implies that $V-\underline{W}$ is stably equivalent to $V'-\underline{W}'$ is and only if $V\oplus \underline W'\oplus \underline U\cong V'\oplus \underline W\oplus \underline U$ for some $U\in Rep_O(G)$. Given a virtual bundle $V-\underline{W}$, define its Thom spectrum as $\Sigma_W^{-1}X^V$. This is a well-defined genuine $G$-equivariant spectrum, we denote it by $X^{V-\underline{W}}$ as before. 

Note that for most of the paper, we consider virtual bundles of the form $V-\underline{\bR}^n$, where $\underline{\bR}^n$ is the trivial representation. Also, when discussing Floer theory, we only consider the case $G=S^1$. 

With this preparation, define a \emph{$G$-equivariant framing} on a $G$-equivariant flow category $\cM$ to be a choice of virtual $G$-bundles $V_i$ on each $X_i$, and choice of stable equivalences \begin{equation}
	V_i\oplus T_{X_i}\cong T_{\cM(i,j)}\oplus \underline{\bR}\oplus V_j
\end{equation}
where $\underline{\bR}$ is trivial as before. We require the same compatibility conditions. The following is analogous to \Cref{proposition:j-modules}: 
\begin{prop}
There is a $\cJ$-module in genuine $G$-spectra with $X(i):=\Sigma^{-i}X_i^{V_i}$.
\end{prop}
The construction and the proof are the same. One needs to use an equivariant version of the neat embedding theorem, which can be proven by combining the proofs of \cite{laures} and Mostow-Palais theorem. See \Cref{appendix:k-manifolds} for details. 

The construction of the geometric realization goes through in the category of genuine equivariant spectra. Therefore we have 
\begin{prop}
There is a genuine $G$-spectrum $|\cM|$.
\end{prop}
The equivariant framings for bimodules and homotopies are defined similarly. The proofs of \Cref{prop:inducedmap} and \Cref{prop:homotopyinduced} hold in the equivariant setting too, giving us equivariant maps and homotopies. 

For a $G$-space $P$, a $P$-relative $G$-equivariant module $\cN$ over $\cM$ is defined as before, one additionally requires $\cN(i)$ and the structure maps to be $G$-equivariant. Given $G$-equivariant virtual bundle $\nu$ on $P$, one can define similar framings, to obtain a $G$-equivariant map $|\cM|\to P^\nu$ as in \Cref{prop:relative-module}.

\section{The stable equivariant homotopy type of a manifold via Morse--Bott theory}\label{sec:equivariantmorse}

Let $M$ be a smooth, oriented, compact manifold and $G$ be a compact Lie group acting on $M$. The purpose of this section is to construct $\Sigma^+_\infty(EG \times M)$ as a genuine $G$-spectrum using Morse(--Bott) theory. In particular, we also obtain a Morse-theoretic model for $M$ as a Borel $G$-spectrum. 

Typically, it is hard to find $G$-equivariant Morse functions on $M$, and even if they exist, finding Morse--Smale pairs is harder. As a result, we will apply the Borel construction to construct the equivariant stable homotopy type. In other words, we will construct ``$EG\times M$'' using Morse theory.

We start by fixing a model for $EG$ having certain desirable properties. For the remainder of this section, $EG$ denotes a fixed contractible topological space endowed with a free $G$-action. We assume that there is an exhaustion $EG= \bigcup_{n \geq 1} E_nG$ where the $E_nG$ are smooth compact manifolds on which $G$ acts freely, and $E_nG \hookrightarrow E_{n+1}G$ is a smooth $G$-equivariant inclusion. As a result, $BG:=EG/G$ is filtered by $B_nG:=E_nG/G$. We also require there exists a Morse--Smale pair $(f,g)$ on the ind-manifold $BG$. In other words, 
\begin{enumerate}
	\item $f:BG\to\bR$ is a function such that $f|_{B_nG}$ is smooth and Morse for all $n$. 
	\item $g$ is a metric on $BG$ (i.e.\ a metric on each $B_nG$ that is compatible with the inclusions) such that $(f|_{B_nG},g|_{B_nG})$ is Morse--Smale for all $n$.
	\item the gradient vector field of $f|_{B_nG}$ and $f|_{B_{n+1}G}$ agree on ${B_nG}$ (i.e.\ $df$ vanishes in normal directions ${B_{n}G}\subset {B_{n+1}G}$).
	\item the critical points satisfy $crit(f|_{B_nG})\subset crit(f|_{B_{n+1}G})$ (which follows from the previous condition) and the indices of critical points are stable (i.e. the Hessian is positive in normal directions ${B_{n}G}\subset {B_{n+1}G}$).
\end{enumerate}
The fourth condition implies that the only negative gradient trajectories from a critical point in $B_nG$ are to these in $B_nG$, i.e.\ the negative gradient flow decreases the filtration level $n$. These conditions essentially guarantee that the Morse theory of $f|_{B_nG}$ is embedded into that for $f|_{B_{n+1}G}$. More precisely, the corresponding flow categories are nested by inclusions in the sense of \Cref{defn:inclusion} (indeed, in this case, each $B_nG$ can be equipped with framings compatible with the inclusions). 

The pullback $\tilde f:EG\to \bR$ is automatically Morse--Bott. Moreover, by fixing a connection on the principal $G$-bundle $EG\to BG$ (i.e.\ a nested sequence of equivariant, horizontal splittings for each $T_{E_nG}$), one can guarantee that $g$ lifts to a $G$-equivariant metric on $EG$ whose restriction to each $E_nG$ is Morse--Bott--Smale (cf.\ also \cite[Ex.\ 2.5]{latschev2000gradient}). 

The following example is standard, and will be used in \Cref{subsubsection:equivariant-SH-construction} for constructing $S^1$-equivariant closed-string Floer homology. 
\begin{exmp}\label{exmp:circlemorsebott}
Let $G=S^1$, and $EG=S^\infty$. Let $(z_1,z_2,\dots)$ denote the complex coordinates. Define $\tilde f:=|z_1|^2+2|z_2|^2+3|z_3|^2+\dots$. This descends to a Morse function on $\bC\bP^\infty$ and one can construct a metric satisfying the desired conditions inductively. 
For each $i$, the circle given by the set of points such that $z_j=0$ if $j\neq i$ is the critical manifold of index $2(i-1)$, and there is no critical point of odd index. 
\end{exmp}
\begin{exmp}\label{exmp:quoternion1morse}
In the same spirit, let $G=SU(2)=S^3$ and let $EG=S^\infty\subset\bH^\infty$. Let $(w_1,w_2\dots)$ denote the quaternionic coordinates and define $\tilde f=|w_1|^2+2|w_2|^2+3|w_3|^2+\dots$. The critical manifolds of $\tilde f$ are similar, and they are copies of $G=S^3$. The corresponding indices are $4(i-1)$. 
\end{exmp}
\begin{exmp}\label{exmp:unitarymorse}
Let $G=U(k)$ and $EG=V(k,\bC^\infty)$, the space of orthonormal frames. The quotient is the infinite Grassmanian $Gr(k,\bC^\infty)$, and one can describe various Morse functions on it. The following is taken from \cite{morsegrassmaniannotesmgualt}: a dimension $k$ subspace of $\bC^\infty$ can be realized as a rank $k$ projection operator $P:\bC^\infty\to \bC^\infty$ satisfying $P^2=P$, $P^*=P$, and as an infinite matrix, $P$ has finitely many non-vanishing coefficients. If $(v_1,\dots, v_k)$ is an orthonormal frame of the corresponding subspace, $P=\sum_{i=1}^{k}v_iv_i^*$. Let $Z:\bC^\infty\to \bC^\infty$ be given by $diag(1,2,3,4,\dots)$ (or any positive increasing sequence) and define $f(P)=tr(PZ)$. Explicitly, if $v_i$ has coefficients $z_{ij}$, $f(P)=\sum_{i=1}^{k}\sum_j j|z_{ij}|^2$. The critical points of the function on $Gr(k,\bC^\infty)$ are given by subspaces spanned by $k$ of the standard basis elements $e_i$, and the index of such a critical point spanned by $(e_{i_1},\dots,e_{i_k})$ is given by $(i_1-1)+(i_2-2)+(i_3-3)+\dots$, and it is the dimension of the corresponding Schubert cell. 
See \cite{morsegrassmaniannotesmgualt}. One can inductively construct an appropriate Morse--Smale metric on $Gr(k,\bC^\infty)$, and a lift to a pair on $V(k,\bC^\infty)$.
\end{exmp}
For an arbitrary compact group $G$, one can pick an embedding $G\hookrightarrow U(k)$, and inductively construct a Morse function on $V(k,\bC^\infty)/G$ and a metric, satisfying the desired properties.
\begin{rk}
One can presumably combine the ideas in \Cref{exmp:quoternion1morse} and \Cref{exmp:unitarymorse} to describe explicit Morse functions for $G=Sp(k)$.
\end{rk}
\begin{note}
The finite Grassmanians can be obtained by Hamiltonian reduction of a faithful representation of $U(k)$ on $\bC^{k+n}$. This could possibly be used to give a more explicit description of Morse--Bott functions on $EG$ of desired type. Namely, by picking a directed system of unitary (hence, symplectic) representations $\{V_i\}_{i\in I}$ and inclusions, and compatible Hamiltonians $V_i\to Lie(G)^*$, one can possibly construct a model for $BG$ as the colimit of reduction, and use the Morse theory on symplectic reduction to endow $BG$ with such a Morse function. 
\end{note}
Assume we are given such $\tilde f:EG\to\bR$ with a suitable metric. We will use this to describe the equivariant stable homotopy type of a compact manifold via Morse theory: let $M$ be a closed manifold with $G$-action. Choose a smoothly varying family $(h_s,g_s)$, $s\in EG$ of functions $h_s:M\to\bR$ and $g_s$ a metric on $M$. We require that
\begin{enumerate}
	\item $(h_s,g_s)$ is fixed under the $G$-action, i.e. $g_*(h_s,g_s)=(h_{gs},g_{gs})$
	\item $(h_s,g_s)$ is Morse--Smale when $s\in crit(\tilde f)$ and with respect to a $G$-equivariant tubular neighborhood of a critical set $X\cong G$, it depends only on the $X$-coordinate
	\item given $s_0,s_1\in crit(\tilde f)$, $x_0\in crit(h_{s_0})$, $x_1\in crit(h_{s_1})$, the moduli space of $-grad_{g_s}(h_s)$ trajectories from $x_0$ to $x_1$ is cut out transversally
\end{enumerate}
One can also assume $h_s$ is $C^1$-small, so that $\tilde f+\{h_s\}:EG\times M\to \bR$ is Morse--Bott. Instead, we formally imitate Morse theory on the product, by considering pairs of trajectories.

Define a Morse--Bott (ind-)flow category $\cM^G_{\tilde f, h_s}$ whose objects are given by pairs $(s,x)$, where $s\in crit(\tilde f)$, $x\in crit(h_s)$. The objects are $G$-torsors. The morphisms are given by pairs $(s(t),\gamma(t))$, where $s(t)$ is a negative gradient trajectory of $\tilde f$, and $\gamma(t)$ is a solution to $\frac{d}{dt}\gamma=-grad_{g_{s(t)}}(h_{s(t)})$. We will will explain in \Cref{sec:gluing} how to construct smooth charts making $\cM^G_{\tilde f, h_s}$ into a $G$-equivariant Morse--Bott flow category by the assumptions above.  One can endow it with $G$-equivariant framings analogous to \Cref{exmp:morsebottfunction} and \Cref{example:flow-cat-act}. Therefore, we have a genuine $G$-equivariant spectrum $|\cM^G_{\tilde f, h_s}|$. 

Our next goal is to prove the following equivariant analogue of the Cohen--Jones--Segal theorem:
\begin{thm}\label{thm:equivariantcjs}
$|\cM^G_{\tilde f, h_s}|$ is homotopy equivalent to $\Sigma^\infty(EG\times M)_+$ as a genuine $G$-spectrum (therefore, homotopy equivalent to $\Sigma^\infty_+M$ as a ``homotopy/Borel'' $G$-spectrum). 
\end{thm}
As a warm-up, we prove a similar statement in the presence of $G$-equivariant Morse--Bott--Smale pairs. Notice that, if $N$ is a $G$-equivariant closed manifold, where $G$ acts freely, such pairs can be obtained by considering Morse--Smale pairs $(f,g)$ on $N/G$, pulling back $f$, and lifting $g$. More precisely, choose a connection on the principal bundle $N\to N/G$ to lift $g$ to the horizontal tangent bundle, and extend orthogonally to the entire $T_N$ using a biinvariant metric on $G$. 
In the presence of a $G$-equivariant Morse--Bott--Smale pair on $N$, \emph{the standard framed $N$-relative module} that associates every object its descending manifold is also $G$-equivariant. We prove:
\begin{prop}\label{prop:mbscjs}
Let $N$ be a $G$-equivariant closed manifold equipped with a $G$-equivariant Morse--Bott--Smale pair. Then the standard $N$-relative module on the corresponding flow category induces a $G$-equivariant homotopy equivalence of genuine equivariant spectra.	
\end{prop}
Note that the subtlety here is in establishing a genuine $G$-equivariant homotopy equivalence: a $G$-equivariant map that is also a homotopy equivalence is not necessarily a genuine $G$-equivariant homotopy equivalence: the geometric fixed points may be different.
\begin{proof}
This is almost identical to the proof of \Cref{thm:morsebottrecovers}. Indeed, the relative module constructed there is $G$-equivariant. Since the filtrations are compatible with the $G$-action and the induced maps on the subquotients are equivalences of genuine $G$-spectra, the conclusion follows. 
\end{proof}
By using similar ideas, we show:
\begin{proof}[Proof of \Cref{thm:equivariantcjs}]
Let $\cN$ be the equivariant $EG\times M$-relative module that associates to an object $(s,x)$ its descending manifold. 
Explicitly, $\mathcal{N}(s,x)$ is the compactification of the moduli of pairs $(\beta,\gamma)$, where:
\begin{itemize}
    \item $\beta: (-\infty, 0] \to EG$ is a half-gradient trajectory of $-\tilde f$, i.e. it satisfies $ \frac{d}{dt}\beta+grad \tilde{f}( \beta)=0$ and $\lim_{t \to -\infty} \beta(t)= s$ 
    \item $\gamma: -(-\infty, 0] \to M$  is a ``half-gradient trajectory relative to the base'', i.e. it satisfies $ \frac{d}{dt} \gamma+ grad_{g_{\tilde{s}}}(h_{\tilde{s}})=0$ and $\lim_{t \to -\infty} \gamma(t)= x$ 
\end{itemize}
The evaluation map to $EG\times M$ is given by sending such a trajectory to the end-point $(\beta(0),\gamma(0))$. 

It is easy to see that $\cN$ is $G$-equivariant, and it can be equivariantly framed analogously to \Cref{subsec:applcjsmorse}, with $\nu$ given by $0$. 
%
\Cref{prop:relative-module} thus produces a $G$-equivariant map \begin{equation}\label{eq:inducedeqmap}
	|\cM^G_{\tilde f, h_s}|\to \Sigma^\infty(EG\times M)_+
\end{equation}
More precisely, (the restriction of) $\cN$ induces maps $|\cM^G_{\tilde f,n, h_s}|\to \Sigma^\infty(E_nG\times M)_+$, where $\cM^G_{\tilde f,n, h_s}$ is the flow category obtained by restricting $\tilde f$ to $E_nG$, and \eqref{eq:inducedeqmap} is produced by taking colimit $n\to\infty$. Therefore, it suffices to prove the claim for each $n$. 
%

The proof is similar to that of \Cref{thm:morsebottrecovers}. Filter both $|\cM^G_{\tilde f,n, h_s}|$ and $\Sigma^\infty(EG\times M)_+$ by the value of $\tilde f$. Let the $G$-torsor $X$ be the critical set of $\tilde f$ at level $p$. The subquotients of $|\cM^G_{\tilde f,n, h_s}|$ do not see the non-constant trajectories of $\tilde f$, and are homotopy equivalent to the $G$-spectrum $X^{T^dX}\wedge |\cM_{h_{s_0}}|$. 
Here $s_0\in X$, and recall $h_s$, $s\in X$ can be obtained by applying $G$-action to $h_{s_0}$ due to equivariance. The subquotients on the right hand side are the same as $X^{T^dX}\wedge \Sigma^\infty M_+$ and the $G$-action comes from $X^{T^dX}$ alone. $|\cM_{h_{s_0}}|\to  \Sigma^\infty M_+$ at $s_0$ is an equivalence by \cite{cohen1995floer} or by \Cref{thm:morsebottrecovers}, which implies that $G$-equivariant map $X^{T^dX}\wedge |\cM_{h_{s_0}}|\to X^{T^dX}\wedge \Sigma^\infty M_+$ is a genuine equivalence. This finishes the proof. 
%
%
\end{proof}

One can prove Morse--Bott and twisted versions of \Cref{thm:equivariantcjs} similar to \Cref{thm:morsebottrecovers}. The proof is a combination of these two statements. Namely, in the definition of $\cM_{\tilde f,h_s}^G$, we can allow $(h_s,g_s)$ to be Morse--Bott--Smale for $s\in crit(\tilde f)$. Even this is stronger than what we need: we can allow $h_s$ to be Morse--Bott, and instead of requiring the pair to be Morse--Bott, we only assume $g_s$ is generic so that transversality holds. Given $s\in crit(\tilde f)$ and $x\in crit(h_s)$, index $(s,x)$ by the value of $\tilde f+\{h_s\}$ as in \Cref{example:flow-cat-act}. Call the resulting flow category $\cM_{\tilde f,h_s}^{G,act}$. Frame it analogous to $\cM_{\tilde f,h_s}^G$, except add the $G$-equivariant virtual bundle $\nu$ over $M$ to each index bundle. In this case we have:
\begingroup
\def\thethm{\ref*{thm:equivariantcjs}$'$}
\begin{thm}\label{thm:equivariantcjsalternative}
$|\cM^{G,act}_{\tilde f, h_s}|$ is homotopy equivalent to $(EG\times M)^\nu\simeq EG_+\wedge M^\nu$ as a genuine $G$-spectrum (therefore, homotopy equivalent to $M^\nu$ as a ``homotopy/Borel'' $G$-spectrum). 
\end{thm}
\addtocounter{thm}{-1}
\endgroup
As an application of this, we can prove an equivariant Morse--Bott inequality. List the indices of the critical points of $\tilde f$ in increasing order $j_0,j_1,\dots$ (numbers can be repeated). Consider a sequence of Morse--Bott functions $h_0,h_1,\dots: M \to \mathbb{R}$ and let $X_{i,j_k} \subset M$ denote the index $i$ critical set of $h_k$ and $V_{i,j_k}$ denote the corresponding index bundle over $X_{i,j_k}$. Then we have
\begin{cor}\label{cor:eqmbineq}
Let $R$ be an extraordinary cohomology theory as in \Cref{corollary:betti-bound}, and further assume it is bounded below. Then,  $rk(R^G_*(M^\nu))\leq \sum\limits_{i,k} rk(R_{*-j_k}(X_{i, j_k}^{V_{i, j_k}+\nu}))$. In particular, when each $V_i$ is $R$-orientable
\begin{equation}\label{eq:eqmbineq}
rk(R^G_*(M_+))\leq \sum_{i,k} rk(R_{*-j_k-rk(V_{i,j_k})}(X_{i, j_k}))	
\end{equation}
\end{cor}
Recall, $R^G_*(M^\nu)$ is the equivariant homology, defined by $R_*(EG_+\wedge_G M^\nu)$, i.e. $R$ applied to the homotopy quotient. The extra bounded below condition is to ensure that \eqref{eq:eqmbineq} is a finite sum for a fixed degree $*$. 
\begin{proof}
Without loss of generality, one can increase each $h_k$ by a constant to ensure that whenever $k<k'$, $max(h_k)<min(h_{k'})$. One can first extend $h_k$ to an equivariant family of Morse--Bott functions parametrized by $crit(\tilde f)$, then extend to an $\{h_s\}$ as above. As a result, we obtain a flow category $\cM^{G,act}_{\tilde f, h_s}$, whose critical sets are isomorphic to $X_{i, j_k}\times G$	as $G$-torsors, and the corresponding framing bundle on it is given by $V_{i, j_k}+\nu+ \bR^{j_k}$. One can filter $\cM^{G,act}_{\tilde f, h_s}$ equivariantly with subquotients given by $X_{i, j_k}^{V_{i, j_k}+\nu}\wedge G^{\bR^{j_k}}=\Sigma^{j_k}X_{i, j_k}^{V_{i, j_k}}\wedge G_+$, and the homotopy quotient of this spectrum is given by $\Sigma^{j_k}X_{i, j_k}^{V_{i, j_k}}$. Taking homotopy quotients is an exact functor; hence, one can conclude as in the proof of \Cref{corollary:betti-bound}.
\end{proof} 
One can fix a nice model for $EG$ or let $h_0=h_1=\dots$ to obtain simpler statements. Denote the critical sets of $h_0$ by $X_i$ and the index bundle by $V_i$. Let $G=S^1$. Then we have 
\begin{cor}\label{cor:circleeqmbineq}
In the setting of \Cref{cor:mbineqs}
\begin{equation}
rk(R^{S^1}_*(M_+))\leq \sum_{k\geq 0,i} rk(R_{*-2k-rk(V_i)}(X_i))		
\end{equation}
\end{cor}
\begin{rk}
That the left side of \Cref{cor:eqmbineq} and \Cref{cor:circleeqmbineq} has equivariant cohomology and the right the ordinary cohomology can be seen as an instance of Atiyah--Bott localization. 
\end{rk}













\section{Genuine $S^1$-equivariant symplectic cohomology over the sphere spectrum}
Given a Liouville manifold $M$ satisfying appropriate topological assumptions, it is now possible thanks to work of Large \cite{largethesis} to define symplectic cohomology $SH(M; \mathbb{S})$ with coefficients in the sphere spectrum. This invariant expected to carry an ``$S^1$-action'' corresponding to loop rotation. The purpose of this section is to construct said action. 

In \cite{ganatracyclic}, Ganatra constructed a Borel $S^1$-action on symplectic cohomology $SH(M; k)$ with coefficients in an ordinary commutative ring $k$.\footnote{Recall from \Cref{section:equivariant-homotopy} that a genuine (resp.\ Borel) $G$-action on a spectrum $X$ is a lift of $X$ to $G \operatorname{Sp}$ (resp.\ to $\operatorname{Sp}^{BG}$) along the forgetful functors $G \operatorname{Sp} \to \operatorname{Sp}^{BG} \to \operatorname{Sp}$.} Presumably the same moduli spaces can be used to also endow $SH(M,\bS)$ with a Borel $S^1$-action. However, we will not take this route. 

Instead, we will construct a \emph{genuine} $S^1$-action on $SH(M; \mathbb{S})$ by using the Morse--Bott formalism developed earlier in the paper. More precisely, to a Liouville manifold $M$, we will associate an object $SH_S(M, \bS) \in S^1\operatorname{Sp}$ whose image under $S^1\operatorname{Sp} \to \operatorname{Sp}$ is $SH(M; \mathbb{S})$. Our construction is based on a Floer theoretic version of the ideas in \Cref{sec:equivariantmorse}, that were already used by Seidel and Bourgeois--Oancea to construct $S^1$-equivariant (ordinary) symplectic cohomology; see for instance \cite{bourgeois2009symplectic}.

For the construction of the $S^1$-equivariant spectral symplectic cohomology, we will rely on the general setup of Bourgeois--Oancea \cite[Sec.\ 2.2]{bourgeois2009symplectic} and will generally follow their notation. Morally, we will apply the construction of \Cref{sec:equivariantmorse} to the loop space of $M$, where the family of Morse--Smale pairs is replaced by a family of Floer data parametrized by $ES^1$. 

%
\subsection{Reminders on spectral symplectic cohomology}\label{subsubsec:reminderslarge}
In this section, we recall the construction of $SH(M,\bS)$, without the equivariance. We will follow \cite{largethesis}.

Let $M^{2n}$ be a Liouville manifold. Assume $M$ is polarizable, i.e. $TM$ is stably trivial as a symplectic bundle. Fix a stable symplectic trivialization of $TM$, i.e. an isomorphism of the symplectic bundles $TM\oplus \underline\bC^k\cong \underline\bC^{n+k}$. This data is called \emph{a background framing} by \cite{largethesis}. 

Let $H:S^1\times M\to \bR$ be a non-degenerate Hamiltonian, and assume it is linear at infinity with positive slope. Then, Large proves: 
\begin{thm}[{\cite[Theorem 6.12,\S8]{largethesis}}] Let $J:S^1\to \cJ(M)$ be a generic family of almost complex structures that is cylindrical at infinity. Then there exists a framed flow category $\cM_{H,J}$ such that
\begin{itemize}
		\item $ob(\cM_{H,J})$ is the (finite, discrete) set of time-$1$ orbits of $H$
		\item the index of $x\in ob(\cM_{H,J})$ is the Conley--Zehnder index
		\item for $x,y\in ob(\cM_{H,J})$, $\cM_{H,J}(x,y)$ is the set of (broken) Floer trajectories from $x$ to $y$
	\end{itemize}
\end{thm}
The major input of \cite{largethesis} is the smooth gluing: the compactified moduli spaces $\cM_{H,J}(x,y)$ of Floer trajectories form manifolds with corners, and the boundary strata are given by the set of broken trajectories. Large's proof of smooth gluing relies crucially on estimates of Fukaya--Oh--Ohta-Ono \cite{foooexponential}. We expect that the methods of \cite[Theorem 1.1]{baixuarnold} or \cite{rezchikovarnold} could be used to give an alternative proof.

Note that we choose our conventions so that $x$ is the ``input'' and $y$ is the ``output''. Hence $x$ corresponds to the positive end of the cylinder and $y$ to the negative end, a cylinder from $x$ to $y$ really means a map $\bR\times S^1\to M$ that is asymptotic to $x$, resp. $y$, as $s\to+\infty$, resp. $s\to -\infty$. We swapped the notational conventions of \cite{largethesis} to fit into our framework.

Next, we recall the construction of framings from \cite{largethesis}. As $TM\oplus \bC^k$ is trivialized, one can realize the Cauchy--Riemann operator on the cylinder $\bR\times S^1$ as a Fredholm operator $W_1^2(\bR\times S^1, \bC^{n+k})\to L^2(\bR\times S^1, \Omega_{\bR\times S^1}^{0,1}\otimes \bC^{n+k})$. More precisely, the choice of Floer data gives rise to a Fredholm operator $W_1^2(\bR\times S^1, u^*TM)\to L^2(\bR\times S^1, \Omega_{\bR\times S^1}^{0,1}\otimes u^*TM)$, which can be extended to $W_1^2(\bR\times S^1, u^*TM\oplus \bC^k)\to L^2(\bR\times S^1, \Omega_{\bR\times S^1}^{0,1}\otimes u^*TM\oplus \bC^k)$ without changing the index. To extend, one takes the direct sum of this operator $k$ times with a fixed index zero Cauchy--Riemann operator 
\begin{align} \label{equation:CR-stabilized}
    W_1^2(\bR\times S^1,\bC)&\to L^2(\bR\times S^1, \Omega_{\bR\times S^1}^{0,1}\otimes  \bC) \\
    X &\mapsto \partial_sX+i(\partial_t-B)X
\end{align}
where $B$ does not depend on $s,t$. Note that the index of \eqref{equation:CR-stabilized} is zero independently of the choice of $B$ \cite[Thm.\ 8.8.1]{audindamianmorse}; later it will be mildly convenient to add an extra assumption on $B$, see \Cref{note:stabilizationchange}. 

After this stabilization, we obtain a Cauchy--Riemann operator $W_1^2(\bR\times S^1, \bC^{n+k})\to L^2(\bR\times S^1, \Omega_{\bR\times S^1}^{0,1}\otimes \bC^{n+k})$ thanks to the trivialization of $TM\oplus \bC^k$. Observe that this operator is the Cauchy--Riemann operator corresponding to a pair $(J_\ell,Y_\ell)$, where $J_\ell$ is a domain dependent almost complex structure on $\bC^{n+k}$ and $Y_\ell\in \Omega_{\bR\times S^1}^{0,1} \otimes \operatorname{End}( \bC^{n+k})$.


One can compute the index of this problem by capping off the cylinder at $y$. The caps allow \cite{largethesis} to define the framings as well. More precisely, let $S_-:= \bC=\bP^1\setminus \{\infty\}$. Recall that, a \emph{tubular end} is a conformal map $\epsilon: [0,\infty) \times S^1 \to S_-$ such that $[0,\infty)  \times \{1\}$ is tangent to a fixed line in $T_0\bP^1$. 
\begin{defn}[{\cite[Defn 8.3]{largethesis},\cite[\S11.3]{abouzaidblumberg}}]\label{defn:abstractcap}
Given orbit $x\in ob(\cM_{H,J})$, an \emph{abstract cap} at $x$ is a tuple $\{(L, J, Y, g)\}$ where
\begin{itemize}
	\item $\epsilon$ is a tubular end
	\item $L>0$ is a real number
	\item $(J,Y)$ is a Floer datum on $S_-\times \bC^{n+k}$, which equals $(J_\ell,Y_\ell)$ on $\epsilon([L,\infty) \times S^1)$
	\item $g$ in a metric on $S_-$ which agrees with the standard metric on $\epsilon([L,\infty) \times S^1)$ 
\end{itemize}	
The space of abstract caps is denoted by $\cU(x)$, and it is contractible.
\end{defn}
Each abstract cap determines a Fredholm problem; therefore, there is a continuous map $\cU(x)\to Fred(W_1^2(S_-, \bC^{n+k}), L^2(S_-, \Omega_{S_-}^{0,1}\otimes \bC^{n+k})) $. The latter represents the $0^{th}$ space of the real $K$-theory, and comes equipped with a virtual bundle, \emph{the index bundle}. Therefore, one has a natural virtual bundle on $\cU(x)$, denoted by $V_x$. As $\cU(x)\to \{x\}$ is a homotopy equivalence, $V_x$ gives rise to a natural bundle (vector space) on $\{x\}$ of the same rank.  

A cylinder $u\in \cM_{H,J}(x,y)$ and an abstract cap at $y$ can be glued to give an abstract cap at $x$. In other words, there exists a map $G^{ab}_{x,y}:\cM_{H,J}(x,y)\times \cU(y)\to\cU(x)$. The index bundles are additive under gluing, and there is a natural isomorphism 
\begin{equation}\label{eq:gluingindexbundles}
(G^{ab}_{x,y})^*V_x\cong T_{\cM_{H,J}(x,y)}\oplus\underline\bR\oplus V_y
\end{equation}
This defines a framing on the flow category $\cM_{H,J}$. For more details, see \cite{largethesis}. 
\begin{note}\label{note:stabilizationchange}
One can change the background framing $TM\oplus\underline{\bC}^k\cong \underline\bC^{n+k}$ by further stabilizing. When doing so, one must correspondingly stabilize the abstract caps. More precisely, one adds multiples of a fixed Cauchy--Riemann operator $\partial_s+i(\partial_t-B)$. The matrix $B$ should have the same asymptotics as the matrix which occurs in \eqref{equation:CR-stabilized}, so that stabilization is compatible with the gluing. By choosing $B$ appropriately, we may assume that the index of the stabilized Cauchy--Riemann problem on the cap is also $0$. An alternative is to replace $V_x$ by $V_x-\bR^{2k\mu_0}$, where $\mu_0$ is the index of the stabilizing problem $\partial_s+i(\partial_t-B)$. Note that a background framing automatically determines a grading: a trivialization of $TM$ over the corresponding orbit. Our construction ensures that the index of $V_x$ coincides with the index determined by the grading. 

\end{note}
In the subsequent sections, we will extend this to the equivariant setting. But before moving on, we describe an alternative (and equivalent) way of framing the moduli spaces, that will become useful later. 

Let $S_+:=\bP^1\setminus\{0\}$, and define a \emph{negative cap} to be the same as \Cref{defn:abstractcap}, except that $S_-$ is replaced by $S_+$ and the tubular end is replaced by a negative tubular end $\epsilon: (-\infty,0]\times S^1\to S_+$. Denote the space of negative caps by $\cU^-(x)$, which is still contractible. As before, there is a continuous map $\cU^-(x)\to Fred(W_1^2(S_+, \bC^{n+k}), L^2(S_+, \Omega_{S_+}^{0,1}\otimes \bC^{n+k})) $, defining a virtual bundle on $\cU^-(x)$. Denote this bundle by $V_x^-$. Notice that a negative cap can be glued to a Floer trajectory at the input, to give a negative cap at the output. Therefore, similar to before, we have an identification
\begin{equation}\label{eq:negframe}
V_x^-\oplus T_{\cM_{H,J}(x,y)}\oplus \bR\cong V_y^-	
\end{equation}
Note the slight abuse of notation. One can check the compatibility of \eqref{eq:negframe} with the gluing Floer trajectories as in \cite[\S8]{largethesis} and $V_x':=\underline{\bR}^{2(n+k)}-V_x^-$ defines a stable framing. In particular, 
\begin{equation}\label{eq:alternativeframe}
	V_x'\cong T_{\cM_{H,J}(x,y)}\oplus \bR \oplus V_y'	
\end{equation}
\begin{rk}
The same issue in \Cref{note:stabilizationchange} about the stabilization is present for the negative caps as well, and if one assumes that the rank $1$ stabilizing problem on the cylinder and the positive cap have both index $0$, the index of the stabilizing problem on the negative cap is necessarily $1$ (by the gluing and Riemann--Roch, see the proof of \Cref{lem:negposframingsum}). On the other hand, the $k$-factor in the definition of $V_x'$ removes the dependency on $k$, similar to the alternative described in \Cref{note:stabilizationchange}. 
\end{rk}
We have:
\begin{lem}\label{lem:negposframingsum}
$V_x\oplus V_x^-\cong \underline{\bR}^{2(n+k)}$ and the stable framings $V_x$ and $V_x'$ are equivalent.
\end{lem}
\begin{proof}
There is a contractible space $\cU_{S^2}$ of Floer data $(J,Y)$ parametrized by the sphere (similar to \Cref{defn:abstractcap}, but the $\epsilon$, $L$, and conditions involving them are removed, $S_-$ is replaced by $S^2$), and there is a gluing map $\cU(x)\times \cU^-(x)\to\cU_{S^2}$. As before, there is a natural map 
\begin{equation}\label{eq:captofredsphere}
\cU_{S^2}\to Fred(W_1^2(S^2, \bC^{n+k}), L^2(S^2, \Omega_{S^2}^{0,1}\otimes \bC^{n+k}))
\end{equation}
It can be seen similar to \eqref{eq:gluingindexbundles} that the corresponding virtual bundle pulls back to $V_x^-\oplus V_x$ under the gluing map. On the other hand, by contractibility of $\cU_{S^2}$, its image contracts to the standard Cauchy--Riemann operator on $S^2=\bP^1$, which has a kernel of complex rank $n+k$ and vanishing cokernel (as $h^{0,0}(S^2)=1$ and $h^{0,1}(S^2)=0$). In other words, \eqref{eq:captofredsphere} lands in the index $2(n+k)$ operators, and the corresponding vector bundle on the contractible space $\cU_{S^2}$ is naturally isomorphic to $\underline{\bR}^{2(n+k)}$. Thus, $V_x\oplus V_x^-\cong \underline{\bR}^{2(n+k)}$, and $V_x'= \underline{\bR}^{2(n+k)}-V_x^-\cong V_x$. 

To see that \eqref{eq:alternativeframe} and the isomorphism 
\begin{equation}\label{eq:stdfrm}
	V_x\cong T_{\cM_{H,J}(x,y)}\oplus \bR \oplus V_y
\end{equation} coincide under the isomorphisms $V_x'\cong V_x$ and $V_y'\cong V_y$, one uses the gluing map 
\begin{equation}\label{eq:doublegluing}
\cU^-(x)\times \cM_{H,J}(x,y)\times \cU(y)\to \cU_{S^2}
\end{equation}
\eqref{eq:doublegluing} can be defined by gluing a cap at $y$ first, then at $x$, or vice versa, or by gluing caps on both ends simultaneously (these are homotopic maps). The virtual vector bundle on $\cU_{S^2}$ that is isomorphic to $\underline{\bR}^{2(n+k)}$ pulls-back to $V_x^-\oplus (T_{\cM_{H,J}(x,y)}\oplus \bR)\oplus V_y$, i.e. 
\begin{equation}\label{eq:doublegluinglinear}
V_x^-\oplus (T_{\cM_{H,J}(x,y)}\oplus \bR)\oplus V_y\cong \underline{\bR}^{2(n+k)}
\end{equation}  
As gluing caps at both ends in different orders lead to homotopic maps, the isomorphism \eqref{eq:doublegluinglinear} can also be obtained by adding $V_y$ to the two sides of \eqref{eq:negframe} and using $V_y^-\oplus V_y\cong \underline{\bR}^{2(n+k)}$ or by adding $V_x^-$ to the two sides of \eqref{eq:stdfrm} and using $V_x^-\oplus V_x\cong \underline{\bR}^{2(n+k)}$. Therefore, by adding $V_x^-\oplus V_y^-\oplus V_y$ to the both sides of \eqref{eq:alternativeframe}, we obtain $\eqref{eq:doublegluinglinear}\oplus \underline{\bR}^{2(n+k)}$, and which can be obtained from \eqref{eq:stdfrm}, by adding $V_x^-\oplus\underline{\bR}^{2(n+k)}$ to the both sides. This shows that \eqref{eq:alternativeframe} and \eqref{eq:stdfrm} agree. 

To check the compatibility of the isomorphisms $V_x\cong V_x'$ with the gluing-compatibility maps, one uses a similar argument for 
\begin{equation}
\cU^-(x)\times \cM_{H,J}(x,y)\times\cM_{H,J}(y,z)\times \cU(z)\to \cU_{S^2}
\end{equation}
\end{proof}

\subsection{$S^1$-equivariant spectral symplectic cohomology}\label{subsubsection:equivariant-SH-construction}
In this section, we use the Morse--Bott model for $ES^1$ that is introduced in \Cref{exmp:circlemorsebott}. In other words, we let $ES^1=S^\infty=\{(z_1,z_2,\dots)\}\subset \bC^\infty$ and $\tilde f:=|z_1|^2+2|z_2|^2+3|z_3|^2+\dots$, the Morse--Bott function with exactly one critical set of index $2(i-1)$, for each $i\geq 1$. The finite dimensional approximation $E_dS^1$ is given by $S^{2d-1}\subset S^\infty$. We will denote $ES^1=S^\infty$ also by $S$. 
For each critical manifold of $X\subset S$, fix a local slice for the $S^1$-action on $S$ which is transverse to the critical manifold and tangent to the gradient of $\tilde{f}$. This determines an equivariant tubular neighborhood $U$ of the critical manifold $X\cong S^1$ and a trivialization of it, i.e. an equivariant isomorphism $U\cong X\times B$. 

Following \cite{bourgeois2009symplectic}, consider the pairs $(H,J)$, where $H: S\times S^1\times M\to \bR$ and $J:S\times S^1\to\cJ(M)$ (the space of almost complex structures) such that
\begin{enumerate}
	\item $(H,J)$ is invariant under the diagonal circle action on $S\times S^1$
	\item if $z\in crit(\tilde f)$, $H_z:S^1\times M\to \bR$ is a non-degenerate Hamiltonian
	\item there exists a compact subset of $M$ such that $(H_z,J_z)$ is cylindrical outside this compact subset and $H_z$ is linear in a fixed Liouville parameter of fixed slope
    \item near each critical manifold $X\subset S$, $(H, J)$ only depends on the $X$ coordinate (with respect to the trivialization of the equivariant tubular neighborhood mentioned above)
\end{enumerate}


We consider the moduli space of solutions $(\gamma, u)$ to the system of equations
\begin{equation} \label{equation:sh-tilde}
    \left\{
  \begin{array}{@{}rr@{}}
    \dot{\gamma}+ \nabla \tilde{f}(\gamma) &=0 \\
    \partial_s u + J_{\gamma(s)}(\partial_t u - X_{H_{\gamma(s)}}) &=0
  \end{array}\right.
\end{equation}
where $\gamma: \bR \to S$ and $u: \bR \times S^1 \to M$, subject to the condition that they are positively/negatively asymptotic to pairs $(a,x)$, where $a\in crit(\tilde f)$ and $x\in orb(H_a)$.

Observe that the first equation of \eqref{equation:sh-tilde} is just the negative gradient flow equation for $\tilde{f}$ on $S$. Moreover, given a trajectory $\gamma$ solving the first equation, the second equation is just a standard continuation equation. (Here we use our assumptions that  $(H, J)$ are independent of $z \in S$ along each slice, and that $\nabla \tilde{f}$ is tangent to each slice). The moduli space of solutions $(\gamma, u)$ to \eqref{equation:sh-tilde} carries an $\mathbb{R}$-action by translation in the domain, and there is also a natural free $S^1$-action.

We will now construct a flow category $\cM_{\tilde f,H,J}$ as follows:
\begin{itemize}
\item the objects $ob(\cM_{\tilde f,H,J})$ are the     critical manifolds formed by pairs $(a,x)$ where        $a\in crit(\tilde f)$, $x\in orb(H_a)$. The index of $(a,x)$ is defined as the sum of indices of $a$ and $x$. Note that the object space is a disjoint union of circles.
\item the morphism spaces $\cM_{\tilde f,H,J}((a,x),(b,y))$ are the compactified moduli spaces of solutions to \eqref{equation:sh-tilde}, considered up to $\mathbb{R}$-translation. Concretely, interior points are pairs $(\gamma, u)$ solving \eqref{equation:sh-tilde} and the boundary elements correspond as usual to broken trajectories. By our conventions, $\gamma$ is asymptotic to $a$ at $-\infty$ and $b$ at $+\infty$, but $u$ is $\gamma$ is asymptotic to $x$ at $+\infty$ and $y$ at $-\infty$ (i.e. the input of a Floer trajectory is the positive end). 
The dimension of $\cM_{\tilde f,H,J}((a,x),(b,y))$ is $ind(a)+ind(x)-ind(b)-ind(y)-1$ and we will argue that it carries the structure of an $\langle ind(a)+ind(x)-ind(b)-ind(y)-2\rangle$-manifold. 
\end{itemize}

\begin{prop}[cf.\ {\cite{largethesis}}, \cite{foooexponential}]\label{proposition:s1-flow-cat} $\cM_{\tilde f,H,J}$ is an $S^1$-equivariant flow category.
\end{prop}
A proof will be given in \Cref{sec:gluing}. 

\subsection{Equivariant framings}\label{subsubsec:equivframings}
In this section, we combine the previous notions to extend the framings to the equivariant setting. The space of objects of $\cM_{\tilde f,H,J}$ is a disjoint union of $S^1$-torsors. Let $X$ be such a torsor. Then the spaces of abstract caps $\cU(x)$, where $(a,x)\in X$, forms an $S^1$-equivariant fiber bundle over $X$ with contractible fibers. We call this bundle $\cU(X)$. The projection $\cU(X)\to X$ is an $S^1$-equivariant homotopy equivalence. Indeed, one can extend a cap in $\cU(x)$, where $(a,x)\in X$, to an equivariant section $X\to \cU(X)$, and $X$ is an equivariant deformation retract of $\cU(X)$. As before, we have a map from $\cU(X)$ to the space of Fredholm operators, and a corresponding index bundle. The index bundle carries the structure of an $S^1$-equivariant virtual bundle. Call this bundle $V_X^{cap}$. 
\begin{rk}\label{rk:s1eqfred}
Notice that $\cU(X)\cong\cU(x)\times X\cong \cU(x)\times S^1$, and the restriction of the index bundle of the $S^1$-equivariant map 
\begin{equation}
\cU(X)\to  Fred(W_1^2(S_-, \bC^{n+k}), L^2(S_-, \Omega_{S_-}^{0,1}\otimes \bC^{n+k})) 
\end{equation} 
to $\cU(x)\times\{\alpha\}$ is canonically the same for every $\alpha\in X$. More precisely, to construct the index bundle over $\cU(x)\times\{\alpha\}$, one kills the kernel of the Fredholm operators uniformly over $\cU(x)\times\{\alpha\}$ by replacing the the target with some $L^2(S_-, \Omega_{S_-}^{0,1}\otimes \bC^{n+k})\oplus \bC^m$, and further changing the family of Fredholm operators by adding maps $W_1^2(S_-, \bC^{n+k})\to \bC^m$ parametrized by $\cU(x)\times\{\alpha\}$. On the other hand, if one chooses such a family for a single $\alpha$, they one can extend it to $\cU(X)\cong \cU(x)\times X$ by using the $S^1$-action on $W_1^2(S_-, \bC^{n+k})$ induced by the rotation action on $S_-$. Hence, we have an $S^1$-equivariant map 
\begin{equation}
F:	\cU(X)\to  Fred(W_1^2(S_-, \bC^{n+k}), L^2(S_-, \Omega_{S_-}^{0,1}\otimes \bC^{n+k})\oplus \bC^m) 
\end{equation} 
with image landing in the injective operators, and the index bundle defined as $\underline \bC^m-coker (F)$ is clearly $S^1$-equivariant.  
Note that there is an equivariant version of the Atiyah--J\"anich theorem; however, we do not rely on it; see \cite{matumotoeqkfredholm}.
\end{rk}
Let $X_0\subset S^\infty$ be the underlying critical set of $\tilde f$. Define $V_X:=T^d_{X_0}\oplus V_X^{caps}$ (we consider the pull-back of $T_{X_0}^d$ to be more precise). This is our framing bundle. 

The tangent bundle $T_{\cM_{\tilde f,H,J}}\oplus\underline \bR$ has a subbundle $T_{\cM_{\tilde f,H,J}}^{cyl}$. At $(\gamma,u)$, its fibers are given by the infinitesimal deformations of the pair that do not change $\gamma$. In other words, when $\gamma$ is non-constant, it is the kernel of the projection map $T_{\cM_{\tilde f,H,J}}\oplus\underline\bR\to T_{\cM_{\tilde f}}\oplus\underline\bR$ (the objects are suppressed from the notation). When $\gamma$ is the constant trajectory at $s\in X_0$ (and $u$ is a Floer trajectory with respect to $(H_s,J_s)$), one has 
\begin{equation}
(T_{\cM_{\tilde f,H,J}}^{cyl})_{(\gamma,u)}=(T_{\cM_{H_s,J_s}})_{(\gamma,u)}\oplus \underline\bR
\end{equation}

The same construction as in \Cref{subsubsec:reminderslarge} gives us isomorphisms $V_X^{caps}\cong T_{\cM_{\tilde f,H,J}(X,Y)}^{cyl} \oplus V_Y^{caps}$. On the other hand, if $X_0$, $Y_0$ denote the underlying critical sets of $\tilde f$, one also has $T^d_{X_0}\oplus T_{X_0}\cong T_{\cM_{\tilde f}(X_0,Y_0)}\oplus \underline\bR\oplus T^d_{Y_0}$. Obviously $X\xrightarrow{\cong} X_0$ and $Y\xrightarrow{\cong} Y_0$. These combine to give us 
\begin{equation}
V_X\oplus T_X\cong 	T_{\cM_{\tilde f,H,J}(X,Y)}\oplus\underline\bR \oplus V_Y^{caps}
\end{equation}
as $S^1$-equivariant virtual bundles. The case $X_0=Y_0$ and $\gamma$ is constant is similar. 
In summary,
\begin{prop}
$\cM_{\tilde f,H,J}$ is an $S^1$-equivariant framed flow category.
\end{prop}
One needs to check the compatibility of the framing with the $\langle k\rangle$-manifold structure on moduli spaces, which can be done similar to \cite[\S7,\S8]{largethesis}.

Note the slight abuse of notation in that this is an ind-flow category: $\tilde f:S=S^\infty=ES^1\to \bR$ itself is defined on an infinite dimensional manifold. However, as remarked before, to define its geometric realization, one can apply the geometric realization functor to the corresponding constructions for $E_jS^1=S^{2j-1}$, and take a homotopy colimit.

We denote the genuine $S^1$-equivariant spectrum $|\cM_{\tilde f,H,J}|$ by $HF_S(H,\bS)$, suppressing $\tilde f$ and $J$ from the notation. The independence from $J$ is standard (see also \Cref{subsubsec:contmaps}). Note that, as a Borel equivariant spectrum, $HF_S(H,\bS)$ is also independent of $(S,\tilde f)$. 

\subsection{Continuation maps and invariance}\label{subsubsec:contmaps}
In this section we construct continuation maps $HF_S(H_0,\bS)\to HF_S(H_1,\bS)$, whenever $H_1\geq H_0$ over $S\times S^1\times M$. The construction is standard in Floer theory, we only need to implement it using the language of bimodules. One can show independence of the construction from $\tilde f$ using similar methods; however, we exclude this discussion. 

Choose a monotone $S^1$-equivariant interpolation from $H_0$ to $H_1$, i.e. a family of $S^1$-equivariant functions $H_r:S\times S^1\times M\to\bR$, $r\in [0,1]$ such that at $r=0,1$ these match with $H_0,H_1$ above and such that $H_{r'}\geq H_r$ for $r'\geq r$ (outside a fixed compact subset of $M$). Similarly, choose an $S^1$-equivariant interpolation $J_r$ of the corresponding almost complex structures. We assume all $H_r$ are linear at infinity in a uniform manner. 

Consider the Morse function $g(t)=t^3/3-t^2/2$ on $I$, and let $r(t)$ denote a gradient trajectory from $0$ to $1$ (the choice is up to translation). We define an $\cM_{\tilde f,H_0}$-$\cM_{\tilde f,H_1}$-bimodule $\cN_{\tilde f,H_r,J}$ by considering the moduli of (broken) pairs $(\gamma,v)$ satisfying 
\begin{equation} \label{equation:continuationsh-tilde}
	\left\{
	\begin{array}{@{}rr@{}}
		\dot{\gamma}+ \nabla \tilde{f}(\gamma) &=0 \\
		\partial_s v + J_{r(s),\gamma(s)}(\partial_t v - X_{H_{r(s),\gamma(s)}}) &=0
	\end{array}\right.
\end{equation}
The asymptotic conditions are similar. We let $\cN_{\tilde f,H_r,J}(i,j)$ to be the compactification of the moduli of solutions to this equation (not up to translation, the translation symmetry is already broken). It is easy to show that $\cN_{\tilde f,H_r,J}$ is a bimodule, and can be framed in an analogous way to $\cM_{\tilde f,H_0}$ and $\cM_{\tilde f,H_1}$. 
\begin{rk}
One can let the pair $(\tilde f,g)$ vary in a controlled and monotone way as $(\tilde f_r,g_r)$ and modify the first equation as $\dot{\gamma}(s)+ \nabla_{g(r(s))} {\tilde f_r(s)}(\gamma(s))=0$ in order to obtain an $\cM_{\tilde f_0,H_0}$-$\cM_{\tilde f_1,H_1}$-bimodule. Ignoring the Cauchy--Riemann equation, this is how one would define continuation maps in Morse theory. We omit this. 	
\end{rk}
As a result, we obtain a map of genuine $S^1$-equivariant spectra
\begin{equation}
HF_S(H_0,\bS)=|\cM_{\tilde f,H_0}| \to HF_S(H_1,\bS)=|\cM_{\tilde f,H_1}|
\end{equation}

\begin{defn}
Define $SH_{S}(M; \bS):= \colim HF_{S}(H; \bS)$, where the colimit is over the space of all admissible Floer data and the maps induced by all possible interpolations. 
\end{defn}

\begin{rk}\label{rk:fixedinterpolation}
One can interpolate Floer data by adding constant functions to $H$. Notice that $|\cM_{\tilde f,H+a,J}|$, $a\geq 0$, canonically identifies with $|\cM_{\tilde f,H,J}|$, and it is easy to prove using a filtration argument that the interpolation is an equivalence. In fact, it is the identity. This can be shown by explicit methods too; however, we omit this. One can avoid such an explicit proof by replacing $c:HF_{S}(H; \bS)\to HF_{S}(H+a; \bS)$ with $c\circ e_{H}^{-1}$ or equivalently by $e_{H+a}^{-1}\circ c$. Using \Cref{prop:homotopyinduced}, one can prove that the induced map is the identity. In particular, for $a=0$, this can used to prove independence from the choice of almost complex structure. 
\end{rk}




\subsection{Comparison of $SH(M,\bS)$ and $SH_S(M,\bS)$}\label{subsec:comparisonsheqsh}
In this section, we prove
\begin{prop}\label{prop:comparisonsheqsh}
There is a natural homotopy equivalence $SH(M,\bS)\xrightarrow{\simeq} SH_S(M,\bS)$ of non-equivariant spectra. 
\end{prop}
The equivalence we produce is induced by an inclusion (in a sense slightly broader than \Cref{defn:inclusion}). More precisely, let $X_0$ denote the index $0$ critical manifold of $\tilde f$. There is an inclusion of flow categories $\cM_{\tilde f,H,J}^0\hookrightarrow \cM_{\tilde f,H,J}$, where $\cM_{\tilde f,H,J}^0$ denotes the span of objects $(a,x)$, where $a\in X_0$ and $x\in orb(H_a)$. Notice that, as $(H_a,J_a)$ are related by a rotation for different $a\in X_0$, $\cM_{H_a,J_a}$ and $\cM_{H_{a'},J_{a'}}$ are canonically identified. By extension, one can identify $\cM_{\tilde f,H,J}^0$ with $X_0\times \cM_{H_a,J_a}$ for a fixed $a\in X_0$, i.e. the flow category obtained from $\cM_{H_a,J_a}$ by taking products of objects and morphisms with $X_0$. Hence, the geometric realization of  $\cM_{\tilde f,H,J}^0$ is the same as $(X_0)_+\wedge |\cM_{H_a,J_a}|$. The inclusion $\{a\}\hookrightarrow X_0$ induces a map from $|\cM_{H_a,J_a}|$ to $ |\cM_{\tilde f,H,J}^0|$; hence, we have 
\begin{equation}\label{eq:inclsnoneqHF}
HF(H_a,\bS)=|\cM_{H_a,J_a}|\to |\cM_{\tilde f,H,J}^0|\to |\cM_{\tilde f,H,J}|=HF_S(H,\bS)
\end{equation} 
inducing 
\begin{equation}\label{eq:inclnoneqSH}
SH(M,\bS)\to SH_S(M,\bS)
\end{equation}
\begin{proof}[Proof of \Cref{prop:comparisonsheqsh}]
We show that \eqref{eq:inclnoneqSH} is an equivalence. To see this, observe that one can use a constant family $(H_s,J_s)$, $s\in S$, if one drops the assumption that $(H_s,J_s)$ and $(H_{z.s},J_{z.s})$ are related by rotation. In other words, for a fixed pair $(H_0,J_0)$, where $H_0$ is non-degenerate and $J_0$ is generic, one can assume $(H_s,J_s)=(H_0,J_0)$. We denote such a family by $(H^{const},J^{const})$. One has a map $HF(H_0,\bS)\to HF_S(H^{const},\bS)$ constructed similarly to \eqref{eq:inclsnoneqHF}, and this map is an equivalence. To see this is an equivalence, filter the corresponding flow categories by the action of $H_0$. The maps induced on subquotients are of the form $\bS\to \Sigma^\infty S_+$, which is an equivalence (we are implicitly using \Cref{prop:mbscjs} for $S$, ignoring the equivariant structure). Note that one can see both this map and \eqref{eq:inclsnoneqHF} as maps induced by a bimodule, produced using the constant interpolation of the corresponding Hamiltonian (cf. \Cref{rk:fixedinterpolation}). As a result, as one takes the colimit as $slope(H_0)\to \infty$, one still obtains an equivalence. 

Observe that if one ignores the equivariant structure, $SH_S(M,\bS)$ defined using the equivariant Hamiltonians and the constant Hamiltonians agree: indeed one can take the colimit over a bigger diagram containing both, and each of these two types form a cofinal sequence within the diagram. The result follows from this observation. 
\end{proof}
\section{Cieliebak--Latschev maps and the spectral equivariant Viterbo isomorphism}\label{sec:clmaps}

Given a Liouville manifold $M$, and an exact, compact, spin Lagrangian $Q\subset M$, there is a map $SH^*(M)\to H_{n-*}(\cL Q)$ called the \emph{Cieliebak--Latschev map}; see \cite{cieliebaklatschev}, \cite[Sec.\ 4.4]{zhaothesis}, \cite{abouzaid2013symplectic}.  
When $M=T^*Q$ and $Q$ is the zero section, this map is an isomorphism called the \emph{Viterbo isomorphism} (when the spin assumption is dropped, one needs to consider the homology of $\cL Q$ with local coefficients). In this \namecref{sec:clmaps}, we lift the construction of Cieliebak--Latschev maps to the category of genuine $S^1$-equivariant spectra and we prove that this map is a homotopy equivalence of Borel equivariant spectra. 

\subsection{Statement of the theorems}
We begin by explaining the precise statement we will prove. Let $M$ be a Liouville manifold endowed with a background framing $TM\oplus \underline{\bC}^k\cong \underline\bC^{n+k}$. Given this, one obtains a stable unitary/symplectic trivialization of $TM|_Q\cong TQ\otimes \bC=:T_\bC Q$, i.e. an isomorphism $T_\bC Q\oplus \underline\bC^k\cong \underline\bC^{n+k}$. As a result, we obtain a Gauss map $\rho$ from $Q$ to the Lagrangian Grassmanian $LGr(\bC^{n+k})$, given by $x\mapsto T_xQ\oplus \bR^k$. 

One can define a virtual bundle on $\cL LGr(\bC^{n+k})$, whose rank on a specific component is given by the Maslov index (so the rank may vary). An index theoretic description of this bundle is as follows: for every loop $\Lambda\in \cL LGr(\bC^{n+k})$, one has a linear Cauchy--Riemann problem on the disc $\bD^2$ with the Lagrangian boundary conditions $\Lambda$. In other words, one obtains a map from $\cL LGr(\bC^{n+k})$ to $Fred\simeq BU\times \bZ$, defining a virtual bundle. Here we are using Fredholm operators on a disc, which carries an $S^1$-action by rotation. The classifying map can be made $S^1$-equivariant, endowing the virtual bundle with an $S^1$-equivariant structure (c.f. \Cref{rk:s1eqfred}, we will give more details). Note that the rank of this bundle is different on different components of $\cL LGr(\bC^{n+k})$. 
Let $W_{mas}$ denote the virtual bundle on $\cL LGr(\bC^{n+k})$ obtained by subtracting $\underline\bR^{n+k}$ from the index bundle. We prove:
\begin{thm}\label{thm:equivariantcl}
There is a morphism of genuine $S^1$-equivariant spectra $SH_{S}(M,\bS)\to \cL Q^{\underline \bR^n-\rho^*W_{mas}}$.
\end{thm}
Note the slight abuse of notation: we denote the map $\cL Q\to\cL LGr(\bC^{n+k})$ induced by $\rho$ by the same letter. Observe that, when $Q$ is simply connected, $\rho^*W_{mas}$ is trivial of rank $0$, and $\cL Q^{\underline \bR^n-\rho^*W_{mas}}=\Sigma^{\infty+n}\cL Q$. 

Applying homology to \Cref{thm:equivariantcl}, one obtains the map in \cite{zhaothesis}. In particular, the classical Viterbo isomorphism theorem tells us that this map is an isomorphism on integral homology for $M=T^*Q$. We will show
\begin{thm}\label{thm:viterbo}
The map $SH_{S}(T^*Q,\bS)\to \cL Q^{\underline \bR^n-\rho^*W_{mas}}$ constructed in \Cref{thm:equivariantcl} is an equivalence of Borel/homotopy $S^1$-equivariant spectra. 
\end{thm}
Recall that a Borel equivariant spectrum is a functor $BS^1\to Sp$. A map of Borel equivariant spectra (i.e. a natural transformation of the corresponding functors) that induce an homotopy equivalence of the underlying non-equivariant spectra is automatically an equivalence. 

\subsection{An $\mathcal{L}Q$-valued module}

To construct the Cieliebak--Latschev map at homology level, one constructs a moduli space of half-cylinders with boundary on $Q$, makes consistent choices of fundamental chains on them, and pushes-forward along the evaluation maps to $\cL Q$ (given by restricting to the boundary); see \cite[Sec.\ 4.4]{zhaothesis}. Instead, we will use these moduli spaces to produce $\cL Q$-valued modules. 

As a warmup, we first explain the non-equivariant case. Choose $(H,J)$ as in \Cref{subsubsec:reminderslarge}. Let $Z_+:=[0,\infty)\times S^1$ be endowed with generic Floer data that matches $(H,J)$ in a fixed positive cylindrical end. Consider the moduli space of maps $v: Z_+\to M$ that are solutions to the non-linear Cauchy--Riemann equation for $(H, J)$, subject to the condition that $v$ is asymptotic to $x\in orb(H)$ at infinity and sends $\{0\}\times S^1$ into $Q$. This moduli space can be compactified by adding broken half cylinders. Let $\cN_{H,J}(x)$ denote the compactification. The analysis in \cite[Sec.\ 6]{largethesis} implies that this is a smooth manifold with corners. 
Moreover, there is a natural evaluation map $\cN_{H,J}(x)\to \cL Q$ sending a half cylinder to its restriction to the boundary. In other words, 
\begin{prop}\label{proposition:module-warmup}
$\cN_{H,J}$ is an $\cL Q$-valued module on $\cM_{H,J}$.	
\qed
\end{prop}


We now explain how to generalize this construction to the equivariant setting. Let $(\tilde f,H,J)$ be as in \Cref{subsubsection:equivariant-SH-construction}. Choose Floer data on $Z_+$-parametrized by $S=S^\infty$ satisfying the equivariance condition in \Cref{subsubsection:equivariant-SH-construction}. In other words, choose a family of Hamiltonians $H:S\times Z_+\to\bR$ and almost complex structures $J:S\times Z_+\to\cJ(M)$ extending the previously chosen Floer data on the cylindrical end (hence, we use the same notation), and satisfying
\begin{enumerate}
	\item $(H,J)$ is invariant under the diagonal circle action on $S\times Z_+$
	\item there exists a compact subset of $M$ such that $(H_s,J_s)$ is cylindrical outside this compact subset and $H_s$ is linear in a fixed Liouville parameter of fixed slope
\end{enumerate} 
We construct an $\cM_{\tilde f,H,J}$-module $\cN_{\tilde f,H,J}$ as follows: given $(a,x)$ (where $a\in crit(\tilde f)$, $x\in orb(H_a)$), $\cN_{\tilde f,H,J}(a,x)$ is the moduli space of (broken) pairs $(\beta,v)$ where 
\begin{itemize}
	\item $\beta: (-\infty,0]\to S$ is an half gradient trajectory of $-\tilde f$ such that $\beta(s)\to a$ as $s\to-\infty$
	\item $v:Z_+\to M$ is a solution to the non-linear Cauchy--Riemann equation with respect to the domain dependent data $(H_{\beta(s)}, J_{\beta(s)})$ and it is asymptotic to $x$ at $+\infty$, and satisfies $v(\{0\} \times S^1) \subset Q$.
\end{itemize}
As before, $(a,x)$ varies continuously, and $\cN_{\tilde f,H,J}(i)$ is given by all such pairs $(\beta,v)$ where $(a,x)$ has index $i$. Analogously to $\cN_{H,J}$, $\cN_{\tilde f,H,J}$ is a module over $\cM_{\tilde f,H,J}$. It is $S^1$-equivariant and admits an equivariant map to $S\times \cL Q$, given by $(\beta,v)\mapsto (\beta(0), v(0,\cdot))$. In other words,
\begin{prop}\label{prop:eqlqvaluedmodule}
    $\cN_{\tilde f,H,J}$ is an $S\times\cL Q$-valued module. 
\end{prop}
One can ignore the $S$-component and obtain an $\cL Q$-valued module as in \Cref{proposition:module-warmup}. 

\subsection{Framing the moduli of half cylinders}
The goal of this section is to show that $\cN_{\tilde f,H,J}$ is framed with respect to the virtual bundle $\underline\bR^n-\rho^*W_{mas}$ on $\cL Q$. In other words, 
\begin{prop}\label{prop:framingoncl}
There are natural isomorphisms of virtual bundles \begin{equation}
	V_X\oplus T_X\cong T_{\cN_{\tilde f,H,J}(X)}\oplus (\underline \bR^n-\rho^*W_{mas})
\end{equation} defining a framing on $\cN_{\tilde f,H,J}$.
\end{prop}
The argument uses the equivalent framing via negative caps. We find explaining the non Morse--Bott case for $\cN_{H,J}$ (where $H,J$ are not $S$-dependent) more convenient. The extension to the equivariant Morse--Bott case is analogous to \Cref{subsubsec:equivframings}. 
First we define
\begin{defn}[{cf. \cite[Defn 8.4]{largethesis}}]
Let $\Lambda\in\cL LGr(\bC^{n+k})$ be a loop. Define an \emph{abstract Lagrangian cap at $\Lambda$} to be a tuple $(J,Y,g)$ where
\begin{itemize}
	\item $J$ is a family of complex structures on $\bC^{n+k}$ parametrized by $\bD^2$ such that $\Lambda(t)$ is totally real for all $t\in S^1$
	\item $Y\in \Omega_{\bD^2}^{0,1}\otimes_\bC \bC^{n+k}$
	\item $g$ is a metric on $\bD^2$
\end{itemize}
The space of loops endowed with abstract caps form an $S^1$-equivariant fibration over $\cL LGr(\bC^{n+k})$ with contractible fibers. We denote this fibration by $\cV$. Denote its pullback under $\rho: \cL Q\to \cL LGr(\bC^{n+k})$ by $\cV_Q$, and call this the space of $Q$-caps. 
\end{defn}
As before, every abstract cap defines a Fredholm operator
\begin{equation}
W_1^2(\bD^2,\partial\bD^2; \bC^{n+k},\Lambda)\to L^2(\bD^2, \Omega_{\bD^2}^{0,1}\otimes \bC^{n+k})	
\end{equation}
Here, the domain is the $W_1^2$-space of $\bC^{n+k}$-valued functions on $\bD^2$ that map the boundary point $t\in S^1$ to the Lagrangian subspace $\Lambda_t\subset \bC^{n+k} $. As a result, one obtains an $S^1$-equivariant map 
\begin{equation}\label{eq:indexmaptofredlagr}
\cV\to Fred(	W_1^2(\bD^2,\partial\bD^2; \bC^{n+k},\Lambda), L^2(\bD^2, \Omega_{\bD^2}^{0,1}\otimes \bC^{n+k}))	
\end{equation}
and an $S^1$-equivariant virtual bundle on $\cV\simeq \cL LGr(\bC^{n+k})$ (see \Cref{rk:s1eqfred}). Let $W_{mas}$ denote the virtual bundle such that $W_{mas}+\underline\bR^{n+k}$ is the index bundle of \eqref{eq:indexmaptofredlagr} (or the corresponding $S^1$-equivariant bundle on $\cL LGr(\bC^{n+k})$).
\begin{note}
Strictly speaking, the domains of the Fredholm operators change as one varies $\Lambda$. An easy way around in the non-equivariant case is to use the fact that each space of Fredholm operators is isomorphic to $Fred(\cH_1,\cH_2)$, where $\cH_1$, $\cH_2$ are fixed Hilbert spaces, and the isomorphism is unique up to a contractible choice (by Kuiper's theorem). More concisely, one has a bundle over $\cV$ with fibers given by $Fred(\cH_1,\cH_2)$, and a section. The structure group of this bundle is $U(\cH_1)\times U(\cH_2)\simeq 1$, and so it is trivial, and the section defines a map to $Fred(\cH_1,\cH_2)$. In the equivariant case, first replace $\cV$ by $\widetilde \cV=\cV\times ES^1$, to assume the action is free. Let $\scrH_1$ denote the $S^1$-equivariant bundle over $\widetilde{\cV}$ with fibers given by $W_1^2(\bD^2,\partial\bD^2; \bC^{n+k},\Lambda)$ and $\scrH_2$ denote the trivial $S^1$-equivariant fiber bundle with the same fibers. Then the bundle $\scrU=\cU(\scrH_1,\scrH_2)$ carries an $S^1$-equivariant structure, and has contractible fibers by Kuiper's theorem. $\scrU$ descends to a bundle over $\widetilde{\cV}/S^1=\cV\times_{S^1} ES^1$, and the descent admits a section. As a result, $\scrU$ admits an $S^1$-equivariant section, and using this, one can equivariantly trivialize the bundle whose fibers are as in the right hand side of \eqref{eq:indexmaptofredlagr}. The trivialization is unique up to equivariant homotopy as before, and one obtains an equivariant index bundle over $\widetilde{\cV}$ (hence, over $\cV$ and $\cL LGr(\bC^{n+k})) $. 

One can also prefer to define the equivariant index bundle explicitly over the free $S^1$-space $\widetilde{\cV}$, similar to \Cref{rk:s1eqfred}. 
\end{note}
One can stabilize abstract caps by taking direct sums with the standard Cauchy--Riemann operator \begin{equation}\label{eq:stabdisc}
	W_1^2(\bD^2,\partial\bD^2; \bC,\Lambda_0)\to L^2(\bD^2, \Omega_{\bD^2}^{0,1}\otimes \bC)
\end{equation}
where $\Lambda_0$ is the constant loop $\bR$. This problem has index $1$, and the effect on the corresponding index bundle is adding a trivial line bundle. As a result, $W_{mas}$ does not change (in other words, $W_{mas}$ defined on the caps of rank $n+k+1$ pull-back to $W_{mas}$ on the caps of rank $n+k$ under the stabilization map).

Note that, there is an analogous stabilizing problem 
\begin{equation}\label{eq:stabforhalfcylinder}
	W_1^2(Z_+,\partial Z_+; \bC,\Lambda_0)\to L^2(Z_+, \Omega_{Z_+}^{0,1}\otimes \bC)
\end{equation}
on $Z_+$ given by $\partial_s+i(\partial_t-B)$, where $B$ is as in \eqref{equation:CR-stabilized}. The index bundle of \eqref{eq:stabforhalfcylinder} direct summed with the index bundle of the rank $1$ stabilizing problem on a negative cap $S_-$ gives the index bundle of \eqref{eq:stabdisc}. By our choice of $B$, the index on the negative cap is $2$; therefore, the index of \eqref{eq:stabforhalfcylinder} is $-1$. 

As a result of this discussion, given $u\in \cN_{H,J}(x)$, when one stabilizes 
\begin{equation}
	W_1^2(Z_+,\partial Z_+; u^*T_M,u^*T_Q)\to L^2(Z_+, \Omega_{Z_+}^{0,1}\otimes u^*T_M)
\end{equation}
by adding multiples of \eqref{eq:stabforhalfcylinder}, the index bundle becomes $T_{\cN_{H,J}(x),u}-\underline\bR^k$. Given $u\in \cN_{H,J}(x)$, one can glue a negative cap at $x$ to $u$ to obtain a $Q$-cap. In other words, there is a map $\cU^-(x)\times \cN_{H,J}(x)\to \cV_Q$. Moreover, the index bundles glue, and we have
\begin{lem}\label{lem:indexgluehalfneg}
$(\rho^*W_{mas}+\underline\bR^{n+k})=(T_{\cN_{H,J}(x)}-\underline\bR^k)\oplus V_x^-$.
\end{lem}
The left hand side is by definition the index bundle of \eqref{eq:indexmaptofredlagr}. 
\begin{proof}[Proof of \Cref{prop:framingoncl}]
This now follows directly from \Cref{lem:indexgluehalfneg}. In other words, \Cref{lem:indexgluehalfneg} implies 
\begin{equation}
V_x'=\underline\bR^{2(n+k)}-V_x^-=T_{\cN_{H,J}(x)}\oplus (\underline\bR^n-\rho^*W_{mas})
\end{equation}
Checking compatibility of these isomorphisms with the module structure on $\cN_{H,J}$ is similar to \Cref{lem:negposframingsum} and \cite[\S8]{largethesis} and we omit this proof. 
\end{proof}
As a result we have maps $HF(H,\bS)\to \cL Q^{\underline \bR^n-\rho^*W_{mas}}$. It is not hard to show that these are compatible with the continuation maps, giving us a map $SH(H,\bS)\to \cL Q^{\underline \bR^n-\rho^*W_{mas}}$. 

One can endow $\cN_{\tilde f,H,J}$ with equivariant framings with respect $\underline \bR^n-\rho^*W_{mas}$. The construction is the same as \Cref{subsubsec:equivframings}. Hence, we have $S^1$-equivariant maps $HF_{S}(H,\bS)\to \cL Q^{\underline \bR^n-\rho^*W_{mas}}$ which induce an $S^1$-equivariant map 
\begin{equation}
	SH_{S}(M,\bS)\to \cL Q^{\underline \bR^n-\rho^*W_{mas}}
\end{equation}
This finishes the proof of \Cref{thm:equivariantcl}.

\subsection{The spectral equivariant Viterbo isomorphism theorem}\label{sec:viterbo}
In this section, we prove \Cref{thm:viterbo}. In other words, for a closed manifold $Q$ such that $M=T^*Q$ is stably framed, we show that 
\begin{equation}\label{eq:viterboinnewsection}
SH_{S}(T^*Q,\bS)\to \cL Q^{\underline \bR^n-\rho^*W_{mas}}
\end{equation}	 
constructed in \Cref{thm:equivariantcl} is an equivalence of Borel equivariant $S^1$-spectra. As remarked, it suffices to show that this map is an homotopy equivalence. By \cite[\href{https://kerodon.net/tag/01DK}{tag01DK}]{kerodon}, a map of Borel equivariant spectra that induce an homotopy equivalence of the underlying non-equivariant spectra is automatically an equivalence of Borel equivariant spectra. This claim is an infinity-categorical version of the classical statement that a natural transformation inducing isomorphisms at object level is invertible. 

For simplicity, first assume $Q$ is simply connected; hence, $\rho^*W_{mas}=0$. It is easy to see that the composition
\begin{equation}
SH(T^*Q,\bS)\xrightarrow{\simeq} SH_S(T^*Q,\bS)\to \Sigma^{\infty+n}\cL Q	
\end{equation}
of \eqref{eq:inclnoneqSH} and \eqref{eq:viterboinnewsection} is the same as the map $SH(T^*Q,\bS)\to \Sigma^{\infty+n}\cL Q$ induced by $\cN_{H,J}$ (\eqref{eq:inclsnoneqHF} and \eqref{eq:inclnoneqSH} are given by (general) inclusions, and $\cN_{H,J}$ is the restriction of $\cN_{\tilde f,H,J}$ along this inclusion). 
 
Therefore, it suffices to show that $C=cone (SH(T^*Q,\bS)\to \Sigma^{\infty+n}\cL Q	)$ is $0$. We observe the following 
\begin{enumerate}
	\item\label{item:connective} both $SH(T^*Q,\bS)$ and $\Sigma^{\infty+n}\cL Q$ are bounded below (hence so is $C$)
	\item\label{item:classviter} by the classical Viterbo isomorphism $SH(T^*Q,\bS)\to \Sigma^{\infty+n}\cL Q$ induces an isomorphism in integral homology
\end{enumerate}
\eqref{item:connective} follows from the fact that Conley--Zehnder indices agree with the geodesic (Morse) index (up to a universal shift and/or sign change, depending on one's conventions) \cite[Sec.\ 1]{abbondandoloschwarz1}, \cite[Thm.\ 1.2]{weber2002perturbed}. 
As a result of \eqref{item:classviter}, $H_*(C,\bZ)=0$. By the bounded below assumption \eqref{item:connective}, one can apply the (stable) Hurewicz theorem to $C$ to conclude that $C\simeq 0$. 
\begin{rk}
Note that we do not apply Hurewicz to $SH(M,\bS)$ or $\Sigma^{\infty+n}\cL Q	$ separately, even though these are bounded below. Rather we apply this theorem to their cone $C$, to conclude inductively that $\pi_k^{st}(C)=0$ for all $k$. 
\end{rk}

When the simply connected assumption is dropped, $\cL Q^{\underline \bR^n-\rho^*W_{mas}}$ is not a priori bounded below. However, one can decompose $SH(T^*Q,\bS)$, $\cL Q^{\underline \bR^n-\rho^*W_{mas}}$, and \eqref{eq:viterboinnewsection} into 
\begin{equation}\label{eq:decomposedviterboinnewsection}
	SH_{S}(T^*Q,\bS)_\alpha\to \cL_\alpha Q^{\underline \bR^n-\rho^*W_{mas}}
\end{equation}	 
indexed by the homotopy classes $\alpha\in\pi_1(Q)$. The cone of \eqref{eq:decomposedviterboinnewsection} is bounded below; therefore, $C$ is a wedge of bounded below spectra, and one can still apply Hurewicz theorem to conclude it is $0$ if $H_*(C,\bZ)=0$. On the other hand, the homology of $\cL Q^{\underline \bR^n-\rho^*W_{mas}}$ is given by $H_{*-n}(\cL Q, \eta)$, where $\eta$ is a local system corresponding to $-\rho^*W_{mas}$. The map induced on homology is given by $SH(M,\bZ)\to H_{*-n}(\cL Q, \eta)$, and the Viterbo isomorphism still holds, with twisted coefficients. See \cite{abouzaid2013symplectic, zhaothesis} for more details. 

This concludes the proof of \Cref{thm:viterbo}.

\begin{rk}
The map we constructed has no chance of being a genuine equivalence: even when $M$ is a point, it coincides with $\Sigma^\infty S_+\to \bS$. On the other hand, one can use the same relative module to induce a map $SH(T^*Q,\bS)\to S_+\wedge\cL Q^{\underline \bR^n-\rho^*W_{mas}}$, and presumably one can show that this is a genuine equivalence by using a filtration argument as in \Cref{sec:equivariantmorse}. More precisely, we expect it is possible to use physical Hamiltonians as in \cite{abbondandoloschwarz1} and filter the loop space by the Lagrangian action in loc.\ cit.\ to apply the filtration argument. 
\end{rk}

\begin{rk}
The map induced by \eqref{eq:viterboinnewsection} in homology matches the construction in \cite{zhaothesis}, which is also defined using half-cylinders; however, it is defined slightly differently from \cite{abouzaid2013symplectic}. One can show that these maps agree on homology. Hence Abouzaid's proof implies the isomorphism statement for the map in \cite{zhaothesis}. One could also adapt Abouzaid's map to our setup to define the Cieliebak--Latschev maps. At first glance, his version looks incompatible with the circle action; however, one could possibly tackle this problem using similar Borel constructions on the target. 
\end{rk}

\section{Gluing}\label{sec:gluing} 
The main goal of this section is to prove \Cref{proposition:s1-flow-cat}. In other words, we will produce $S^1$-equivariant coordinate systems on $\cM_{\tilde f,H,J}(i,j)$, and show that they assemble into manifolds with corners. Related to this, we will show the evaluation maps are smooth. The arguments also work \emph{mutatis mutandis} for the bimodules and $P$-valued modules (i.e. for statements such as \Cref{prop:eqlqvaluedmodule}), so we focus only on the flow categories themselves.


In general, Morse--Bott gluing requires the use of weighted Sobolev spaces. However, we can avoid this here thanks to the particular geometry of our situation. The basic strategy is as follows: if one considers the quotient $\cM_{\tilde f,H,J}/S^1$ of $\cM_{\tilde f,H,J}$ by the $S^1$-action (i.e.\ $\cM_{\tilde f,H,J}/S^1$ is the flow category obtained by quotienting both objects and morphism spaces by $S^1$), then the resulting flow category is non-degenerate. The techniques of Large \cite{largethesis}, which build on crucial work of Fukaya--Oh--Ota--Ono \cite{foooexponential}, allow us to produce smooth coordinate charts on $\cM_{\tilde f,H,J}/S^1$. To prove that $\cM_{\tilde f,H,J}$ is an $S^1$-equivariant Morse--Bott flow category, we only need to produce one more coordinate. This coordinate will be given by the asymptotic evaluation map to the domain of the trajectory, which is an $S^1$-torsor.

It is instructive to contemplate this strategy in the special case where $M$ is a point. Then $\cM_{\tilde f,H,J}/S^1$ is a flow category formed by Morse trajectories in $\mathbb{CP}=S/S^1$, for which smooth charts can be constructed e.g.\ as in \cite[Sec.\ 6]{largethesis}. Our plan is to equip the moduli spaces upstairs on $S$ with $S^1$-equivariant smooth charts, using the asymptotic evaluation map as our additional coordinate. In fact, the argument for a general symplectic manifold $M$ is almost the same as when $M$ is a point, so will only be discussed at the end of this section. 

Rather, let us begin by considering the following general situation: let $N$ carry a free $G$-action, and $f:N\to\bR$ be lifted from a Morse function on $N/G$. We choose a generic metric on $N/G$ that satisfies the Morse--Smale condition, and lift it to a $G$-equivariant metric on $N$. Such a lift exists thanks to: 
\begin{lem}
The following data are equivalent
\begin{itemize}
	\item a $G$-equivariant metric on $N$
	\item a metric on $N/G$, a connection on the principal bundle $p:N\to N/G$ and a $G$-equivariant metric on the vertical distribution
\end{itemize}
\end{lem}
This, together with the invariance of $f$, guarantees that the gradient flow is \emph{horizontal}. As a result, the data of a negative gradient flow line from the critical set $X=p^{-1}(x)$ to another $X'=p^{-1}(x')$ is the same as the data a negative gradient flow line from $x$ to $x'$ together with the asymptotic at $X$ (or $X'$). 

It will also be technically convenient to assume: 
\begin{assump}\label{assump:localflat}
The corresponding connection on $N\to N/G$ is flat near the critical points. 
\end{assump}

Assume $ind(x)=ind(x')+d+1$, and let $\cW(x,x')$, resp. $\cW(X,X')$ denote the moduli of parametrized negative gradient trajectories in $N/G$, resp. $N$ from $x$ to $x'$, resp.\ from $X$ to $X'$. By the remarks above, $\cW(X,X')\cong X\times \cW(x,x')\cong \cW(x,x')\times X'$, as $G$-equivariant spaces. We will now explain how to endow the compactified moduli spaces with the structure of a smooth manifold with corners, which can be organised into a $G$-equivariant flow category which we denote by $\cM_f$. 

As mentioned above, our strategy is roughly to use the smooth coordinates on $\cW(x,x')$ as constructed in \cite{largethesis}, and to use the $X$-component as the extra coordinate function. We will use the constructed smooth coordinates to show the compactification is a smooth manifold with corners. As part of \Cref{defn:flowcategory}, we have to show $\cW(X,X')\to X$ and $\cW(X,X')\to X'$ are smooth too. 

Note that we cut out the moduli space $\cW(X,X')$ in the Banach manifold of paths $v$ such that (i) $v$ is $W^{2,k}$-regular, (ii) $\dot v$ is horizontal with respect to the chosen connection. We restrict the target of the gradient operator $v\mapsto \dot v+grad_g(f)$ to the horizontal sections as well. In this way, we obtain a Fredholm operator, and the linear problem is essentially equivalent to the one on $N/G$, except it differs by an extra $Lie (G)$-component. 
\begin{lem}\label{lem:smteva}
The maps $\cW(X,X')\to X$ and $\cW(X,X')\to X'$ are smooth.
\end{lem}
\begin{proof}
This would presumably be by construction if we used weighted Sobolev norms. Instead, we appeal to \Cref{assump:localflat}. Let $\gamma\in \cW(X,X')$ and $\gamma_0=p\circ \gamma$. Let $U$ be a small convex neighborhood of $x$ so that the connection is flat over $U$. Pick an $s_0\in\bR$ such that $\gamma_0(s_0)\in U$. Hence $\gamma(s_0)\in p^{-1}(U)$ and there exists a small neighborhood $\cU\subset \cW(X,X')$ of $\gamma$ such that $v(s_0)\in p^{-1}(U)$ for every $v\in \cU$. Observe that $p^{-1}(U)\cong U\times X$ as smooth $G$-bundles with flat connections. As the gradient flow is horizontal (and hence every $v\in \cU$ as well), $ev_X(v)=\lim\limits_{s\to-\infty} v(s)\in X$ is the same as the composition of $v\mapsto v(s_0)\in p^{-1}(U)$ with the projection $p^{-1}(U)\cong U\times X\to X$. The smoothness of $ev_{X'}$ is similar. 
\end{proof} 
Given $\gamma\in\cW(X,X')$, there exists a unique map of $G$-torsors $X\to X'$ carrying $ev_X(\gamma)$ to $ev_{X'}(\gamma)$. Moreover, this map is the same for $g\gamma$, and therefore, only depends on the $\gamma_0=p\circ\gamma$. Indeed, this map is the same as the parallel transport along $\gamma_0$. We denote this map by $\rho_{XX'}(\gamma_0)$ or by $\rho_{XX'}(\gamma)$, and its inverse by $\rho_{X'X}(\gamma_0)$ or by $\rho_{X'X}(\gamma)$. As a result, we obtain a function $\rho_{XX'}:\cW(x,x')\to Hom_G(X,X')\cong G$ that factors through $\cM^\circ(x,x')$. It follows from \Cref{lem:smteva} that $\rho_{XX'}$ is smooth. Under the identification $\cW(X,X')\cong X\times \cW(x,x')$, $\rho_{X'X}$ satisfies $ev_{X'}(x,\gamma_0)= \rho_{XX'}(\gamma_0)x$.

To construct equivariant coordinates around a given negative gradient trajectory $\gamma:\bR\to N$, we shall first construct non-equivariant coordinates around $\gamma_0=\pi\circ \gamma:\bR\to N/G$. To this end, fix $s_0<\dots<s_d$. Recall from \cite[Sec.\ 6]{largethesis} that a \emph{complete system of hypersurfaces} for $\gamma_0$ is a sequence of local hypersurfaces $\underline H_i$, $i=0,\dots ,d$ such that $\gamma_0(s_i)\in \underline H_i$ and satisfying
\begin{enumerate}
	\item\label{item:secondcomplete} $ev_{s_0,\dots, s_d}:\cW(x,x')\to (N/G)^{d+1}$ is transverse to $\prod \underline H_i$ at $\gamma_0$
	\item\label{item:thirdcomplete} $\gamma_0$ is transverse to each $\underline H_i$ at $s_i$
\end{enumerate}
%
Using \eqref{item:thirdcomplete}, one constructs coordinate functions $v\mapsto s_i(v)$ near $\gamma_0\in\cW(x,x')$ satisfying (i) $s_i(\gamma_0)=s_i$ and (ii) $v(s_i(v))\in \underline H_i$. (i) and (ii) uniquely determine the functions near $\gamma_0$. By \eqref{item:secondcomplete}, the function $v\mapsto (s_0(v),\dots s_d(v))$ is a local diffeomorphism. Indeed, the preimage of $\prod\underline H_i$ in $\bR^{d+1}\times\cW(x,x')$ under the evaluation map is the graph of a local inverse. 

The functions $s_i(v)$ are translation invariant, giving a smooth coordinate chart $$(s_1(v)-s_0(v),\dots ,s_d(v)-s_{d-1}(v))$$ on $\cM^\circ(x,x'):=\cW(x,x')/\bR$ defined near $\gamma_0$. 
\begin{rk}
The open moduli spaces $\cW(x,x')$ and $\cW(X,X')$ carry smooth structures by the standard transversality arguments, and so far we merely constructed smooth charts (on the former). 
\end{rk}

We now to constructing equivariant coordinates around our negative gradient trajectory $\gamma:\bR\to N$. Let $H_i:=p^{-1}(\underline H_i)$. The $H_i$ are smooth hypersurfaces that contain the orbit of $\gamma(s_i)$, are local near the orbit and satisfy
\begin{enumerate}
	\item $ev_{s_0,\dots,s_d}:\cW(X,X')\to N^{d+1}$ is transverse to $\prod H_i$ at $\gamma$
	\item $\gamma$ is transverse to each $H_i$ at $s_i$
	\item $H_i$ is $G$-invariant
\end{enumerate}
One can also start with $H_i\subset N$ satisfying these conditions, and construct a complete system $\underline{ H}_i$. The functions $s_i$ are well-defined on $\cW(X,X')$ and satisfy (i) $s_i(\gamma)=s_i$ , (ii) $v(s_i(v))\in H_i$, (iii) $s_i(gv)=s_i(v)$. Note that the dimension of the transverse intersection $ev_{s_0,\dots,s_d}\pitchfork \prod H_i$ is now $dim(G)$, and $(s_0,\dots,s_d)$ no longer determines a coordinate system on $\cW(X,X')$. On the other hand, if we also use the evaluation map $ev_X$ into $X$, $(s_0,\dots,s_d,ev_X)$ is a local diffeomorphism. 
Hence $(s_1-s_0,\dots,s_d-s_{d-1},ev_X)$ is a smooth chart in the interior of $\cM^\circ(X,X'):=\cW(X,X')/\bR$ near $\gamma$. We denote the chart corresponding to $\mathbf{\underline H}:=(\underline H_i)_i$, resp. $\mathbf H$ by $U_{\mathbf{\underline H}}\subset \cM^\circ (x,x')/\bR$, resp. $U_{\mathbf{H}}\subset \cM^\circ (X,X')/\bR$. By construction, $U_{\mathbf{H}} \cong X\times U_{\mathbf{\underline H}}$ as smooth manifolds. 

We shall now explain how to extend the coordinate charts to the boundary, in order to show that the compactification of $\cM^\circ(X,X') :=\cW(X,X')/\bR$ is a manifold with corners. We first explain how to do this for the non-equivariant, Morse moduli spaces $\cM^\circ(x,x') :=(\cW(x,x')/\bR)$, following \cite{largethesis}. 

Let $x_0,\dots,x_n\in N/G$ be a sequence of critical points and let $\gamma^\ell_0\in \cW(x_{\ell-1},x_\ell)$. Assume we have chosen a complete system $\mathbf{\underline H}^\ell$ for each $\gamma^\ell_0$. Given $\mathbf{v}:= (v^\ell)_\ell\in \prod_{\ell=1}^n U_{\mathbf{\underline H}^\ell}$ and $\mathbf{T}= (T_1,\dots, T_{n-1})\in [T_0,\infty)^{n-1}$, where $T_0\gg 0$, define 
\begin{itemize}
	\item $L(v^\ell):=\frac{1}{2}(s_{d^\ell}(v^\ell)-s_0(v^\ell))$ (the half sum of all coordinates)
	\item $Z_1(\mathbf{v},\mathbf{T})=(-\infty,L(v^1)+T_1]$ and $Z_n(\mathbf{v},\mathbf{T})=[-L(v^n)-T_{n-1},\infty)$ 
	\item $Z_\ell(\mathbf{v},\mathbf{T})=[-L(v^\ell)-T_{\ell-1},L(v^\ell)+T_\ell]$, for $1<\ell<n$
	\item $W_\ell(T_\ell)=[-T_\ell,T_\ell]$
\end{itemize}
Identify $\bR$ with $Z=Z_1(\mathbf{v},\mathbf{T})\cup W_1(T_1)\cup Z_2(\mathbf{v},\mathbf{T})\cup \dots W_{n-1}(T_{n-1})\cup Z_n(\mathbf{v},\mathbf{T})$ and preglue the curves $v^\ell$ by defining a smooth map $v:\bR\to N/G$ that agrees with $v^\ell|_{Z_\ell(\mathbf{v},\mathbf{T})}$ on the region $Z_\ell(\mathbf{v},\mathbf{T})\subset Z$. More precisely, and as explained in \cite[Sec.\ 6]{largethesis}, for $T_0\gg 0$ there is a pregluing map
\begin{equation}
pre\cG:	 U_{\mathbf{\underline H}^1}\times \dots \times U_{\mathbf{\underline H}^n}\times (T_0,\infty)^{n-1}\to C^\infty (\bR, N/G)
\end{equation}
such that $v=pre\cG(\mathbf{v},\mathbf{T})$ satisfies $||\dot{v}+grad(f)||_{2,k}\leq C.e^{-\delta min \mathbf{T}}$ for some $C$ and $\delta$. Indeed, by construction $\dot{v}+grad(f)$ is supported on $\bigsqcup W_\ell(T_\ell)$ and the above inequality can be guaranteed using the exponential decay of each $v^\ell$.

Then Large \cite[Sec.\ 6]{largethesis} proceeds to use implicit function theorem to obtain a uniquely defined vector field $\xi=\xi_{\mathbf{v},\mathbf{T}}$ in $N/G$ along $Z$ such that $||\xi||_{2,k}\leq C.e^{-\delta min \mathbf{T}}$ and $\cG(\mathbf{v},\mathbf{T}):= exp_{pre\cG(\mathbf{v},\mathbf{T})}\xi  $ satisfies the gradient equation. Observe that the union $\mathbf{\underline H}$ of complete systems associated to each $\gamma^\ell_0 $  is a complete system for the corresponding preglued curves and that by assuming $\xi$ is tangent to each hypersurface and that each $H_i^\ell$ is totally geodesic, we can guarantee the same for $\cG(\mathbf{v},\mathbf{T})$. In summary, we obtain a smooth map
\begin{equation}
	\cG: U_{\mathbf{\underline H}^1}\times \dots \times U_{\mathbf{\underline H}^n}\times (T_0,\infty)^{n-1}\to  U_{\mathbf{\underline H}}.
\end{equation}
By making a coordinate change via the gluing profile $r=1/T$, and setting $r_i:=1/T_i$, $i=0,\dots, n-1$, we obtain a smooth map
\begin{equation}\label{eq:gluingbottom}
	\cG: U_{\mathbf{\underline H}^1}\times \dots \times U_{\mathbf{\underline H}^n}\times (0,r_0)^{n-1}\to  U_{\mathbf{\underline H}}.
\end{equation}
We write $\mathbf{r}:=(r_i)_i$ and denote the image of this map by $\cG(\mathbf{v},\mathbf{r})$. By the remarks above, the map respects the previously constructed $s_i$-coordinates on $U_{\mathbf{\underline H}^\ell}$, but $U_{\mathbf{\underline H}}$ has $n-1$ more coordinates. The map $\cG$ extends to a continuous map from $U_{\mathbf{\underline H}^1}\times \dots \times U_{\mathbf{\underline H}^n}\times [0,r_0)^{n-1}$ to the compactified moduli space, where a point with $r_i=0$ maps to a curve broken between $x_{i-1}$ and $x_i$. Following \cite[Sec.\ 6]{largethesis}, these map form a system of charts which cover the compactified moduli space and endow it with the structure of a smooth manifold with corners. The verification that the transition functions are smooth relies on a refinement of the aforementioned  exponential decay estimate; see \cite{largethesis,foooexponential} for details. 


We now turn our attention to our main task, namely of constructing $G$-equivariant charts on the moduli spaces of trajectories in $N$ itself. Let $X_i=p^{-1}(x_i)\subset N$ be the pre-image of the critical point $x_i\in N/G$ and $\mathbf{H}^\ell$ be the collection of hyperplanes given by the pre-images of the hyperplanes in $\mathbf{\underline H}^\ell$.
By similar considerations, we can construct a gluing map
\begin{equation}\label{eq:gluetop}
	\widetilde \cG: U_{\mathbf{ H}^1}\times_{X_1} \dots \times_{X_{n-1}} U_{\mathbf{ H}^n}\times [0,r_0)^{n-1}\to  \cM_f(X_0,X_n)
\end{equation}
where the latter denotes the compactified moduli space, consistently with our conventions. The construction is exactly the same, and the only thing to note is our assumption that the moduli spaces are cut out from spaces of horizontal paths. Moreover, the construction can be made equivariant. 

Recall the smooth identification $U_{\mathbf{ H}^\ell}\cong X_{\ell-1}\times U_{\mathbf{\underline H}^\ell}$. This implies
\begin{equation}
	U_{\mathbf{ H}^1}\times_{X_1} \dots \times_{X_{n-1}} U_{\mathbf{ H}^n}\cong X_0\times U_{\mathbf{\underline H}^1}\times_{X_1} X_1 \times U_{\mathbf{\underline H}^2}\times \dots \times_{X_{n-1}} X_{n-1}\times  U_{\mathbf{\underline H}^n}
\end{equation}
and the points of the latter are of the form 
\begin{equation}\label{eq:zibit}
	(x_0,v^1,\rho_{X_0X_1}(v^1)x_0,v^2,\rho_{X_1X_2}(v^2)\rho_{X_0X_1}(v^1)x_0,v^3,\dots )
\end{equation}
giving an explicit diffeomorphism 
\begin{equation}
	X_0\times U_{\mathbf{\underline H}^1}\times U_{\mathbf{\underline H}^2}\times \dots \times  U_{\mathbf{\underline H}^n}\cong X_0\times U_{\mathbf{\underline H}^1}\times_{X_1} X_1 \times U_{\mathbf{\underline H}^2}\times \dots \times_{X_{n-1}} X_{n-1}\times  U_{\mathbf{\underline H}^n}
\end{equation}
sending $(x_0,v^1,v^2,v^3,\dots )$ to \eqref{eq:zibit}. As a result, the restriction of \eqref{eq:gluetop} to $(0,r_0)^{n-1}$
\begin{equation}\label{eq:openchartmb1}
	X_0\times U_{\mathbf{\underline H}^1}\times U_{\mathbf{\underline H}^2}\times \dots \times  U_{\mathbf{\underline H}^n} \times (0,r_0)^{n-1} \to X_0\times U_{\mathbf{\underline H}}\cong U_{\mathbf{\underline H}}
\end{equation}
can be written as $\tilde \cG(x_0,\mathbf{v},\mathbf{r})=(x_0,\cG(\mathbf{v},\mathbf{r}))$. Moreover, the extension 
\begin{equation}\label{eq:chartsmb1}
	X_0\times U_{\mathbf{\underline H}^1}\times U_{\mathbf{\underline H}^2}\times \dots \times  U_{\mathbf{\underline H}^n} \times [0,r_0)^{n-1} \to X_0\times \cM_{f_0}(x_0,x_n)\cong \cM_f(X_0,X_n)
\end{equation}
satisfies the same formula. As a result, for different choices of complete systems, the transition maps are the same as those for $N/G$, with an added $id_{X_0}$ component. As a result, they are smooth and $\cM_f(X_0,X_n)$ is a manifold with corners. 

One can also construct charts of the form 
\begin{equation}
	 U_{\mathbf{\underline H}^1}\times \dots\times U_{\mathbf{\underline H}^{n-1}} \times  U_{\mathbf{\underline H}^n}\times X_n \times [0,r_0)^{n-1} \to  \cM_{f_0}(x_0,x_n)\times X_n\cong \cM_f(X_0,X_n)
\end{equation}
exactly in the same way; however, the compatibility with \eqref{eq:chartsmb1} is not clear. This is equivalent to the following statement, which is also one of the conditions in \Cref{defn:flowcategory}:
\begin{lem}
The evaluation map $\cM_f(X_0,X_n)\to X_n$ is smooth with respect to the smooth structure induced by \eqref{eq:chartsmb1}.
\end{lem}
\begin{proof}
In the image of \eqref{eq:openchartmb1}, the evaluation map satisfies 
\begin{equation}
	ev_{X_n}(\tilde \cG(x_0,\mathbf{v},\mathbf{r}))=\rho_{X_0X_n}(\cG(\mathbf{v},\mathbf{r}))x_0
\end{equation}
Therefore, it suffices to show that $\rho_{X_0X_n}\circ \cG$ extends to a smooth function 
\begin{equation}
U_{\mathbf{\underline H}^1}\times U_{\mathbf{\underline H}^2}\times \dots \times  U_{\mathbf{\underline H}^n} \times [0,r_0)^{n-1} \to Hom_G(X_0,X_n)
\end{equation}
whose restriction to $\mathbf{r}=0$ is given by $\mathbf{v}\mapsto\rho_{X_{n-1}X_n}(v^n)\circ\dots  \circ \rho_{X_0X_1}(v^1)$. For a given broken curve $\mathbf{v}$, denote $\rho_{X_{n-1}X_n}(v^n)\circ\dots  \circ \rho_{X_0X_1}(v^1)$ also by $\rho_{X_0X_n}(\mathbf{v})$. By construction,  $\rho_{X_0X_n}(pre\cG(\mathbf{v},\mathbf{r}))=\rho_{X_0X_n}(\mathbf{v})$. Recall that $\cG(\mathbf{v},\mathbf{r})=exp_{pre\cG(\mathbf{v},\mathbf{r})}(\xi_{\mathbf{v},\mathbf{r}})$. Therefore, to prove the claim, it suffices to show that the rate of change in $\rho_{X_0X_n}$ is controlled by $\xi_{\mathbf{v},\mathbf{r}}$.

One can trivialize the principal bundle $N\to N/G$ in an open neighborhood $\cU$ of the broken trajectory. Moreover, we can choose the trivialization such that the restriction of the connection to small neighborhoods of the critical points is the standard flat connection. For simplicity, assume that for each chosen complete system $\mathbf{\underline H}^\ell$, the first and the last hyperplane is in the standard flat neighborhood. For simplicity, also assume $G$ is a torus and let $\mathfrak{g}$ denote its Lie algebra. In this case, a connection on the principal bundle $N|_\cU\to \cU$ is the same as a $\mathfrak g$-valued $1$-form $\theta$ on $\cU$. For a given $v\in T_\cU$, its horizontal lift is given by $(v,\theta(v))$. Our assumption that the connection is the standard flat one near the critical points implies that $\theta$ vanishes near the critical points. Therefore, given small $\mathbf{r}$, the support of $\theta$ along a curve of the form $exp_{pre\cG(\mathbf{v},\mathbf{r})}(\xi_{\mathbf{v},\mathbf{r}})$ is contained in $Z_\ell(\mathbf{v},\mathbf{r})$. 

For a given path $a(s)$ in $\cU$, the parallel transport is given by $\int \theta(\dot a(s))ds\in G$ (this is normally an element of $\mathfrak{g}$, but we project to the torus $G$ via the exponential map $\mathfrak{g}\to G$). In summary,
\begin{equation}
\rho_{X_0X_n} \cG(\mathbf{v},\mathbf{r})=\int_{-\infty}^{\infty}\theta\bigg(\frac{d}{ds} \cG(\mathbf{v},\mathbf{r})\bigg)ds=\sum_\ell\int_{Z_\ell(\mathbf{v},\mathbf{r})}\theta\bigg(\frac{d}{ds} \cG(\mathbf{v},\mathbf{r})\bigg)ds
\end{equation}
and thus we have to check the smoothness of each
\begin{equation}\label{eq:zumuuth}
	(\mathbf{v},\mathbf{r})\mapsto \int_{Z_\ell(\mathbf{v},\mathbf{r})}\theta\bigg(\frac{d}{ds} \cG(\mathbf{v},\mathbf{r})\bigg)ds=\int_{Z_\ell(\mathbf{v},\mathbf{r})}\theta\bigg(\frac{d}{ds} exp_{pre\cG(\mathbf{v},\mathbf{r})}(\xi_{\mathbf{v},\mathbf{r}})\bigg)ds
\end{equation}
For this it suffices to show that $\nabla_\mathbf{u}\partial^p_\mathbf{r}\int_{Z_\ell(\mathbf{v},\mathbf{r})}\theta\bigg(\frac{d}{ds} exp_{pre\cG(\mathbf{v},\mathbf{r})}(\xi_{\mathbf{v},\mathbf{r}})\bigg)ds$ is decaying exponentially. Notice, 
\begin{equation}\nabla_\mathbf{u}\partial^p_\mathbf{r}	\int_{Z_\ell(\mathbf{v},\mathbf{r})}\theta\bigg(\frac{d}{ds} exp_{pre\cG(\mathbf{v},\mathbf{r})}(\xi_{\mathbf{v},\mathbf{r}})\bigg)ds =	\int_{Z_\ell(\mathbf{v},\mathbf{r})} \nabla_\mathbf{u}\partial^p_\mathbf{r}\theta\bigg(\frac{d}{ds} exp_{pre\cG(\mathbf{v},\mathbf{r})}(\xi_{\mathbf{v},\mathbf{r}})\bigg)ds
\end{equation}
The boundaries of the integral is changing in $\mathbf{r}$ and there would normally be additional terms corresponding to the derivatives of the boundaries of the integral. However, $\theta=0$ near the boundaries of the integral, for $\mathbf{r}\ll 1$, and these terms vanish. 
$exp_{pre\cG(\mathbf{v},\mathbf{r})}(\xi_{\mathbf{v},\mathbf{r}})$ is the composition of the path $s\mapsto (pre\cG(\mathbf{v},\mathbf{r})(s), \xi_{\mathbf{v},\mathbf{r}}(s))$ in $T(N/G)$ with the function $exp:T(N/G)\to N/G$. As a result $\frac{d}{ds} exp_{pre\cG(\mathbf{v},\mathbf{r})}(\xi_{\mathbf{v},\mathbf{r}})$ is given by the composition of $dexp$, the differential of $exp:T(N/G)\to N/G$ with $s\mapsto (\frac{d}{ds}pre\cG(\mathbf{v},\mathbf{r})(s), \nabla_s\xi_{\mathbf{v},\mathbf{r}}(s))$. In summary, $\nabla_\mathbf{u}\partial^p_\mathbf{r}\theta\bigg(\frac{d}{ds} exp_{pre\cG(\mathbf{v},\mathbf{r})}(\xi_{\mathbf{v},\mathbf{r}})\bigg)$ is bounded above by a constant multiple of sum of terms of the form $\bigg|\nabla_\mathbf{u'}\partial^{p'}_\mathbf{r}\xi_{\mathbf{v},\mathbf{r}}\bigg|$ and $\bigg|\nabla_\mathbf{u'}\partial^{p'}_\mathbf{r}\frac{d}{ds}\xi_{\mathbf{v},\mathbf{r}}\bigg|$, where $p'$ runs on the terms of order less than $|p|$ and $\mathbf{u'}$ on the terms of order less than $|\mathbf{u}|$.  The constants depend on the derivatives of $\theta$, $exp$ and $pre\cG(\mathbf{v},\mathbf{r})$, and they can be chosen to be uniformly bounded over $Z_\ell(\mathbf{v},\mathbf{r})$. As a result, 
\begin{align}
\Bigg|	\int_{Z_\ell(\mathbf{v},\mathbf{r})} \nabla_\mathbf{u}\partial^p_\mathbf{r}\theta\bigg(\frac{d}{ds} exp_{pre\cG(\mathbf{v},\mathbf{r})}(\xi_{\mathbf{v},\mathbf{r}})\bigg)ds\Bigg|\leq  \\
\int_{Z_\ell(\mathbf{v},\mathbf{r})} \bigg|\nabla_\mathbf{u}\partial^p_\mathbf{r}\theta\bigg(\frac{d}{ds} exp_{pre\cG(\mathbf{v},\mathbf{r})}(\xi_{\mathbf{v},\mathbf{r}})\bigg)\bigg|ds\preceq \\
\sum_{\mathbf{u'},p'} \int_{Z_\ell(\mathbf{v},\mathbf{r})} \bigg|\nabla_\mathbf{u'}\partial^{p'}_\mathbf{r}\xi_{\mathbf{v},\mathbf{r}}\bigg| + \bigg|\nabla_\mathbf{u'}\partial^{p'}_\mathbf{r}\frac{d}{ds} \xi_{\mathbf{v},\mathbf{r}}\bigg|  ds \\
 \leq 
\sum_{\mathbf{u'},p'} ||\nabla_\mathbf{u'}\partial^{p'}_\mathbf{r}\xi_{\mathbf{v},\mathbf{r}}|_{Z_\ell(\mathbf{v},\mathbf{r})}||_{2,1}(2L(\gamma^\ell)+2T_\ell)
\end{align}
The last line follows from H\"older's inequality (here, $1<\ell<n$, $\ell=1$ and $\ell=n$ require a similar but special treatment). 
The inequality $||\xi||_{2,k}\leq C.e^{-\delta min \mathbf{T}}$ can be refined to 
\begin{equation}
||\nabla_\mathbf{u}\partial^{p}_\mathbf{r}\xi|_{Z_\ell(\mathbf{v},\mathbf{r})}||_{2,k-|\mathbf{u}|-|p|}\leq C.e^{-\delta min \mathbf{T}}
\end{equation}
See \cite[Proposition 6.10]{largethesis}. Therefore, for $k\gg 0$, 
\begin{equation}
	\Bigg|	\int_{Z_\ell(\mathbf{v},\mathbf{r})} \nabla_\mathbf{u}\partial^p_\mathbf{r}\theta\bigg(\frac{d}{ds} exp_{pre\cG(\mathbf{v},\mathbf{r})}(\xi_{\mathbf{v},\mathbf{r}})\bigg)ds\Bigg|\leq C.e^{-\delta min \mathbf{T}}T_\ell\leq C.e^{-\delta /max (\mathbf{r})}/r_\ell
\end{equation}
where the constants may be different. This suffices to prove the claim that \eqref{eq:zumuuth} is smooth. 
\end{proof}
This completes the proof that $\cM_f$ is a Morse--Bott flow category. 
\begin{rk}
    The moduli spaces $\cM_f(X,X')$ and $X\times \cM_{f_0}(x,x')$ are smoothly identified, and one can describe the second evaluation maps as well as the compositions in terms of the functions $\rho_{XX'}:\cM_{f_0}(x,x')\to Hom_G(X,X')$. Conversely, if we are given a Morse flow category $\cM_{f_0}$, and a smooth functor from $\cM_{f_0}$ to the category of $G$-torsors, we can construct a $G$-equivariant Morse--Bott flow category whose space of objects have connected components given by $G$-torsors. 
\end{rk}

We now explain how to adapt the above argument to other flow categories which occur in the paper. We first prove the following:
\begin{prop}\label{prop:Geqflowcategory}
$\cM_{\tilde f,h_s}^G$ is a $G$-equivariant Morse--Bott flow category (see \Cref{sec:equivariantmorse}).
\end{prop}


We only explain the construction of the smooth coordinate charts, as the rest is the argument is the same. Recall that $\tilde f$ is an equivariant Morse--Bott function on $EG$ that lifts a Morse function $f$ on $BG$, and $h_s$, $s\in EG$ is a family of functions on the closed manifold $M$, that is Morse over the critical points of $\tilde f$ and that depend on $s$ in a $G$-equivariant way. For simplicity, we will pretend that $EG$ is a finite dimensional closed manifold. (Strictly speaking, the following construction should be conducted over $E_nG$, and then by induction on $n$. We wish to avoid the notational cumbersomeness that would arise from one more parameter.)

Let $x,x'\in BG$ be two critical points of $f$, and let $\tilde x\in X:=p^{-1}(x)$, $\tilde x'\in X'=p^{-1}(x')$. Let $y\in crit(h_{\tilde x})$ and $y'\in crit(h_{\tilde x'})$; thus $(\tilde x,y), (\tilde x',y')\in ob(\cM_{\tilde f,h_s})$. For simplicity, assume $x\neq x'$, and consider a smooth gradient trajectory $(\gamma,\eta)$ from $(\tilde x,y)$ to $(\tilde x',y')$. Let $ind(x)-ind(x')=d_1+1$, $ind(y)-ind(y')=d_2$ and choose a complete system $s_0<\dots<s_{d_1}$, $\underline{H}_0,\dots ,\underline H_{d_1}\subset BG$ associated to $\gamma_0=p\circ\gamma$. As before, let $H_i:=p^{-1}(\underline H_i)$. As we have seen, such a choice gives rise to a smooth equivariant coordinate chart $U_{\mathbf{H}}\cong X\times U_{\mathbf{\underline H}}\subset \cM_{\tilde f}(X,X')$.

Now choose a (non-equivariant) complete system $s_1'<\dots <s_{d_2}'$, $H_1',\dots,H_{d_2}'\subset M$ associated to the path $\eta$. More precisely, we require that $(\{s_i\},\{s_j'\},H_i,H_j')$ is a complete system for $(\gamma,\eta)$. Given $\eta'=g\eta$ in the orbit of $\eta$, we obtain a complete system for $\eta'$ with the same $s_i'$ and with the hyperplanes $H_i'(g\eta):=gH_i'$. In particular, in an equivariant neighborhood $\cW$ of the orbit of $(\gamma,\eta)$, we have coordinate functions $(v,w)\mapsto s_i(v)$, we have an evaluation map $\cW\to X$, $(v,w)\mapsto ev_X(v)$ and finally we have functions $s_j'(v,w)$ satisfying $s_j'(\gamma, \eta)=s_j'$ and 
\begin{equation}
	w(s_j'(v,w))\in H_j'(g\eta)=gH_j'
\end{equation}
for the unique $g\in G$ such that $ev_X(v)=gev_X(\gamma)$. Hence, we obtain a smooth equivariant coordinate chart $	U_{\mathbf{ H},\mathbf{ H'}}\subset\cW((x,y),(x'y'))/\bR$, $	U_{\mathbf{ H},\mathbf{ H'}} \to X\times \bR^{d_1+d_2}$ given by 
\begin{equation}
(v,w)\mapsto (ev_X(v),s_1(v)-s_0(0),\dots ,s_{d_1}(v)-s_{d_1-1}(v),s_1'(v,w)-s_0(v),s_2'(v,w)-s_0(v),\dots )
\end{equation}
One can apply the arguments above to define gluing maps, and smooth equivariant charts on the compactification. The smoothness of the second evaluation map follows from \Cref{lem:smteva}, as this map depends only on the $v$ component.

Finally, we explain how to modify the above arguments for $\cM_{\tilde f,H,J}$. 
\begin{proof}[Proof of \Cref{proposition:s1-flow-cat}]
As in the above discussion, let $S=ES^1=S^\infty$. Let $x,x'\in \bC\bP^\infty$ be two critical points of $f$, which we assume to be different for simplicity, and let $\tilde x\in X:=p^{-1}(x)$, $\tilde x'\in X'=p^{-1}(x')$.  Let $y\in orb(H_{\tilde x})$ and $y'\in orb(H_{\tilde x'})$; thus, $(\tilde x,y), (\tilde x',y')\in ob(\cM_{\tilde f,H_s})$. Consider an unbroken pair $(\gamma,u)$ from $(\tilde x,y)$ to $(\tilde x',y')$. Let $ind(x)-ind(x')=d_1+1$, $ind(y)-ind(y')=d_2$ and choose a complete system $s_0<\dots<s_{d_1}$, $\underline{H}_0,\dots ,\underline H_{d_1}\subset S$ associated to $\gamma_0=p\circ\gamma$.

Now, we choose a sequence of points $z_j'=(s_j',t_j')$ and hypersurfaces $H_1',\dots,H_{d_2}'\subset M$ on the cylinder $S^1\times\bR$ such that $z_j'$ is regular for $u$ (see \cite{largethesis} for the definition), that satisfies $s_1'<\dots <s_{d_2}'$. We further assume this is a \emph{complete system for $u$}. More precisely
\begin{enumerate}
	\item\label{item:secondcompletehol} $ev_{\{s_i\},\{z_j'\}}:\cW((\tilde x,y),(\tilde x',y'))\to S^{d_1+1}\times M^{d_2}$ is transverse to $\prod H_i\times \prod  H_j'$ at $(\gamma,u)$
	\item\label{item:thirdcompletehol} $u|_{\bR\times\{t_j'\}}$ is transverse to each $H'_j$ at $s'_j$ and this is the unique intersection point.
\end{enumerate}
Analogously, this produces functions $s_i(v)$, $s_j'(v,w)$ in an equivariant neighborhood of the orbit of $(\gamma,u)$ such that
\begin{itemize}
	\item $s_i(\gamma)=s_i$ and $v(s_i(v))\in H_i$
	\item $s_j'(\gamma,u)=s_j'$ and $w(s_j'(v,w),gt_j')\in H_j'$ where $g\in S^1$ is the unique element satisfying $ev_X(v)=gev_X(\gamma)$
\end{itemize}
Similar to above, we obtain a smooth equivariant coordinate chart $U_{\mathbf{ H},\mathbf{ H'}}\subset\cW((x,y),(x'y'))/\bR$, $U_{\mathbf{ H},\mathbf{ H'}} \to X\times \bR^{d_1+d_2}$ given by 
\begin{equation}
	(v,w)\mapsto (ev_X(v),s_1(v)-s_0(0),\dots ,s_{d_1}(v)-s_{d_1-1}(v),s_1'(v,w)-s_0(v),s_2'(v,w)-s_0(v),\dots )
\end{equation}
The above arguments similarly produce us gluing maps, and smooth equivariant charts on the compactification. The smoothness of the second evaluation map follows from \Cref{lem:smteva}, as this map depends only on the $v$ component.
\end{proof}

\appendix 

\section{$\langle k \rangle$-manifolds and equivariant neat embeddings}\label{appendix:k-manifolds}

The purpose of this appendix is to review the standard notions of $\langle k \rangle$-manifolds, neat embeddings, etc. We also briefly explain how to generalize these notions to the equivariant setting. We refer to \cite{hajek2014manifolds} for foundational material on the category of manifolds with corners (charts, diffeomorphisms, etc.).  

\subsection{$\langle k \rangle$-manifolds}
 
%


For a given manifold with corners, the codimension $c(x)$ of a point $x$ is defined to be the highest codimension of boundary faces containing $x$. For instance, the interior points have codimension $0$, the points in the interior of the boundary has codimension $1$, etc. 
\begin{defn}[{\cite[Def.\ 3]{cohenrevisited},\cite{laures}}]\label{defn:kmfd} 
A $\langle k \rangle$-manifold is a manifold with corners $M$ along with an ordered $k$-tuple $(\partial_1 M, \dots, \partial_k M)$ of disjoint unions of codimension $1$ faces satisfying the following conditions:
\begin{itemize}
\item[(i)] $\partial_1 M \cup \dots \cup \partial_k M = \partial M$
\item[(ii)] for all $ i \neq j$ we have that $\partial_i M \cap \partial_j M$ is a face of both $\partial_i M$ and $\partial_j M$
\item[(iii)] every $x\in M$ is contained exactly in $c(x)$ connected boundary faces
\end{itemize}
\end{defn}

\begin{defn}
Let $M$ and $N$ be $\langle k\rangle$-manifolds. A \emph{neat embedding} $M\to N$ is a smooth embedding such that each $\partial_i M$ maps to $\partial_i N$ and the pre-image of $\partial_i N$ is given by $\partial_i M$.
\end{defn}

\begin{exmp}\label{example:neat-embedding}
The canonical example of a $\langle k \rangle$-manifold is $\mathbb{R}_+^k \times \mathbb{R}^N$, for any $N \geq 0$, where 
\begin{equation}
    \partial_i (\mathbb{R}_+^k \times \mathbb{R}^N):= \{ (x_1, x_k; y_1,\dots, y_N) \mid x_i=0 \}. 
\end{equation}
\end{exmp}

\Cref{example:neat-embedding} is universal in the following sense. 
\begin{prop}\cite{laures}\label{proposition:laures}
    Any $\langle k \rangle$-manifold admits a neat embedding (in the category of $\langle k \rangle$-manifolds) into some $\mathbb{R}_+^k \times \mathbb{R}^N$. Any two such neat embeddings are isotopic for $N$ large enough.
\end{prop}

We record the following fact: 
\begin{lem}[Thm.\ 4 in \cite{hajek2014manifolds}]\label{lemma:fiber-j-module}
For $i=1,2$, let $M_i$ be a $\langle k_i \rangle$-manifold and let $Y$ be a closed manifold. If $M_1 \to Y \leftarrow M_2$ are transverse, then $M_1 \times_Y M_2 \hookrightarrow M_1 \times M_2$ is a neat embedding of $\langle k_1+k_2 \rangle$-manifolds.
\end{lem} 
In particular, since neat embeddings of $\langle k \rangle$ manifolds are natural under products and compositions, the data of neat embeddings $M_i \hookrightarrow \bR_+^{k_i} \times \bR^{N_1}$ induces a neat embedding of $\langle k_1+ k_2\rangle$-manifold $M_1 \times_Y M_2 \hookrightarrow \bR_+^{k_1+k_2} \times \bR^{N_1+N_2}$. 

\subsection{$G$-equivariant $\langle k \rangle$-manifolds and equivariant neat embeddings}
Fix a compact Lie group $G$. 

\begin{defn} 
A \emph{$G$-equivariant $\langle k \rangle$-manifold} is a $\langle k \rangle$-manifold $M$ equipped with a smooth $G$-action such that the action preserves $\partial_i M$ (this is automatic if $G$ is connected). A \emph{neat embedding} of $G$-equivariant $\langle k \rangle$-manifolds is a neat embedding of $\langle k \rangle$-manifolds that is also $G$-equivariant. 
\end{defn}
\begin{exmp}\label{example:equiv-universal}
Let $V$ be a real representation of $G$. Then $\mathbb{R}_+^k \times V$ is a $G$-equivariant $\langle k \rangle$-manifold. 
\end{exmp}

Recall the Mostow--Palais theorem: one can embed every $G$-equivariant smooth closed manifold $G$-equivariantly into some $V$. The proof of \Cref{proposition:laures} in \cite{laures} can be adapted to show the existence of equivariant neat embeddings into $\mathbb{R}_+^k \times V$. 

An important ingredient in the proof of the standard neat embedding theorem in \cite{laures} is a generalization of the collars of smooth manifolds with boundary to $\langle k\rangle$-manifolds. See \cite[Lemma 2.1.6]{laures} for the precise notion. One can adapt this to the $G$-equivariant setting. For simplicity, we explain this only for $\langle 1\rangle$-manifolds, i.e. for smooth manifolds with boundary. For a $G$-equivariant smooth manifold with boundary $M$, a collar $\partial M \times\bR_+\hookrightarrow M$ is determined by a vector field $v$ in a neighborhood of $\partial M$. One can fix a bi-invariant metric on $G$ and replace $v$ by the vector field $v^{avg}$ defined by $v^{avg}_p:=\frac{1}{vol(G)} \int_{g \in G} {g}_* (v_{g^{-1}p})$. This vector field determines a $G$-equivariant collar of $\partial M$ over some $\partial M \times[0,\epsilon)$.

\begin{prop}\label{proposition:equivariant-neat}
Any $G$-equivariant $\langle k \rangle$-manifold admits an equivariant neat embedding into some $\bR_+^k \times V$, where $V$ is an orthogonal representation of $G$. Any two such embeddings become isotopic after stabilization (by another representation of $G$). 
\end{prop}

\begin{proof}
The proof is essentially the same as \cite[Prop.\ 2.1.7]{laures}. The only differences are (i) one uses $G$-equivariant collars instead of standard collars, (ii) one appeals to the Mostow--Palais theorem wherever the Whitney embedding theorem is applied. 
	

For the uniqueness of $G$-equivariant neat embeddings, the proof of \cite[Lem.\ 2.2.3]{laures} works $G$-equivariantly without change.
\end{proof}

Finally, the above discussion readily generalizes to flow categories. 

\begin{defn}\label{definition:neat-embedding-flow-cat}
    A ($G$-equivariant) neat embedding of a $G$-equivariant flow category  $\cM$ (\Cref{subsection:extension-equivariant}) is the data of:
    \begin{itemize}
        \item integers $q_i$ and orthogonal $G$-actions on each $\bR^{q_i-q_j}$ such that the action on $\bR^{q_i-q_k}=\bR^{q_i-q_j}\times \bR^{q_j-q_k}$ coincides with the product action (hence, we have $G$-actions on each $J_{q_i-q_j}(i,j)$ compatible with the product structure)
        \item $G$-equivariant neat embeddings $\cM(i,j)\hookrightarrow J_{q_i-q_j}(i,j)$ such that the following diagram commutes
\begin{equation}\label{eq:neatflowembgeq}
	\xymatrix{\cM(i,j)\times_j\cM(j,k) \ar[r]\ar@{^{(}->}[d]& \cM(i,k)\ar@{^{(}->}[d] \\ J_{q_i-q_j}(i,j) \times J_{q_j-q_k}(j,k)\ar[r]& J_{q_i-q_k}(i,k) }
\end{equation}
    \end{itemize}
\end{defn}

\begin{cor}\label{collary:neat-embedding-category}
    Any (Morse--Bott) flow category $\cM$ with a $G$-action admits a $G$-equivariant neat embedding, unique up to stabilization.
\end{cor}
\begin{proof}
This follows by applying the argument in \Cref{proposition:equivariant-neat} inductively. 

\end{proof}

\section{Equivariant spectra} \label{section:equivariant-homotopy}
There are multiple non-equivalent categories whose objects deserve to be called $G$-equivariant spectra. We briefly summarize the notions which are relevant in this paper.

Let $G$ be a compact Lie group and let $\operatorname{Sp}$ be the infinity category of spectra. The infinity category of \emph{spectra with a $G$-action} (also called \emph{Borel $G$-spectra}) is the functor category $Fun(BG, \operatorname{Sp})$. Concretely, a Borel $G$-spectrum is can be viewed as an ordinary spectrum along with maps $\phi_g: E \to E$ for all $g \in G$, homotopies $\phi_h \circ \phi_g \simeq \phi_{gh}$, etc. In other words, they can considered as spectra with a homotopy $G$-action. 
As natural isomorphisms can be detected pointwise \cite[\href{https://kerodon.net/tag/01DK}{tag01DK}]{kerodon}, a morphism $E \to E'$ of Borel $G$-spectra is a weak equivalence in $Fun(BG, \operatorname{Sp})$ iff it is a weak equivalence in $\operatorname{Sp}$ after forgetting the $G$-action. 

We are also interested in the infinity category $G\operatorname{Sp}$ of \emph{genuine $G$-spectra}. Let $\operatorname{Sp}^O$ be the $1$-category of orthogonal spectra. Then one can consider the \emph{$1$-category} $G\operatorname{Sp}^O:= Fun(BG, \operatorname{Sp}^O)$ of orthogonal $G$-spectra. Concretely, an orthogonal $G$-spectrum $E$ can be viewed as a system of $G$-actions on each level $E_n$ which are compatible with the structure maps \cite[Def.\ 2.1]{Schwede}. The infinity category $G\operatorname{Sp}$ is obtained from $G\operatorname{Sp}^O$ by inverting equivalences. We emphasize that to be an equivalence in $G\operatorname{Sp}^O$, it is not enough to merely induce a non-equivariant equivalence; rather, equivalences in $G\operatorname{Sp}^O$ induce non-equivariant equivalences on \emph{geometric fixed points} for all closed subgroups of $G$. In particular, there are forgetful functors 
\begin{equation}
	G \operatorname{Sp} \to \operatorname{Sp}^{BG} \to \operatorname{Sp}.
\end{equation}

\begin{exmp}
	The $G$-equivariant map $EG \to \{*\}$ induces an equivalence $\Sigma^\infty EG_+\to\bS$ of Borel $G$-spectra, but $\Sigma^\infty EG_+$ is not equivalent to $\bS$ as a genuine $G$-spectrum. The geometric fixed points of the former are given by $0$, whereas it is $\bS$ for the latter.
\end{exmp}

Given a (non-equivariant) spectrum $E$, a \emph{Borel $G$-action} on $E$ shall mean a lift of $E$ to $\operatorname{Sp}^{BG}$. A \emph{genuine $G$-action} means a lift to $G \operatorname{Sp}$.

\begin{comment}

\section{Applications to symplectic topology} \textcolor{blue}{I suggest starting a new paper here}

\subsection{Floer theory for exact Lagrangians}

\begin{prop}
Let $L$ be an exact Lagrangian in a Liouville manifold. Then $\tilde{HF}(L) \simeq \Sigma^\infty L$. 
\end{prop}

More generally: if $K, L$ are cleanly intersecting exact Lagrangians in a (say) Liouville manifold $M$, Pozniak constructed a spectral sequence computing $HF(K, L)$. We have the following  generalization:
\begin{prop}
[state here?]
\end{prop}

\subsection{Spectral symplectic cohomology of smooth divisor complements}
\red{or more generally when the boundary is M-B, make index restrictions, speculate about normal crossing divisors of depth two}

\subsection{Cotangent bundle and spectral Viterbo isomorphism}
\red{we had this idea somewhere in Slack. Use quadratic Hamiltonians, it can be checked here that it gives the same thing. Indeed, we may simply use the Hamiltonian $r^2$, where $r$ is the norm of a cotangent vector wrt a metric. Try to match the flow category with the energy functional on he free loop space}

\red{it may be better to deal w/ M-B theory on the loop space earlier}

\section{Circle actions and cyclotomic structures on $\widetilde{SH}(M)$}

\section{Fixed point computations}

\red{brief overview such as ``Our tool facilitates the computations of fixed points''}

\subsection{Geometric fixed points}

\begin{thm}
Let $M$ be a Liouville manifold and let $\tilde{SH}(M)$ be the spectral symplectic cohomology spectrum as defined in the previous section. Then $\Phi^{S^1} \widetilde{SH}(X) \simeq \Sigma^\infty X$. 
\end{thm}

\subsection{Homotopy and Tate fixed points of the spectrum $\widetilde{SH}(M)$}

\subsection{Fixed points of $\widetilde{SH}(M)$ after applying extraordinary cohomology theories}
\red{bunch of exmaples, ideally $H\bZ$, $H\bF_p$, maybe some others. M--B setting simplifies this too}

\section{Applications to topology}
\subsection{Homotopy type via Morse--Bott theory and Cohen--Jones--Segal theorem}

\subsection{Equivariant homotopy type}
\red{may emphasize plenty of equivariant M--B function, but not Morse. Also make explicit computations, derive some localization theorems for Hamiltonian $G$-manifolds}

Let $M$ be a smooth manifold with an $S^1$-action. We explain how to construct via Morse theory a spectrum with a $S^1$-action which is isomorphic to $\Sigma_+^\infty M$ (as a spectrum with a $S^1$-action). 

Define $\tilde{q}: S^{2N+1} \times M \to \R$ to have the following properties:
\begin{itemize}
\item $\tilde{q}$ is $S^1$-invariant
\item if $x$ is in the local slice, then $\tilde{q}(t \cdot x, t \cdot w)= q_t(w)$
\item critical points of $\tilde{f}+ \tilde{q}$ lie over the orbits of $\tilde{f}$. 
\end{itemize}

We also choose a family of metrics $h$ on $M$ parametrized by $S^{2N+1}$ which is globally $S^1$ invariant and independent of $z$ in each slice. 
Let $\tilde{h}$ be the product metric on $S^{2N+1} \times M$ induced by the round metric on $S^{2N+1}$ and $h$ on the second factor

\begin{lem}\label{lemma:morse-case-flow-cat}
We can choose $h$ so that the pair $(\tilde{q}, \tilde{h})$ satisfies the Morse--Bott--Smale transversailty condition; the moduli spaces are $\langle k \rangle$-manifolds with an $S^1$-action;  we have an $S^1$-flow category.
\end{lem}

\begin{defn}
Let $\tilde{H}_{S^{2N+1}}(M)$ the geometric realization of the flow category of \Cref{lemma:morse-case-flow-cat}. This is a spectrum with an $S^1$-action.
\end{defn}

\subsubsection{Invariance}

\begin{lem}\label{lemma:morse-increase-N}
Increasing $N$ gives rize to an inclusion of flow categories.
\end{lem}

\begin{defn}
Let $\tilde{H}_{S}(M):= \lim_{N \to \infty} \tilde{H}_{S^{2N+1}}(M)$.  This is a spectrum with an $S^1$-action.
\end{defn}


\begin{lem}
Given a family $(\tilde{q}_t^s, \tilde{g}_t^s)$, (i) there is a bimodule... ii) the induced map on realization is compatible with the induced maps in \Cref{lemma:morse-increase-N}. 
\end{lem}

\begin{lem}
The constant interpolation induces the identity on realizations.
\end{lem}

Remaining questions:
(i) show independence of model of BG?
(ii) show that this spectrum really computes the correct thing.

\subsection{Free loop spaces}
\red{the energy functional on the free loop space is naturally M--B, we may be able to recover the homotopy type in this case too. Involves a colimit. Still pure topology, but can be used later to match $\widetilde{SH}$}

\red{more symplectic applications have to come later}

\section{Plan/Outline (remove this section)}

A rough plan
\begin{enumerate}
	\item introduction
	\item some standard background
	\item Morse--Bott flow categories (cite Zhengyi) (first do non-equivariantly, then explain how to generalize, may include equivariant neat embeddings in an appendix)
	\begin{enumerate}
		\item definition
		\item geometric realization (first explain relative collapse)
		\item maps induced by bi-modules, homotopies etc.
		\item easy applications (e.g. C-J-S theorem, G-equivariant version)
	\end{enumerate}
	\item Spectral symplectic cohomology
	\begin{enumerate}
		\item construction (include transversality)
		\item equivalence to the standard version (non-autonomus version)
		\item may discuss relations to existing autonomous constructions
		\item briefly the quadratic Hamiltonian version, and equivalence [for equivalence, it may be practical to realize ]
		\item some examples
		\begin{enumerate}
			\item Divisor complement/circle bundle
			\item cotangent bundle (use quadratic Hamiltonians)
		\end{enumerate}
	\end{enumerate}
    \item Genuine circle action and the cyclotomic structures
    \begin{enumerate}
    	\item construction and the comparison for different autonomus Hamiltonians (the latter such take a few sentences)
    	\item geometric fixed points, spectral Jingyu 1
    \end{enumerate}
    \item Some homotopy/Tate fixed point computations (of the spectrum itself, as well as its smashes with various cohomology theories)
    \item Non-autonomus circle action and the cyclotomic structure (this may take a whole other paper, so can be left out, or we can skip the comparison part)
    \begin{enumerate}
    	\item construct the circle action and comparison with the autonomus case
    	\item cyclotomic structure and the comparison
    \end{enumerate}
    \item relations to existing operators (?)
    \item further applications (?)
\end{enumerate}

Consider including
\begin{itemize}
	\item other constructions, such as the p-typical cyclotomic structure
	\item relations to existing operators
\end{itemize}

\bibliographystyle{alpha}
\bibliography{biblio}		
\end{document}